\documentclass{amsart}

\pdfoutput=1

\usepackage[margin=1.35in]{geometry}
\usepackage{amsmath,amssymb}
\usepackage[colorlinks=true,linkcolor=blue,citecolor=blue,urlcolor=blue]{hyperref}
\usepackage{graphicx}
\usepackage{subcaption}
\usepackage{float}

\newtheorem{theorem}{Theorem}[section]
\newtheorem{prop}[theorem]{Proposition}
\newtheorem{lem}[theorem]{Lemma}
\newtheorem{cor}[theorem]{Corollary}
\newtheorem{statement}[theorem]{Statement}

\theoremstyle{definition}
\newtheorem{defn}[theorem]{Definition}
\newtheorem{ass}[theorem]{Assumption}

\theoremstyle{remark}
\newtheorem{ex}[theorem]{Example}
\newtheorem{remark}[theorem]{Remark}

\numberwithin{equation}{section}

\newcommand{\be}{\begin{equation}}
\newcommand{\ee}{\end{equation}}

\newcommand{\lb}{\left(}
\newcommand{\rb}{\right)}
\newcommand{\lsb}{\left[}
\newcommand{\rsb}{\right]}
\newcommand{\lcb}{\left\{}
\newcommand{\rcb}{\right\}}

\newcommand{\ip}[1]{\langle#1\rangle}
\newcommand{\norm}[1]{\lVert#1\rVert}

\newcommand{\ve}{\varepsilon}

\newcommand{\x}{X}
\newcommand{\y}{Y}
\newcommand{\z}{Z}

\newcommand{\F}{\mathcal{F}}
\newcommand{\N}{\mathbb{N}}
\renewcommand{\P}{\mathbb{P}}
\newcommand{\Q}{\mathbb{Q}}
\newcommand{\R}{\mathbb{R}}
\renewcommand{\S}{\mathcal{S}}
\newcommand{\U}{\mathcal{N}}
\newcommand{\V}{\mathcal{V}}
\newcommand{\W}{\mathcal{W}}
\newcommand{\X}{X}
\newcommand{\Y}{Y}
\newcommand{\Z}{Z}

\newcommand{\allN}{\mathcal{I}}

\newcommand{\lip}{\kappa}

\newcommand{\spaan}{\text{span}}
\newcommand{\conv}{{\text{cone}}}

\newcommand{\cts}{\mathcal{C}}
\newcommand{\dlr}{\mathcal{D}_{\ell,r}}
\newcommand{\dr}{\mathcal{D}_r}

\newcommand{\sm}{\Gamma}
\newcommand{\esm}{\bar{\Gamma}}
\newcommand{\dm}{\Lambda}

\newcommand{\proj}{\mathcal{L}}

\begin{document}

\title[Directional derivatives of Skorokhod maps]{On directional derivatives of Skorokhod maps in convex polyhedral domains}

\date{\today}

\author[David Lipshutz]{David Lipshutz\textsuperscript{$\ast$}}
\address{Division of Applied Mathematics \\ 
                Brown University \\ 
                182 George Street, Providence \\ 
                RI 02912}
\email{david{\_}lipshutz@brown.edu}

\author[Kavita Ramanan]{Kavita Ramanan\textsuperscript{$\dagger$}}
\email{kavita{\_}ramanan@brown.edu}



\thanks{\textsuperscript{$\ast$}The research of the first author was supported in part by NSF grant DMS-1148284 and AFOSR grant FA9550-12-1-0399.}
\thanks{\textsuperscript{$\dagger$}The research of the second author was supported in part by NSF grant CMMI-1234100 and NSF grant DMS-1407504.}

\dedicatory{Brown University}

\begin{abstract}
The study of both sensitivity analysis and differentiability of the stochastic flow of a reflected process in a convex polyhedral domain is challenging because the dynamics are discontinuous at the boundary of the domain and the boundary of the domain is not smooth. These difficulties can be addressed by studying directional derivatives of an associated extended Skorokhod map, which is a deterministic mapping that takes an unconstrained path to a suitably reflected version. In this work we develop an axiomatic framework for the analysis of directional derivatives of a large class of Lipschitz continuous extended Skorokhod maps in convex polyhedral domains with oblique directions of reflection.  We establish existence of directional derivatives at a path whose reflected version satisfies a certain boundary jitter property, and also show that the right-continuous regularization of such a directional derivative can be characterized as the unique solution to a Skorokhod-type problem, where both the domain and directions of reflection vary (discontinuously) with time. A key ingredient in the proof is establishing certain contraction properties for a family of (oblique) derivative projection operators. As an application, we establish pathwise differentiability of reflected Brownian motion in the nonnegative quadrant with respect to the initial condition, drift vector, dispersion matrix and directions of reflection. The results of this paper are also used in subsequent work to establish pathwise differentiability of a much larger class of reflected diffusions in convex polyhedral domains.
\end{abstract}

\maketitle

\tableofcontents

\section{Introduction}\label{sec:intro}

\subsection{Overview} 

Reflected stochastic processes that are constrained to lie in the closure of a convex polyhedral domain in $J$-dimensional Euclidean space $\R^J$ arise in many contexts, including as diffusion approximations of stochastic networks \cite{Chen1991a,Ramanan2008a,Reiman1984}, in the study of interacting diffusions and limits of interacting particle systems \cite{Burdzy2002,Warren2007}, in chemical and biochemical reaction networks \cite{Kinnally2010} and in mathematical finance \cite{Banner2005}. The analysis of processes with state constraints is challenging due to the fact that the dynamics are often discontinuous on the boundary of the domain, and is further complicated when the boundary is not smooth. In many cases, the extended Skorokhod problem (ESP), which is a generalization of the Skorokhod problem (SP), provides a convenient tool for the pathwise analysis of such processes \cite{Costantini1992,Dupuis1991,Harrison1981,Lions1984,Ramanan2006,Saisho1987,Tanaka1979}. Specifically, the ESP provides an axiomatic framework to constrain a path taking values in $\R^J$ to the closure of a domain in $\R^J$ with a constraining or regulating function that ``pushes'' in prescribed directions on the boundary of the domain, referred to as the directions of reflection, to ensure that the path remains within the closure of the domain. The constrained path can often be represented as the image under the associated extended Skorokhod map (ESM), a generalization of the Skorokhod map (SM), of the unconstrained path, which is typically analytically more tractable. Thus, the analysis of the constrained path is then largely reduced to the study of properties of the ESM.  

The study of the differentiability of the stochastic flow associated with a stochastic process or, more broadly speaking, sensitivity of the stochastic process to perturbations in the initial
condition and other parameters that define the process, is a classical topic in stochastic analysis. For example, there is a substantial body of work that studies these questions for (unconstrained) diffusions in $\R^J$, with contributions from Elworthy \cite{Elworthy1978}, Bismut \cite{Bismut1981}, Ikeda and Watanabe \cite{Ikeda1981}, Kunita \cite{Kunita1981}, Metivier \cite{Metivier1982} and others. The book by Kunita \cite{Kunita1997} contains a summary of many of these results. In contrast, there are relatively few results for reflected processes or even reflected Brownian motions (RBMs), especially in the context of oblique reflection and nonsmooth domains which is relevant in applications (see, e.g., \cite{Chen2014,Dieker2014}). Two exceptions include the work of Andres \cite{Andres2009}, which establishes differentiability of the stochastic flow of a reflected diffusion with identity covariance in a convex polyhedral domain, but only until the first time the reflected diffusion hits a nonsmooth part of the boundary, thus avoiding having to deal with the effect of the nonsmooth part of the boundary; and the work of Dieker and Gao \cite{Dieker2014}, which looks at sensitivities of reflected diffusions in the nonnegative orthant (with reflection matrices that satisfy a so-called Harrison-Reiman condition) with respect to perturbations of the drift in the direction $-{\bf 1}$, the vector in $\R^J$ with negative one in each component. In addition to these works, Deuschel and Zambotti \cite{Deuschel2005} considered differentiability of stochastic flows for a system of one-dimensional reflected diffusions with coupled drifts; Pilipenko (see \cite{Pilipenko2013} and references therein) studied differentiability properties of stochastic flows of reflected diffusions with state-dependent covariance in the half space with normal reflection; and Burdzy \cite{Burdzy2009a} and Andres \cite{Andres2011} characterized derivatives of stochastic flows for RBMs and reflected diffusions with identity covariance, respectively, in domains with smooth boundaries and normal reflection, where geometric challenges arise due to the curvature of the boundary.

In this work we show that the study of both the differentiability of the stochastic flow or sensitivity analysis of a reflected stochastic process can largely be reduced to the study of an associated unconstrained process and so-called directional derivatives of the SM or ESM (see Definition \ref{def:derivative} below). Directional derivatives of the SM were first introduced in the one-dimensional setting by Mandelbaum and Massey \cite{Mandelbaum1995} to analyze diffusion approximations of time-inhomogeneous queues (see also \cite[Chapter 9]{Whitt2002} for a generalization of their results). Directional derivatives of a class of so-called Harrison-Reiman multidimensional SMs on the nonnegative orthant were shown to exist and characterized by Mandelbaum and Ramanan \cite{Mandelbaum2010}. While their work was primarily motivated by the study of time-inhomogeneous queues, the results in \cite{Mandelbaum2010} have subsequently been used by Cudina and Ramanan \cite{Cudina2011} to study asymptotically optimal controls for fluid limits of time-inhomogeneous queueing networks; by Chen \cite{Chen2014} to develop an algorithm for unbiased estimators of sensitivities of a stochastic fluid network; by Dieker and Gao \cite{Dieker2014} to study certain sensitivities of reflected diffusions; by Lipshutz and Williams \cite{Lipshutz2015} to study differentiability of flows of a one-dimensional delay differential equation with reflection; and by Honnappa, Jain and Ward \cite{Honnappa2015a,Honnappa2015} to study the behavior of transient queueing networks. The proof in \cite{Mandelbaum2010}, which considered directional derivatives of multidimensional SMs, relies on the Harrison-Reiman SMs having a fairly explicit representation and satisfying certain monotonicity properties, which do not hold for more general classes of multidimensional SMs or ESMs. In particular, such a multidimensional SM can be viewed as a system of coupled one-dimensional SMs, and when combined with the explicit expression for the one-dimensional SM, the analysis of directional derivatives of the multidimensional SM largely reduces to a careful study of directional derivatives of the supremum functional. Even for SMs that lie in the slightly larger class of so-called generalized Harrison-Reiman maps (see Example \ref{ex:gHR} below), which do not have the same monotonicity property, the approach in \cite{Mandelbaum2010} does not yield existence of the directional derivative.

The main goal of this work is to establish existence and provide a useful characterization of directional derivatives of ESMs associated with a broad class of convex polyhedral domains and oblique directions of reflection that need not satisfy the above mentioned monotonicity properties nor admit an explicit representation, and to demonstrate the use of directional derivatives of an ESM in the study of differentiability properties of reflected stochastic processes. In contrast to \cite{Mandelbaum2010}, we develop a completely different approach to studying directional derivatives of an ESM. Much in the spirit of the ESP, we adopt an axiomatic approach and formulate a time-inhomogeneous Skorokhod-type problem, which we refer to as the derivative problem (DP, see Definition \ref{def:dp} below), in which the domain and directions of reflection vary (discontinuously) with time. We show that under general conditions on the domain and directions of reflection that ensure the ESM is well defined and Lipschitz continuous, directional derivatives of the ESM exist and are uniquely characterized via the DP when the ESM is evaluated at a continuous path whose image under the ESM satisfies a certain boundary jitter property (see Definition \ref{def:jitter} below as well as the discussion following Theorem \ref{thm:main} for a description of its role in establishing existence of directional derivatives). While the boundary jitter property is not necessary to prove existence of directional derivatives (see, e.g., \cite{Mandelbaum2010}), by imposing this property we are able to provide a nice axiomatic characterization of directional derivatives via the DP and consideration of such paths is sufficient for many stochastic applications. As an illustration, we show that the boundary jitter property is satisfied by a large class of RBMs in the nonnegative quadrant, and use this to establish existence of and characterize so-called pathwise derivatives of such RBMs with respect to their initial conditions, drift and dispersion coefficients, and directions of reflection. To the best of our knowledge, sensitivities of a reflected diffusion with respect to its dispersion coefficient or directions of reflection have not been considered in any prior work. The boundary jitter property can be shown to hold for a larger class of reflected diffusions in more general convex polyhedral domains; however, the verification of this property is more involved and is thus deferred to a forthcoming paper, which also establishes pathwise differentiability of reflected diffusions in these domains. The pathwise nature of our analysis allows for our results to be applied in principle to a larger class of continuous reflected processes beyond reflected diffusions; specifically, those processes that satisfy the boundary jitter property. 

In summary, the main contributions of this paper are as follows:
\begin{itemize}
	\item Definition of the boundary jitter property and verification of the boundary jitter property for a large class of RBMs in the nonnegative quadrant (Section \ref{sec:jittermain} and Section \ref{sec:rbmjitter}).
	\item Formulation and analysis of the DP (Section \ref{sec:dpmain} and Section \ref{sec:dp}).
	\item Existence of directional derivatives of the ESM evaluated at paths satisfying the boundary jitter property, and their characterization via the DP (Section \ref{sec:ESMderivative} and Sections \ref{sec:nablasm1}--\ref{sec:tau}).
	\item Differentiability of the ESM with respect to the directions of reflection (Section \ref{sec:directions}).
	\item Existence and characterization of pathwise derivatives of a large class of RBMs in the nonnegative quadrant with respect to their initial conditions, drift and dispersion coefficients, and directions of reflection (Section \ref{sec:pathwise}).
\end{itemize}
Additionally, the results of this paper are used in forthcoming work to establish pathwise differentiability of a large class of reflected diffusions in convex polyhedral domains.
	
\subsection{Outline of the paper}

The paper is organized as follows. In Section \ref{sec:esppoly}, a precise statement of the ESP with a convex polyhedral domain is given and its associated ESM is introduced. Assumptions guaranteeing that the ESM is well defined are stated and the notion of a directional derivative of the ESM is introduced. Our main results on directional derivatives of an ESM and pathwise differentiability of an RBM in the nonnegative orthant are presented in Section \ref{sec:main}. In Section \ref{sec:jitter} some important consequences of the boundary jitter property are shown and a verification that a large class of RBMs in the nonnegative quadrant satisfy the boundary jitter property is given. A discussion of the DP and its properties is given in Section \ref{sec:dp}. The proof of our main result on existence of directional derivatives and their characterization via the DP is given in Sections \ref{sec:nablasm1}--\ref{sec:tau}. Proofs of some useful lemmas are relegated to Appendices \ref{apdx:local}--\ref{apdx:projlip}, and some relevant examples are provided in Appendix \ref{apdx:examples}.

\subsection{Notation}\label{sec:notation}

We now collect some notation that will be used throughout this work. We use $\N=\{1,2,\dots\}$ to denote the set of positive integers, and let $\N_0\doteq\N\cup\{0\}$ and $\N_\infty\doteq\N\cup\{\infty\}$. Let $\Q$ denote the set of rational numbers. For $J\in\N$, let $\R^J$ denote $J$-dimensional Euclidean space and $\R_+^J$ the closed nonnegative orthant in $\R^J$. When $J=1$, we suppress $J$ and simply write $\R$ for $(-\infty,\infty)$ and $\R_+$ for $[0,\infty)$. Given $r,s\in\R$, we let $r\vee s\doteq\max(r,s)$ and $r\wedge s\doteq\min(r,s)$. For a column vector $x\in\R^J$, let $x^j$ denote the $j$th component of $x$. We let $\{e_1,\dots,e_J\}$ denote the standard orthonormal basis for $\R^J$; that is, $e_i^j$ is 1 if $i=j$ and 0 otherwise. For $x,y\in\R^J$, we write $\ip{x,y}\doteq x^1y^1+\cdots+x^Jy^J$ for the usual inner product on $\R^J$. Given $x\in\R^J$, $|x|\doteq\ip{x,x}^{\frac{1}{2}}$ denotes the Euclidean norm of $x$. We use $\mathbb{S}^{J-1}\doteq\{x\in\R^J:|x|=1\}$ to denote the unit sphere in $\R^J$. 

For $J,K\in\N$, let $\R^{J\times K}$ denote the set of real-valued matrices with $J$ rows and $K$ columns. For a matrix $M\in\R^{J\times K}$, let $M_k^j$ denote the entry of the $j$th row and $k$th column, $M^j$ denote the $j$th row and $M_k$ denote the $k$th column. We write $M'$ to denote the transpose of $M$. Given a nondegenerate square matrix $M\in\R^{J\times J}$, we write $M^{-1}$ to denote matrix inverse of $M$.

Given sets $A$ and $B$, we write $A\subseteq B$ when $A$ is a subset of $B$ and we write $A\subsetneq B$ when we want to emphasize that $A$ is a strict subset of $B$. For a subset $A\subseteq\R$, we let $\inf A$ and $\sup A$ denote the infimum and supremum, respectively, of $A$. We use the convention that the infimum and supremum of the empty set are respectively defined to be $\infty$ and $-\infty$. Given a subset $A\subseteq\R^J$, let $A^\circ$, $\overline{A}$ and $\partial A$ denote the interior, closure and boundary, respectively, of $A$. We let $\conv(A)$ denote the convex cone generated by $A$; that is,
	$$\conv(A)\doteq\lcb\sum_{k=1}^Kr_kx_k:K\in\N,x_k\in A,r_k\geq0\rcb.$$
with the convention that $\conv(\emptyset)\doteq\{0\}$. We let $\spaan(A)$ denote the set of all possible finite linear combinations of vectors in $A$; that is,
	$$\spaan(A)\doteq\lcb\sum_{k=1}^Kr_kx_k:K\in\N,x_k\in A,r_k\in\R\rcb,$$
with the convention that $\spaan(\emptyset)\doteq\{0\}$. Given a subset $A\subseteq\R^J$, we let $A^\perp$ denote the orthogonal complement of $\spaan(A)$ in $\R^J$; that is,
	$$A^\perp\doteq\{x\in\R^J:\ip{x,y}=0\;\forall\;y\in A\}.$$

Given $T\in(0,\infty]$ and a closed, convex subset $E\subseteq\R^J$, we let $\dlr([0,T):E)$ denote the set of functions on $[0,T)$ taking values in $E$ that have finite left limits at all $t\in(0,T)$, finite right limits at all $t\in[0,T)$, and are left continuous and/or right continuous at each $t\in(0,\infty)$. We let $\dr([0,T):E)$ denote the set of right continuous functions with finite left limits in $\dlr([0,T):E)$ and let $\cts([0,T):E)$ denote the further subset of continuous functions in $\dr([0,T):E)$. Given a subset $A\subseteq E$, we use $\cts_A([0,T):E)$ to denote the subset of continuous functions $f\in\cts([0,T):E)$ with $f(0)\in A$. When $T=\infty$, $E=\R^J$ and $A\subseteq\R^J$, we simply write $\dlr$, $\dr$, $\cts$ and $\cts_A$ for $\dlr([0,\infty):\R^J)$, $\dr([0,\infty):\R^J)$, $\cts([0,\infty):\R^J)$ and $\cts_A([0,\infty):\R^J)$, respectively. We endow $\dr([0,T):E)$ and its subsets with the topology of uniform convergence on compact intervals in $[0,T)$. For $f\in\dlr([0,T):E)$ and $t\in[0,T)$, define the supremum norm of $f$ over $[0,t]$ by
	$$\norm{f}_t\doteq\sup_{s\in[0,t]}|f(s)|<\infty.$$
Let $|f|(t)\in[0,\infty]$ denote the total variation of $f$ over the interval $[0,t]$; that is,
	$$|f|(t)\doteq\sup_{\{t_k\}}\sum_{k=1}^m|f(t_k)-f(t_{k-1})|,$$
where the supremum is over all finite partitions $\{0=t_0<t_1<\cdots<t_m=t\}$ of the interval $[0,t]$. We let $f(t-)\doteq\lim_{s\uparrow t}f(s)$ for all $t\in(0,T)$ and $f(t+)\doteq\lim_{s\downarrow t}f(s)$ for all $t\in[0,T)$. For $f\in\dlr([0,T):E)$, we say $f$ is increasing (resp. nondecreasing, decreasing, nonincreasing) if $f(s)<f(t)$ (resp. $f(s)\leq f(t)$, $f(s)>f(t)$, $f(s)\geq f(t)$) for all $0\leq s<t<T$. We call the function $g\in\dr([0,T):E)$ defined by $g(t)\doteq f(t+)$ for all $t\in[0,T)$ the \emph{right continuous regularization of $f$}.
	
We abbreviate ``such that'' as ``s.t." and ``almost surely'' as ``a.s.''

\section{The extended Skorokhod reflection problem}\label{sec:esppoly}

In this section we introduce the ESP and directional derivatives of the associated ESM. In Section \ref{sec:espdef} we describe the class of convex polyhedral domains that we consider and give a precise definition of a solution to the ESP. In Section \ref{sec:esmlip} we provide sufficient conditions for the associated ESM to satisfy a Lipschitz continuity condition. In Section \ref{sec:esmwelldefined} we present further conditions under which the ESM is well defined on all of $\cts$. In Section \ref{sec:derivativesdef} we define a directional derivative of the ESM.

\subsection{Statement of the extended Skorokhod reflection problem}\label{sec:espdef}

Let $G$ be the closure of a nonempty convex polyhedral domain in $\R^J$, which can be expressed as the intersection of a finite number of closed half spaces; that is,
	\be\label{eq:G}G\doteq\bigcap_{i=1,\dots,N}\lcb x\in\R^J:\ip{x,n_i}\geq c_i\rcb,\ee
for some positive integer $N\in\N$, unit vectors $n_i\in \mathbb{S}^{J-1}$ and constants $c_i\in\R$, for $i=1,\dots,N$. To each face $F_i\doteq\{ x\in\partial G:\ip{x,n_i}=c_i\}$ of the polyhedron is associated a reflection vector $d_i\in\R^J$ that points into the interior $G^\circ$ of $G$; that is, $\ip{d_i,n_i}>0$. Without loss of generality, the reflection vectors are assumed to be normalized so that $\ip{d_i,n_i}=1$ for $i=1,\dots,N$. For notational convenience, we let $\allN\doteq\{1,\dots,N\}$ and for $x\in G$, we write 
	\be\label{eq:allNx}\allN(x)\doteq\{i\in\allN:x\in F_i\}\ee
to denote the (possibly empty) set of indices associated with the faces that intersect at $x$. For $x\in G$, we let $|\allN(x)|$ denote the cardinality of the set $\allN(x)$. In the following lemma we state an upper semicontinuity property of the set-valued function $\allN(\cdot)$ on $G$.

\begin{lem}\label{lem:allNusc}
For each $x\in G$, there is an open neighborhood $U_x$ of $x$ in $\R^J$ such that
	\be\label{eq:allNusc}\allN(y)\subseteq\allN(x)\qquad\text{for all }y\in U_x\cap G.\ee
\end{lem}

\begin{proof}
See, for example, Lemma 2.1 of \cite{Kang2007}.
\end{proof}

For $x\in\partial G$, we let $d(x)$ denote the cone generated by the permissible directions of reflection at $x$. In other words, for $x\in\partial G$,
	\be\label{eq:dx}d(x)\doteq\conv(\{d_i,i\in\allN(x)\}).\ee
For convenience, we extend the definition of $d(x)$ to all of $G$ by setting $d(x)\doteq\{0\}$ for all $x\in G^\circ$.

We now give a precise formulation of the ESP for continuous paths.

\begin{defn}\label{def:esp}
Suppose that $\{(d_i,n_i,c_i),i\in\allN\}$ and $\x\in\cts$ are given. Then $(\z,\y)\in\cts\times\cts$ solves the ESP $\{(d_i,n_i,c_i),i\in\allN\}$ for $\x$ if $\y(0)\in d(\z(0))$ and if for all $t\in[0,\infty)$, the following conditions hold:
\begin{itemize}
	\item[1.] $\z(t)=\x(t)+\y(t)$;
	\item[2.] $\z(t)\in G$;
	\item[3.] for all $s\in[0,t)$,
		\be\label{eq:ytys} \y(t)-\y(s)\in\conv\lsb\cup_{u\in(s,t]}d(\z(u))\rsb.\ee
\end{itemize}
If there exists a unique solution $(\z,\y)$ to the ESP for $\x$, then we write $\z=\esm(\x)$ and we refer to $\esm$ as the extended Skorokhod map (ESM).
\end{defn}

\begin{remark}
Given $T\in(0,\infty)$ and $\x\in\cts([0,T):\R^J)$, we say that $(\z,\y)\in\cts([0,T):\R^J)\times\cts([0,T):\R^J)$ solves the ESP for $\x$ on $[0,T)$ if $\y(0)\in d(\z(0))$ and conditions 1--3 of the ESP hold for all $t\in[0,T)$.
\end{remark}

\begin{remark}
The inclusion \eqref{eq:ytys} holds if and only if the following inclusion holds:
	\be\label{eq:ytminusys}\y(t)-\y(s)\in\conv\lsb\cup_{u\in(s,t)}d(\z(u))\rsb.\ee
Here the ``only if'' direction is immediate and the ``if'' direction follows from the continuity of $\y$ and because \eqref{eq:ytminusys} implies that
	$$\y(t-)-\y(s)\in\cup_{r<t}\conv\lsb\cup_{u\in(s,r)}d(\z(u))\rsb\subseteq\conv\lsb\cup_{u\in(s,t)}d(\z(u))\rsb.$$
\end{remark}

\begin{remark}
The formulation of the ESP in Definition \ref{def:esp} appears slightly different from the one originally given in \cite[Definition 1.2]{Ramanan2006} since the ESP in \cite{Ramanan2006} was formulated for paths $\x\in\dr$ that satisfy $\x(0)\in G$. In particular, \cite[Definition 1.2]{Ramanan2006} requires that $\y(0)=0$ and $\y(t)-\y(t-)\in\conv\lsb d(\z(t))\rsb$ for all $t\in(0,\infty)$. Here, we only consider continuous paths, so the jump condition holds automatically. In addition, we allow input paths $\x$ that start outside $G$; that is, $\x(0)\not\in G$, but instead allow $\y(0)\neq0$ as long as $\y(0)\in d(\z(0))$. This mild generalization is useful when considering directional derivatives of the ESM, where if $\x(0)\in\partial G$, an $\ve$-perturbation of $\x$ in the direction $\psi\in\cts$ may result in the perturbed initial condition $\x(0)+\ve\psi(0)$ lying outside of $G$. When $\x\in\cts_G$, the conditions of Definition \ref{def:esp} ensure that any solution $(\z,\y)$ of the ESP for $\x$ must satisfy $\y(0)=0$, so Definition \ref{def:esp} coincides with \cite[Definition 1.2]{Ramanan2006}.
\end{remark}

\begin{remark}\label{rmk:sp}
The ESP and the associated ESM are generalizations of the SP and its associated SM. In contrast to the SP, the ESP does not require that the constraining term $\y$ have finite variation on compact intervals. In \cite{Ramanan2006} it was shown that the set $\V$, defined by
	\be\label{eq:setV}\V\doteq\{x\in\partial G:\text{ there exists $d\in \mathbb{S}^{J-1}$ such that $\{d,-d\}\subseteq d(x)$}\}\ee
is important for characterizing whether the constraining term $\y$ can have unbounded variation. In particular, if $\V=\emptyset$, then $(\z,\y)$ solves the ESP for $\x$ if and only if $(\z,\y)$ solves the SP for $\x$. When referring to specific examples of ESPs in which $\V=\emptyset$, we will use SP and SM in place of ESP and ESM, respectively, to emphasize that the constraining term must be of finite variation on compact intervals.
\end{remark}

We close this section with a useful time-shift property of the ESP. Given a solution $(\z,\y)$ of the ESP for $\x\in\cts$ and $S\in[0,\infty)$, define $\x^S,\y^S,\z^S\in\cts$ by
\begin{align}\label{eq:xS}
	\x^S(\cdot)&\doteq\z(S)+\x(S+\cdot)-\x(S),\\\label{eq:zS}
	\z^S(\cdot)&\doteq \z(S+\cdot),\\\label{eq:yS}
	\y^S(\cdot)&\doteq \y(S+\cdot)-\y(S).
\end{align}

\begin{lem}\label{lem:esmshift}
Suppose $(\z,\y)$ solves the ESP for $\x\in\cts$. Let $S\in[0,\infty)$ and define $\x^S,\z^S,\y^S$ as in \eqref{eq:xS}--\eqref{eq:yS}. Then $(\z^S,\y^S)$ solves the ESP for $\x^S$. Moreover, if $(\z,\y)$ is the unique solution to the ESP for $\x$, then for any $0\leq S<T<\infty$, $\z(T)$ depends only on $\z(S)$ and $\{\x(S+t)-\x(S),t\in[0,T-S]\}$.
\end{lem}

\begin{proof}
In \cite[Lemma 2.3]{Ramanan2006} this result was shown in the case $\x(0)\in G$. The same argument can be applied when $\x(0)\not\in G$.
\end{proof}

\subsection{Lipschitz continuity}\label{sec:esmlip}

In this section we provide sufficient conditions on the ESP for the associated ESM to be Lipschitz continuous on its domain of definition. The conditions, stated in Assumption \ref{ass:setB} below, are expressed in terms of the existence of a convex set $B$ whose inward normals satisfy certain geometric properties expressed in terms of the data $\{(d_i,n_i,c_i),i\in\allN\}$. Given a convex set $B$ and $z\in\partial B$, we let $\nu_B(z)$ denote the set of unit inward normals to the set at the point $z$. In other words,
	$$\nu_B(z)\doteq\lcb\nu\in \mathbb{S}^{J-1}:\ip{\nu,y-z}\geq0\text{ for all }y\in B\rcb.$$

\begin{ass}\label{ass:setB}
There exists $\delta>0$ and a compact, convex, symmetric set $B$ with $0\in B^\circ$ such that for $i\in\allN$,
	\be\label{eq:setB}\lcb\begin{array}{l}z\in\partial B\\|\ip{z,n_i}|<\delta\end{array}\rcb\qquad\Rightarrow\qquad\ip{\nu,d_i}=0\qquad\text{for all }\;\nu\in\nu_B(z).\ee
\end{ass}

\begin{remark}\label{rmk:setB}
Suppose $\delta>0$ and $B$ are such that \eqref{eq:setB} holds for $i\in\allN$. Then given any $c>0$, \eqref{eq:setB} holds with $c\delta$ and $cB\doteq\{cz:z\in B\}$ in place of $\delta$ and $B$, respectively. In particular, if $z\in\partial(cB)$ for some $c>0$ and $|\ip{z,n_i}|=0$, then $\ip{\nu,d_i}=0$ for all $\nu\in\nu_{cB}(z)$.
\end{remark}

This assumption was first introduced as \cite[Assumption 2.1]{Dupuis1991} and was shown in \cite[Theorem 2.2]{Dupuis1991} to imply Lipschitz continuity of the associated SM on its domain of definition. In \cite[Theorem 3.3]{Ramanan2006}, it was shown that Assumption \ref{ass:setB} is a sufficient condition for Lipschitz continuity of the ESM as well. An analogue of Assumption \ref{ass:setB} also serves as a sufficient condition for Lipschitz continuity of the map associated with the so-called constrained discontinuous media problem  (see \cite[Theorem 2.9]{Atar2008}). A dual condition on the data $\{(d_i,n_i,c_i),i\in\allN\}$ that implies the existence of a set $B$ that satisfies Assumption \ref{ass:setB} was introduced in \cite{Dupuis1999a,Dupuis1999}. As demonstrated in \cite{Dupuis1998,Dupuis2000}, the dual condition is often more convenient to use in practice. 

We now give a precise statement of the Lipschitz continuity property that follows from Assumption \ref{ass:setB}.

\begin{theorem}\label{thm:esmlip}
Given an ESP $\{(d_i,n_i,c_i),i\in\allN\}$, suppose Assumption \ref{ass:setB} holds. Then there exists $\lip_{\esm}<\infty$ such that if $(\z_1,\y_1)$ solves the ESP for $\x_1\in\cts$ and $(\z_2,\y_2)$ solves the ESP for $\x_2\in\cts$, then for all $T\in[0,\infty)$,
\begin{align}\label{eq:Zlip}
	\norm{\z_1-\z_2}_T&\leq\lip_{\esm}\norm{\x_1-\x_2}_T.
\end{align}
\end{theorem}

\begin{proof}
By \cite[Theorem 3.3]{Ramanan2006}, there exists $\tilde{\lip}<\infty$ such that whenever $\x_1(0),\x_2(0)\in G$, \eqref{eq:Zlip} holds with $\tilde{\lip}$ in place of $\lip_{\esm}$ for all $T\in[0,\infty)$. Now suppose $\x_1,\x_2\in\cts$ are arbitrary. For $i=1,2$, define $\x_i^0,\z_i^0$ as in \eqref{eq:xS}--\eqref{eq:zS}, with $S=0$ and $\x_i,\z_i$ in place of $\x,\z$, respectively, so that $\z_i^0(0)=\x_i^0(0)\in G$. By \eqref{eq:zS}, the time-shift property of the ESP (Lemma \ref{lem:esmshift}), \cite[Theorem 3.3]{Ramanan2006}, \eqref{eq:xS} and condition 1 of the ESP, for all $T\in[0,\infty)$,
\begin{align*}
	\norm{\z_1-\z_2}_T=\norm{\z_1^0-\z_2^0}_T&\leq\tilde{\lip}\norm{\x_1^0-\x_2^0}_T\leq\tilde{\lip}\norm{\x_1-\x_2}_T+\tilde{\lip}|\y_1(0)-\y_2(0)|.
\end{align*}
Therefore, it suffices to show there exists $\hat{\lip}<\infty$ (depending only on the data $\{(d_i,n_i,c_i),i\in\allN\}$) such that $|\y_1(0)-\y_2(0)|\leq\hat{\lip}|\x_1(0)-\x_2(0)|$. Then \eqref{eq:Zlip} will hold with $\lip_{\esm}=\tilde{\lip}(1+\hat{\lip})$. The existence of $\hat{\lip}$ can be shown using an argument that is related to the one used in the proof of \cite[Theorem 2.2]{Dupuis1991}. To avoid redundancy, we omit the argument here.
\end{proof}

\subsection{Existence and uniqueness of solutions}\label{sec:esmwelldefined}

In this section we summarize results on existence and uniqueness of solutions to the ESP. We start by assuming the existence of a certain map that projects points in $\R^J$ onto $G$ in a way that is compatible with the directions of reflection $d(\cdot)$.

\begin{ass}\label{ass:projection}
There is a map $\pi:\R^J\mapsto G$ satisfying $\pi(x)=x$ for all $x\in G$ and $\pi(x)-x\in d(\pi(x))$ for all $x\not\in G$.
\end{ass}

For general results on the existence of a map $\pi$, see \cite[Section 4]{Dupuis1999a}. We now give a precise statement that the ESM is well defined under our stated assumptions.

\begin{theorem}\label{thm:speu}
Given an ESP $\{(d_i,n_i,c_i),i\in\allN\}$, suppose Assumption \ref{ass:setB} and Assumption \ref{ass:projection} hold. Then there exists a unique solution $(\z,\y)$ of the ESP for each $\x\in\cts$ and $\z$ satisfies $\z(0)=\pi(\x(0))$.
\end{theorem}

\begin{proof}
Define $\z(0)\doteq\pi(\x(0))$ and $\x^0\in\cts_G$ as in \eqref{eq:xS}, with $S=0$. By Theorem \ref{thm:esmlip} and Assumption \ref{ass:projection}, together with \cite[Lemma 2.6]{Ramanan2006}, there exists a unique solution $(\z^0,\y^0)$ of the ESP for $\x^0$ and $\z^0(0)=\x^0(0)=\pi(\x(0))\in G$. Define $\z(\cdot)\doteq\z^0(\cdot)$ and $\y(\cdot)\doteq\pi(\x(0))-\x(0)+\y^0(\cdot)$. According to Assumption \ref{ass:projection}, $\y(0)=\pi(\x(0))-\x(0)\in d(\z(0))$. It is readily verified that $(\z,\y)$ satisfies conditions 1--3 of the ESP for $\x$, so $(\z,\y)$ is a solution to the ESP for $\x$. Uniqueness of the solution then follows from the Lipschitz continuity property established in Theorem \ref{thm:esmlip}.
\end{proof}

We close this section with some examples of SPs that satisfy Assumption \ref{ass:setB} and Assumption \ref{ass:projection}. See Appendix \ref{apdx:esp} for an example of an ESP that is not an SP. Additional examples of SPs and ESPs can be found in \cite{Dupuis1991,Dupuis1998,Dupuis1999a,Dupuis1999,Dupuis2000,Ramanan2006}.

\begin{ex}\label{ex:1dsp}
Let $J=1$ and consider the one-dimensional SP $\{(e_1,e_1,0)\}$. In this case $G\doteq\R_+$ and it is readily verified that the SP satisfies Assumption \ref{ass:setB} with $B\doteq[-1,1]$, and Assumption \ref{ass:projection} with $\pi_1(x)\doteq x\vee0$ for all $x\in\R$. The one-dimensional SP was first formulated by Skorokhod \cite{Skorokhod1961} to construct pathwise reflected diffusions on $\R_+$. As is well known (see, e.g., \cite[Chapter 8]{Chung1990}), given $\x\in\cts$ such that $\x(0)\geq0$, the one-dimensional SM, which we denote by $\sm_1$, admits the following explicit representation:
	\be\label{eq:sm1} \sm_1(\x)(t)=\x(t)+\sup_{s \in [0,t]}(-\x(s))\vee0,\qquad t\in[0,\infty).\ee
It is readily verified that the above expression is also valid when $\x(0)<0$. Consequently, if $(\z,\y)$ is the solution of the one-dimensional SP for $\x\in\cts$, then due to the property $\z(t)=\x(t)+\y(t)$ for $t\in[0,\infty)$, it follows that
	\be\label{eq:y}\y(t)=\sup_{s\in[0,t]}(-\x(s))\vee0,\qquad t\in[0,\infty).\ee
\end{ex}

\begin{ex}\label{ex:gHR}
Consider an SP $\{(d_i,n_i,c_i),i=1,\dots,J\}$ with linearly independent directions of reflection $\{d_i,i=1,\dots,J\}$ (normalized so that $\ip{d_i,n_i}=1$ for $i=1,\dots,J$) and define the matrix $Q\in\R^{J\times J}$ by
	\be\label{eq:Q}
	Q_i^j=
	\begin{cases}
		|\ip{d_i,n_j}|&\text{if }i\neq j,\\
		0&\text{if }i=j.
	\end{cases}
	\ee
Suppose $\varrho(Q)$, the spectral radius of $Q$, satisfies $\varrho(Q)<1$. Then, according to the results in \cite[Section 2]{Dupuis1999}, the SP satisfies Assumption \ref{ass:setB} and Assumption \ref{ass:projection}. Furthermore, in the case $J=2$, the condition $\varrho(Q)<1$ is both necessary and sufficient for the SP to satisfy Assumption \ref{ass:setB} and Assumption \ref{ass:projection}. These SPs are sometimes referred to as generalized Harrison-Reiman SPs (see, e.g., \cite[Section 2]{Dupuis1999}) since they are a natural generalization of the so-called Harrison-Reiman SPs considered by Harrison and Reiman \cite{Harrison1981} (this class of SPs has the additional restrictions $n_i=e_i$ for $i=1,\dots,J$ and $\ip{d_j,e_i}\leq0$ for all $i\neq j$), which arise in single class open queueing networks \cite{Reiman1984} as well as in mathematical finance \cite{Karatzas2015}. Figure \ref{fig:gHR} depicts an example of a generalized Harrison-Reiman SP along with its associated set $B$. Since $\ip{d_2,e_1}>0$, this SP does not fall into the class of Harrison-Reiman SPs whose directional derivatives were characterized in \cite{Ramanan2006}.
\end{ex}

\begin{figure}[h!]
	\centering
	\begin{subfigure}{.5\textwidth}
		\centering
		\includegraphics[width=.6\textwidth]{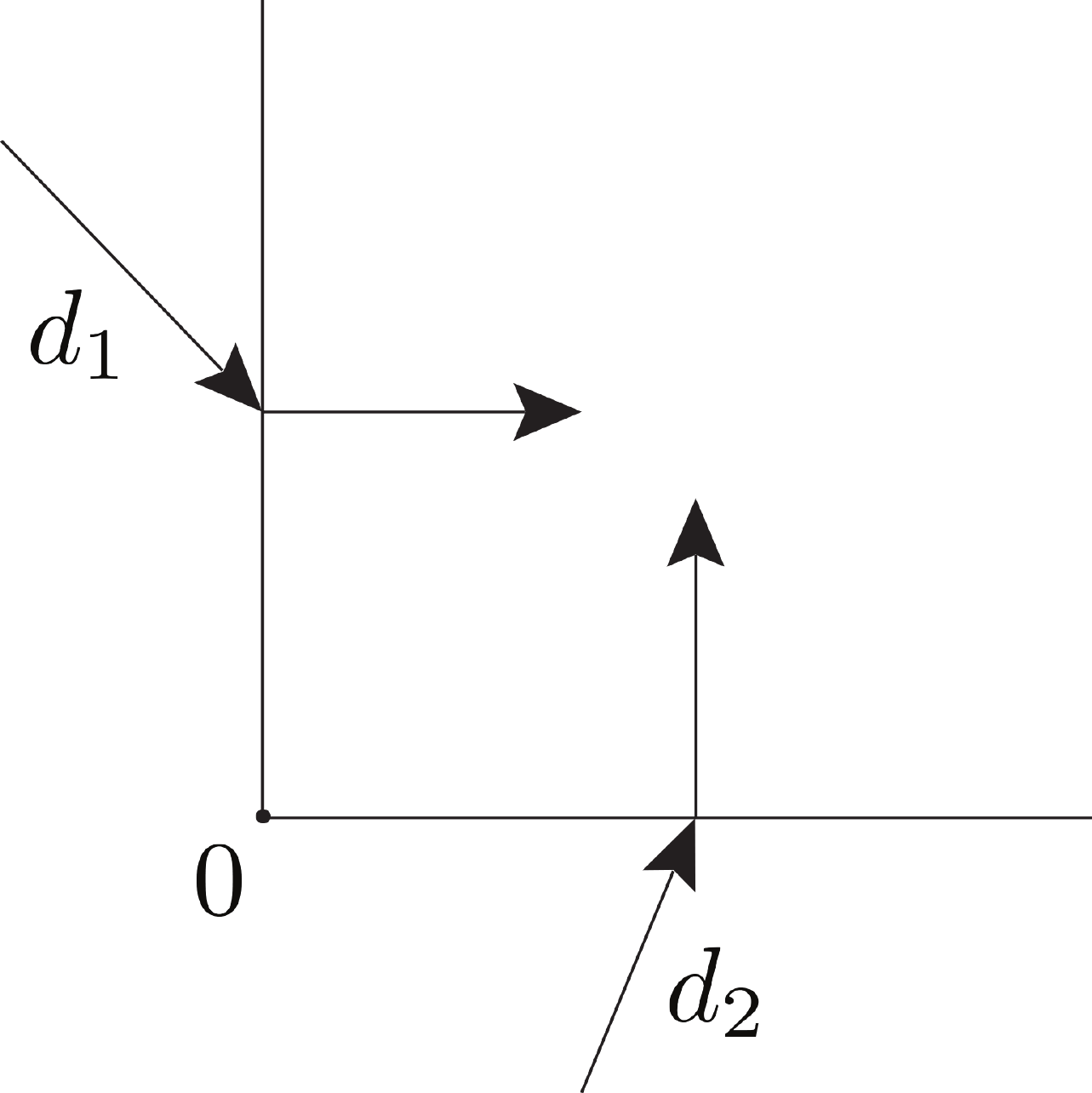}
		\caption{SP $\{(d_i,e_i,0),i=1,2\}$.}
		\label{fig:G}
	\end{subfigure}%
	\begin{subfigure}{.5\textwidth}
		\centering
		\includegraphics[width=.6\textwidth]{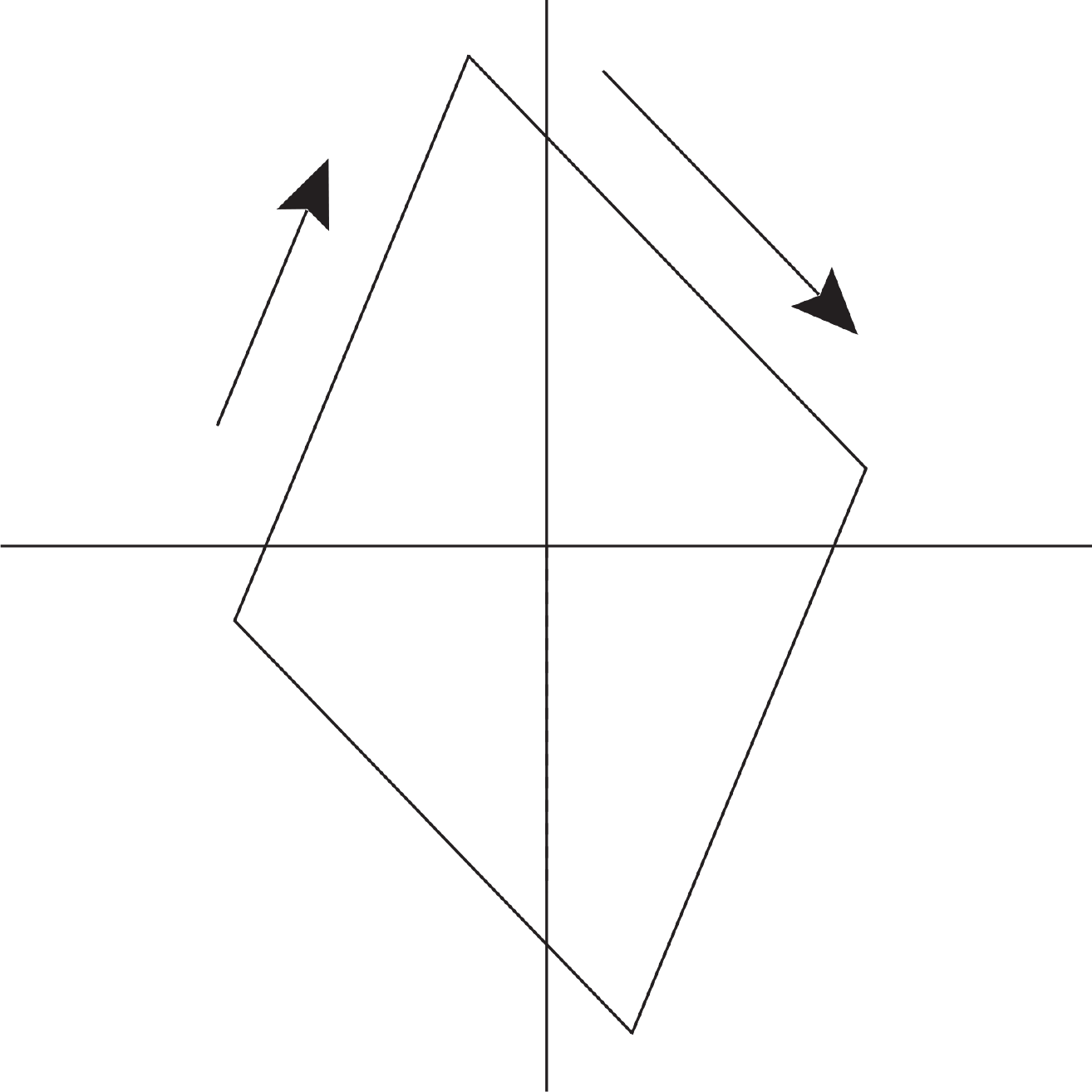}
		\caption{Set $B$.}
		\label{fig:setB}
	\end{subfigure}
	\caption{(A). An example of an SP on the nonnegative orthant with $d_1=(1,-1)'$ and $d_2=(1/2,1)'$. (B). An associated set $B$ satisfying the conditions in Assumption \ref{ass:setB}.}
	\label{fig:gHR}
\end{figure}

\subsection{Directional derivatives of the extended Skorokhod map}\label{sec:derivativesdef}

In this section we give the precise definition of a directional derivative of the ESM. Fix an ESP $\{(d_i,n_i,c_i),i\in\allN\}$ satisfying Assumption \ref{ass:setB} and Assumption \ref{ass:projection}, so by Theorem \ref{thm:esmlip} and Theorem \ref{thm:speu}, the associated ESM $\esm$ is well defined and Lipschitz continuous on $\cts$. Given $\x,\psi\in\cts$ and $\ve>0$, we define
 \begin{align}\label{eq:nablapsive}
	\nabla_\psi^\ve\esm(\x)&\doteq\frac{\esm(\x+\ve\psi)-\esm(\x)}{\ve}.
\end{align}

\begin{defn}\label{def:derivative}
Given $\x,\psi\in\cts$, the directional derivative of the ESM $\esm$ along the direction $\psi$ evaluated at $\x$, denoted $\nabla_\psi\esm(\x)$, is the function from $[0,\infty)$ into $\R^J$ that is equal to the pointwise limit of $\nabla_\psi^\ve\esm(\x)$ as $\ve\downarrow0$; that is, for $t\in[0,\infty)$,
\begin{equation}\label{eq:SMderivative}
	\nabla_\psi\esm(\x)(t)\doteq\lim_{\ve\downarrow0}\nabla_\psi^\ve\esm(\x)(t).
\end{equation}
\end{defn}

\begin{remark}
Let $\x,\psi\in\cts$. If the limit \eqref{eq:SMderivative} exists at $t\in[0,\infty)$, we say that $\nabla_\psi\esm(\x)(t)$ exists. If the limit \eqref{eq:SMderivative} exists for all $t\in[0,T)$ for some $T\in(0,\infty]$, we say that $\nabla_\psi\esm(\x)$ exists on $[0,T)$. We say that $\nabla_\psi\esm(\x)$ exists to mean the limit \eqref{eq:SMderivative} exists for all $t\in[0,\infty)$.
\end{remark}

The following proposition provides sufficient conditions for pointwise convergence of $\nabla_{\psi_\ve}^\ve\esm(\x)$ as $\ve\downarrow0$, where $\x\in\cts$ and $\{\psi_\ve\}_{\ve>0}$ is a family in $\cts$.
	
\begin{prop}\label{prop:psive}
Given an ESP $\{(d_i,n_i,c_i),i\in\allN\}$, suppose Assumption \ref{ass:setB} and Assumption \ref{ass:projection} hold. Let $\x,\psi\in\cts$ and $T\in(0,\infty]$ be such that $\nabla_\psi\esm(\x)(t)$ exists on $[0,T)$, and let $\{\psi_\ve\}_{\ve>0}$ be a family in $\cts$ such that $\psi_\ve\to\psi$ in $\cts$ as $\ve\downarrow0$. Then for all $t\in[0,T)$,
	\be\label{eq:nablapsiveve}\lim_{\ve\downarrow0}\nabla_{\psi_\ve}^\ve\esm(\x)(t)=\nabla_\psi\esm(\x)(t).\ee
\end{prop}

\begin{proof}
Let $t\in[0,T)$. By \eqref{eq:nablapsive} and Theorem \ref{thm:esmlip},
\begin{align*}
	|\nabla_{\psi_\ve}^\ve\esm(\x)(t)-\nabla_\psi^\ve\esm(\x)(t)|&=\frac{|\esm(\x+\ve\psi_\ve)(t)-\esm(\x+\ve\psi)(t)|}{\ve}\leq\lip_{\esm}\norm{\psi_\ve-\psi}_t,
\end{align*}
which converges to zero as $\ve\downarrow0$. This, along with the triangle inequality and \eqref{eq:SMderivative}, implies \eqref{eq:nablapsiveve}.
\end{proof}

\section{Main results}\label{sec:main}

In this section we present our main results on the existence and characterization of directional derivatives of the ESM and pathwise differentiability of a class of reflected diffusions in the nonnegative quadrant. We consider directional derivatives of the ESM evaluated at paths whose images under the ESM satisfy a certain boundary jitter property, which we define in Section \ref{sec:jittermain}. Later, in Section \ref{sec:jitter}, we verify that the boundary jitter property is satisfied by a large class of RBMs in the nonnegative quadrant. In Section \ref{sec:dpmain} we introduce the DP. In Section \ref{sec:ESMderivative} we present our main result on the existence of directional derivatives of the ESM and the characterization of their right continuous regularizations as solutions of the DP. In Section \ref{sec:directions} we use our main result to study directional derivatives of solutions of the ESM with respect to the directions of reflection. In Section \ref{sec:pathwise} we apply our results to study pathwise differentiability of a large class of two-dimensional RBMs in the nonnegative quadrant.

\subsection{Definition of the boundary jitter property}\label{sec:jittermain}

In order to prove existence of directional derivatives of the ESM evaluated at $\x\in\cts_G$, we require that the solution $(\z,\y)$ of the ESP for $\x$ satisfies certain conditions at the boundary $\partial G$. We collectively refer to these conditions as the \emph{boundary jitter property}. Recall that $|\allN(x)|$ denotes the cardinality of the set $\allN(x)$ defined in \eqref{eq:allNx}. Let $\S$ and $\U$ respectively denote the smooth and nonsmooth parts of the boundary $\partial G$; that is,
	\be\label{eq:Sset}\S\doteq\lcb x\in\partial G:|\allN(x)|=1\rcb,\ee
and
	\be\label{eq:Uset}\U\doteq\partial G\setminus\S=\lcb x\in\partial G:|\allN(x)|\geq2\rcb.\ee

\begin{defn}\label{def:jitter}
Given an ESP $\{(d_i,n_i,c_i),i\in\allN\}$, $T\in(0,\infty]$ and $(\z,\y)\in\cts([0,T):G)\times\cts([0,T):\R^J)$. We say that $(\z,\y)$ satisfies the boundary jitter property on $[0,T)$ if the following hold:
\begin{itemize}
	\item[1.] If $t\in[0,T)$ is such that $\z(t)\in\S$, then for all $s<t<u<T$, $\y$ is nonconstant on $(s\vee0,u)$.
	\item[2.] On the interval $[0,T)$, $\z$ does not spend positive Lebesgue time in $\U$; that is,
		$$\int_0^T 1_{\U}(\z(t))dt=0.$$
	\item[3.] If $\z(0)\in\U$, then for each $i\in\allN(\z(0))$ and every $\delta\in(0,T)$, there exists $u\in(0,\delta)$ such that $\allN(\z(u))=\{i\}$.
	\item[4.] If $t\in(0,T)$ is such that $\z(t)\in\U$, then for each $i\in\allN(\z(t))$ and every $\delta\in(0,t)$, there exists $s\in(t-\delta,t)$ such that $\allN(\z(s))=\{i\}$.
\end{itemize}
\end{defn}

\begin{remark}
When $T=\infty$, we omit the interval $[0,\infty)$ and write $(\Z,\Y)$ satisfies the boundary jitter property to mean conditions 1--4 hold on $[0,\infty)$. When the pair $(\z,\y)$ is defined on an interval that contains $[0,T)$, we write $(\z,\y)$ satisfies the boundary jitter property on $[0,T)$ to mean the restriction of $(\z,\y)$ to $[0,T)$ satisfies the boundary jitter property on $[0,T)$. The boundary jitter property depends on the ESP $\{(d_i,n_i,c_i),i\in\allN\}$; however, we omit this dependence since the ESP will be clear from the context. Since conditions 2, 3 and 4 of the boundary jitter property only depend on $\z$, we often write $\z$ satisfies condition 2, 3 or 4 of the boundary jitter property to mean $(\z,\y)$ satisfies condition 2, 3 or 4, respectively, of the boundary jitter property.
\end{remark}

\begin{remark}
Given an ESP on the half space (i.e., when $N=1$), the set $\U$ is empty and so conditions 2--4 of the boundary jitter property hold automatically.
\end{remark}

Condition 1 can be interpreted to mean that whenever $\z$ lies on the smooth part of the boundary, it must be actively constrained to remain in $G$. Condition 2 is self-explanatory. Condition 3 states that if the path $\z$ starts on the nonsmooth part of the boundary, then it must hit the smooth part of each face intersecting the point $\z(0)$ infinitely often immediately after time zero. Condition 4 states that whenever the path $\z$ is on a nonsmooth part of the boundary at time $t\in(0,\infty)$, it must hit the smooth part of each face intersecting the point $\z(t)$ infinitely often immediately before time $t$. See Figure \ref{fig:jitter} (in Section \ref{sec:projdp}) for an illustration of a path $\z$ that satisfies condition 4 of the boundary jitter property.

By imposing the boundary jitter property on $(\z,\y)$, we will be able characterize the right continuous regularizations of directional derivatives of the ESM as solutions of the DP associated with $\z$, which we introduce in the next section. In Section \ref{sec:jitter} we show that a large class of RBMs in the nonnegative quadrant a.s.\ satisfy the boundary jitter property, which allows us to establish pathwise differentiability for these RBMs and obtain a nice characterization of their derivatives. In a forthcoming work, we show the boundary jitter property is a.s.\ satisfied by a large class of reflected diffusions in convex polyhedral domains, thus allowing us to prove pathwise differentiability for these reflected diffusions.

\subsection{Statement of the derivative problem}\label{sec:dpmain}

Let $\x\in\cts$ and $\z\doteq\esm(\x)$. The DP associated with $\z$ is a certain time-inhomogeneous Skorokhod-type problem where both the domain and directions of reflection vary (discontinuously) in time. Other works that have considered SPs or ESPs in domains that vary in time include \cite{Burdzy2004,Burdzy2009} in the one-dimensional setting, and \cite{Nystrom2010a} in the multidimensional setting with time-varying domains and oblique directions reflection.

In order to state the DP, recall the definitions of $\allN(x)$ and $d(x)$ given in \eqref{eq:allNx} and \eqref{eq:dx}, and that $\dr$ denotes the space of $\R^J$-valued right-continuous functions with finite left limits on $[0,\infty)$. For $x\in\partial G$, define the linear subspace
	\be\label{eq:Hx}H_x\doteq\bigcap_{i\in\allN(x)}\lcb y\in\R^J:\ip{y,n_i}=0\rcb,\ee
and for $x\in G^\circ$, set $H_x\doteq\R^J$. We now give a precise formulation of the DP.

\begin{defn}\label{def:dp}
Given an ESP $\{(d_i,n_i,c_i),i\in\allN\}$ and $\x\in\cts$, suppose $(\z,\y)$ is a solution of the ESP for $\x$. Let $\psi\in\dr$. Then $(\phi,\eta)\in\dr\times\dr$ solves the DP associated with $\z$ for $\psi$ if $\eta(0)\in\spaan[d(\z(0))]$ and for all $t\in[0,\infty)$, the following conditions hold:
\begin{itemize}
	\item[1.] $\phi(t)=\psi(t)+\eta(t)$;
	\item[2.] $\phi(t)\in H_{\z(t)}$;
	\item[3.] for all $s\in[0,t)$,
		\be\label{eq:etatetas}\eta(t)-\eta(s)\in\spaan\lsb\cup_{u\in(s,t]}d(Z(u))\rsb.\ee
\end{itemize}
If there exists a unique solution $(\phi,\eta)$ to the DP for $\psi$, we write $\phi=\dm_\z(\psi)$ and refer to $\dm_\z$ as the derivative map (DM) associated with $\z$.
\end{defn}

\begin{remark}
When there is no confusion regarding $\z$, we omit the phrase ``associated with $\z$'' and simply say $(\phi,\eta)$ solves the DP for $\psi$.
\end{remark}

\begin{remark}
Given $T\in(0,\infty)$, we say $(\phi,\eta)\in\dr([0,T):\R^J)\times\dr([0,T):\R^J)$ solves the DP for $\psi$ on $[0,T)$ if $\eta(0)\in\spaan[d(\z(0))]$ conditions 1--3 hold for $t\in[0,T)$. If the pair $(\phi,\eta)$ is defined on an interval that strictly contains $[0,T)$, we say $(\phi,\eta)$ solves the DP for $\psi$ on $[0,T)$ if the restriction of $(\phi,\eta)$ to $[0,T)$ solves the DP for $\psi$ on $[0,T)$.
\end{remark}

\begin{remark}
If $\eta$ is discontinuous at $t\in(0,\infty)$, then by condition 3 of the DP, the definition of $d(\cdot)$ given in \eqref{eq:dx}, the continuity of $\z$ and the upper semicontinuity of $\allN(\cdot)$ (Lemma \ref{lem:allNusc}),
	\be\label{eq:etatetatminus}\eta(t)-\eta(t-)\in\bigcap_{s<t}\spaan\lsb\cup_{u\in(s,t]} d(\z(u))\rsb=\spaan\lsb d(\z(t))\rsb.\ee
\end{remark}

\begin{remark}
The definition of the DP has many similarities to the definition of the ESP. In particular, conditions 1--3 of the DP correspond to conditions 1--3 of the ESP, but with $\phi,\psi,\eta$, $H_{\z(t)}$ and ``span'' in place of $\z,\x,\y$, $G$ and ``cone''. Here, $H_{\z(t)}$ is time-dependent and for each $t\in[0,\infty)$ is equal to the intersection of finitely many hyperplanes, whereas $G$ is fixed (in time) and equal to the intersection of finitely many half spaces.
\end{remark}

In Section \ref{sec:dp} we study relevant properties of the DP and the associated DM.

\subsection{Existence and characterization of directional derivatives}\label{sec:ESMderivative}

Throughout this section we fix an ESP $\{(d_i,n_i,c_i),i\in\allN\}$ satisfying Assumption \ref{ass:setB} and Assumption \ref{ass:projection}.

We prove existence of directional derivatives of the ESM evaluated at $\x\in\cts_G$ provided the solution $(\z,\y)$ of the ESP for $\x$ satisfies the boundary jitter property, and only up until the first time $\z(t)$ reaches a certain, possibly empty, subset of the nonsmooth part of the boundary $\U$. To be precise, let
	\be\label{eq:Wset}\W\doteq\lcb x\in\U:\spaan(H_{x}\cup d(x))\neq\R^J\rcb,\ee
and given a solution $(\z,\y)$ to the ESP, define $\tau$ to be the first time that $\z$ reaches $\W$; that is,
	\be\label{eq:tau}\tau\doteq\inf\{t\in[0,\infty):\z(t)\in\W\}.\ee
In Appendix \ref{apdx:nonemptyW} we demonstrate that the directional derivative of the ESM does not necessarily exist at $t=\tau$. In Lemma \ref{lem:Wempty} we show that if the following mild linear independence assumption on the directions of reflection also holds, then the set $\W$ is empty, so $\tau=\infty$ holds trivially.

\begin{ass}\label{ass:linearlyind}
For each $x\in\partial G$, $\{d_i,i\in\allN(x)\}$ is a set of linearly independent vectors.
\end{ass}

\begin{remark}\label{rmk:linearlyind}
Under Assumption \ref{ass:linearlyind}, the set $\V$, defined in \eqref{eq:setV}, is clearly empty. Therefore, according to Remark \ref{rmk:sp}, $(\z,\y)$ is a solution of the ESP $\{(d_i,n_i,c_i),i\in\allN\}$ if and only if $(\z,\y)$ is a solution of the SP $\{(d_i,n_i,c_i),i\in\allN\}$.
\end{remark}

The following lemma will be used to establish the existence of $\nabla_\psi\esm(\x)$ at $t=0$.

\begin{lem}\label{lem:projxv}
The following limit exists for all $(x,v)\in G\times\R^J$:
	\be\label{eq:pixv}\nabla_v\pi(x)\doteq\lim_{\ve\downarrow0}\frac{\pi(x+\ve v)-\pi(x)}{\ve}.\ee
Furthermore, $\nabla_v\pi(x)-v\in\conv[d(\pi(x))]=\conv[d(x)]$.
\end{lem}

\begin{proof}
According to the discussion in \cite[Section 5.3]{Dupuis1991}, the limit \eqref{eq:pixv} exists. Due to the convergence $\pi(x+\ve v)\to\pi(x)=x$ as $\ve\downarrow0$ and the upper semicontinuity of $\allN(\cdot)$ (Lemma \ref{lem:allNusc}), $\allN(\pi(x+\ve v))\subseteq\allN(\pi(x))=\allN(x)$ for $\ve>0$ sufficiently small. This, along with Assumption \ref{ass:projection}, implies that for all $\ve>0$ sufficiently small,
\begin{align*}
	\frac{\pi(x+\ve v)-\pi(x)}{\ve}-v=\frac{\pi(x+\ve v)-(x+\ve v)}{\ve}\in\conv[d(\pi(x))]=\conv[d(x)].
\end{align*}
The final assertion of the lemma then follows from taking limits as $\ve\downarrow0$ and because $\conv[d(x)]$ is a closed set.
\end{proof}

The second part of our main result is to relate the directional derivative $\nabla_\psi\esm(\x)$ to the unique solution $(\phi,\eta)$ of the DP associated with $\z$ for $\psi$. In order to state this result, we define a functional 
	$$\Theta_{\z}:\dr([0,\tau):\R^J)\mapsto\dlr([0,\tau):\R^J)$$ 
so that $\Theta_\z(\phi)$ and $\nabla_\psi\esm(\x)$ are equal on $(0,\tau)$. To this end, for each $x\in\S$, the smooth part of the boundary, let $i_x\in\allN$ denote the unique index such that $\allN(x)=\{i_x\}$ and define
	\be\label{eq:Gx}G_x\doteq\lcb y\in\R^J:\ip{y,n_{i_x}}\geq0\rcb.\ee
Given $f\in\dr([0,\tau):\R^J)$, define $\Theta_{\z}(f)$ as follows: for each $t\in[0,\tau)$,
\be\label{eq:ThetaZ}
	\Theta_{\z}(f)(t)\doteq
	\begin{cases}
		f(t),&\text{if }\z(t)\in G\setminus\S,\\
		f(t),&\text{if }\z(t)\in\S,\;f(t-)\not\in G_{\z(t)}\\
		f(t-),&\text{if }\z(t)\in\S,\;f(t-)\in G_{\z(t)}.
	\end{cases}
\ee

We can now state our main result on directional derivatives of the ESM.

\begin{theorem}\label{thm:main}
Fix an ESP $\{(d_i,n_i,c_i),i\in\allN\}$ satisfying Assumption \ref{ass:setB} and Assumption \ref{ass:projection}. Given $\x\in\cts_G$, let $(\z,\y)$ denote the solution to the ESP for $\x$ and define $\tau$ as in \eqref{eq:tau}. Suppose $(\z,\y)$ satisfies the boundary jitter property (Definition \ref{def:jitter}) on $[0,\tau)$. Then for all $\psi\in\cts$,
\begin{itemize}
	\item[1.] $\nabla_\psi\esm(\x)$ exists on $[0,\tau)$ and lies in $\dlr([0,\tau):\R^J)$;
	\item[2.] there exists a unique solution $(\phi,\eta)$ to the DP associated with $\z$ for $\psi$ on $[0,\tau)$;
	\item[3.] $\phi(t)=\nabla_\psi\esm(\x)(t+)$ for all $t\in[0,\tau)$;
	\item[4.] $\nabla_\psi\esm(\x)(0)=\nabla_{\psi(0)}\pi(\x(0))$ and $\nabla_\psi\esm(\x)=\Theta_{\z}(\phi)$ on $(0,\tau)$.
\end{itemize}
Moreover, under Assumption \ref{ass:linearlyind}, $\tau=\infty$.
\end{theorem}

Here, we provide a brief outline of the proof of Theorem \ref{thm:main}, which is deferred to Section \ref{sec:psiconst}. Given $\x\in\cts_G$, let $\z\doteq\esm(\x)$. In Section \ref{sec:theta2} we prove existence of and characterize directional derivatives $\nabla_\psi\esm(\x)$ of an ESM up until the first time $\z$ reaches the nonsmooth part of the boundary $\U$. We denote this time by $\theta_2$. Roughly speaking, given an interval such that $\z$ hits at most a single face $F_i$, we can exploit prior results on directional derivatives of the one-dimensional SM, which are reviewed in Section \ref{sec:nablasm1}, to prove existence of and characterize $\nabla_\psi\esm(\x)$ on the interval. We then patch together these results to prove existence of and characterize $\nabla_\psi\esm(\x)$ on $[0,\theta_2)$. The proof of existence of $\nabla_\psi\esm(\x)$ on $[0,\theta_2)$ does not require that the boundary jitter property hold; however condition 1 of the boundary jitter property is needed to characterize the right continuous regularization of $\nabla_\psi\esm(\x)$ as the solution of the DP on $[0,\theta_2)$.

In Section \ref{sec:tau} we prove existence of and characterize $\nabla_\psi\esm(\x)$ on $[0,\tau)$. The key challenge is to characterize $\nabla_\psi\esm(\x)(t)$ at times $t\in[0,\tau)$ that $\z(t)\in\U$. We first show that it suffices to consider $\psi$ that lie in a dense subset of $\cts$ consisting of paths that are constant about times that the path $\z$ lies in $\U$. We then use the boundary jitter property, along with properties of certain (oblique) derivative projection operators, which are introduced in Section \ref{sec:derivativeprojectionoperator}, to characterize $\nabla_\psi\esm(\x)(t)$ at such times. In particular, given $t\in(0,\infty)$ such that $\z(t)\in\U$, then the boundary jitter property implies that $\z$ hits the relative interior of each face that intersects $\z(t)$ infinitely often in any left neighborhood of $t$. Roughly speaking, each time $\z$ reaches the relative interior of a face, $\nabla_\psi\esm(\x)$ is projected onto a hyperplane associated with that face along a direction that lies in the span of the direction of reflection associated with that face. As a consequence, understanding $\nabla_\psi\esm(\x)(t)$ when $\z(t)\in\U$ is largely reduced to the analysis of countable sequences of derivative projection operators, which is carried out in Section \ref{sec:contraction}.

\subsection{Differentiability with respect to directions of reflection}\label{sec:directions}

In this section we prove differentiability of certain ESMs with respect to the directions of reflection. In order to consider perturbations to the directions of reflection, we consider a sub-class of SPs which we first introduce in Section \ref{sec:reflectionmatrix}.

\subsubsection{Reflection matrix and the decomposition of the constraining term}\label{sec:reflectionmatrix}

We consider a class of SPs in which the constraining term $\y$ can be uniquely decomposed to describe its action along each face. This class contains many SPs of interest that arise in applications, including the class of generalization Harrison-Reiman SPs described in Example \ref{ex:gHR}.

Given an SP $\{(d_i,n_i,c_i),i\in\allN\}$ satisfying Assumption \ref{ass:linearlyind}, define the \emph{reflection matrix} $R\in\R^{J\times N}$ by
	\be\label{eq:reflectionmatrix}R\doteq\begin{pmatrix}d_1&\cdots&d_N\end{pmatrix}.\ee
The following lemma describes a decomposition of the $J$-dimensional constraining function $\y$ into an $N$-dimensional nondecreasing (componentwise) path $L$ such that for each $i\in\allN$, $L^id_i$ denotes the pushing along face $F_i$.

\begin{lem}\label{lem:local}
Suppose the SP $\{(d_i,n_i,c_i),i\in\allN\}$ satisfies Assumption \ref{ass:linearlyind}. Then given a solution $(\z,\y)$ of the SP for $\x\in\cts_G$, there exists a unique function $L\in\cts([0,\infty):\R_+^N)$ such that $\y=R L$ and for each $i\in\allN$, $L^i$ is nondecreasing and
	\be\label{eq:dLi}\int_0^\infty1_{\{\z(s)\not\in F_i\}}dL^i(s)=0.\ee
Moreover, there is a positive constant $\lip_L<\infty$ such that if, for $k=1,2$, $(\z_k,\y_k)$ is the solution of the SP $\{(d_i,n_i,c_i),i\in\allN\}$ for $\x_k\in\cts_G$ and $L_k$ is as above, but with $\z_k,\y_k,L_k$ in place of $\z,\y,L$, respectively, then for all $T\in(0,\infty)$,
	\be\label{eq:locallip}\norm{L_1-L_2}_T\leq\lip_L\norm{\x_1-\x_2}_T.\ee
\end{lem}

If $N=J$, then $L$ is given by $L=R^{-1}\y$ and it is readily checked that the properties stated in the lemma follow from the definition of the ESP (Definition \ref{def:esp}).  For completeness, the proof of Lemma \ref{lem:local} for arbitrary $N\in\N$ is given in Appendix \ref{apdx:local}.

\subsubsection{Perturbations of the reflection matrix}

Fix a set of vectors $\{v_i,i\in\allN\}$ in $\R^J$ such that $\ip{v_i,n_i}=0$ for each $i\in\allN$ and define the matrix $V\in\R^{J\times N}$ by 
	\be\label{eq:Smatrix}V\doteq\begin{pmatrix}v_1&\cdots&v_N\end{pmatrix}.\ee
We consider SPs with perturbed directions of reflection given by $\{d_i+\ve v_i,i\in\allN\}$. Since the directions of reflection can always be renormalized so that $\ip{d_i+\ve v_i,n_i}=\ip{d_i,n_i}=1$ for each $i\in\allN$, the condition $\ip{v_i,n_i}=0$ for each $i\in\allN$ is without loss of generality.

\begin{ass}\label{ass:SPperturb}
Given vectors $\{v_i,i\in\allN\}$ satisfying $\ip{v_i,n_i}=0$ for each $i\in\allN$, there is an $\ve_0\in(0,\infty)$ such that for each $\ve\in[0,\ve_0)$, the SP $\{(d_i+\ve v_i,n_i,c_i),i\in\allN\}$ satisfies Assumption \ref{ass:setB}, Assumption \ref{ass:projection}, and Assumption \ref{ass:linearlyind}.
\end{ass}

In the following lemma we show that the generalized Harrison-Reiman SPs introduced in Example \ref{ex:gHR} satisfy Assumption \ref{ass:SPperturb}.

\begin{lem}\label{lem:SPperturb}
Given an SP $\{(d_i,n_i,c_i),i=1,\dots,J\}$ satisfying Assumption \ref{ass:linearlyind}, define the matrix $Q\in\R^{J\times J}$ as in \eqref{eq:Q} and let $\varrho(Q)$ denote its spectral radius. If $\varrho(Q)<1$, then the SP satisfies Assumption \ref{ass:SPperturb}. Furthermore, in the case $J=2$, the SP satisfies Assumption \ref{ass:SPperturb} only if $\varrho(Q)<1$.
\end{lem}

\begin{proof}
Suppose $\varrho(Q)<1$. According to Example \ref{ex:gHR}, the SP satisfies Assumption \ref{ass:setB} and Assumption \ref{ass:projection}. Let $\{v_i,i=1,\dots,J\}$ be such that $\ip{v_i,n_i}=0$ for all $i\in\allN$. For $\ve>0$, define the matrix $Q_\ve\in\R^{J\times J}$ by
	$$
	(Q_\ve)_i^j=
	\begin{cases}
		|\ip{d_i+\ve v_i,n_j}|&\text{if }i\neq j,\\
		0&\text{if }i=j.
	\end{cases}
	$$
Since the spectral radius and determinant of a matrix depend continuously on its entries, there exists $\ve_0\in(0,\infty)$ such that for all $\ve\in(0,\ve_0)$, $\varrho(Q_\ve)<1$ and $\{d_i+\ve v_i,i=1,\dots,J\}$ is a set of linearly independent vectors. For such $\ve$, the SP $\{(d_i+\ve v_i,n_i,c_i),i=1,\dots,J\}$ lies in the class of generalized Harrison-Reiman SPs considered in Example \ref{ex:gHR}. Thus, the SP $\{(d_i,n_i,c_i),i=1,\dots,J\}$ satisfies Assumption \ref{ass:SPperturb}. The last assertion follows because, as stated in Example \ref{ex:gHR}, if $J=2$, the SP $\{(d_i,n_i,c_i),i=1,2\}$ satisfies Assumption \ref{ass:setB} only if $\varrho(Q)<1$.
\end{proof}

We now state our main theorem on differentiability of the ESM with respect to the directions of reflection.

\begin{theorem}\label{thm:directions}
Fix an SP $\{(d_i,n_i,c_i),i\in\allN\}$ satisfying Assumption \ref{ass:SPperturb} and let $\sm$ be the associated SM. Let $(\z,\y)$ denote the solution of the SP for $\x\in\cts_G$ and let $L$ be as in Lemma \ref{lem:local}. Suppose $\{v_i,i\in\allN\}$ are vectors in $\R^J$ satisfying $\ip{v_i,n_i}=0$ for each $i\in\allN$ and define the matrix $V\in\R^{J\times N}$ as in \eqref{eq:Smatrix}. Let $\zeta\in\cts$ and define $\psi\doteq \zeta+VL\in\cts$. Let $\ve_0\in(0,\infty)$ be as in Assumption \ref{ass:SPperturb} and for each $\ve\in(0,\ve_0)$ let $\sm_\ve$ denote the SM associated with the SP $\{(d_i+\ve v_i,n_i,c_i),i\in\allN\}$. Given $T\in(0,\infty]$, assume that $\nabla_\psi\sm(\x)$ exists on $[0,T)$. Then
	\be\label{eq:directionslimit}\lim_{\ve\downarrow0}\frac{\sm_\ve(\x+\ve\zeta)(t)-\sm(\x)(t)}{\ve}=\nabla_\psi\sm(\x)(t),\qquad t\in[0,T).\ee
\end{theorem}

\begin{remark}
In the statement of Theorem \ref{thm:directions} we do not impose conditions under which $\nabla_\psi\sm(\x)$ exists. Broad sufficient conditions for existence follow from our other results. In particular, if $\theta_2$ is defined as in \eqref{eq:theta2}, then by Proposition \ref{prop:theta2}, $\nabla_\psi\sm(\x)$ exists on $[0,\theta_2)$. If $\tau$ is defined as in \eqref{eq:tau}, and $\x$ satisfies the boundary jitter property on $[0,\tau)$, then by Theorem \ref{thm:main}, $\nabla_\psi\sm(\x)$ exists on $[0,\tau)$. Alternatively, if $N=J$ and the reflection matrix $R$ satisfies the Harrison-Reiman condition (see \cite[Definition 1.2]{Mandelbaum2010}), then by \cite[Theorem 1.1]{Mandelbaum2010}, $\nabla_\psi\sm(\x)$ exists.
\end{remark}

\begin{proof}
For each $\ve\in(0,\ve_0)$, let $(\z_\ve,\y_\ve)$ be the solution to the SP $\{(d_i+\ve v_i,n_i,c_i),i\in\allN\}$ for $\x+\ve\zeta$. By Lemma \ref{lem:local}, there exists $L_\ve\in\cts([0,\infty):\R_+^N)$ such that $\y_\ve=(R+\ve V)L_\ve$ and for each $i\in\allN$, $L_\ve^i$ is nondecreasing and can only increase when $\z_\ve$ is on face $F_i$. Thus, $\z_\ve=\x+\ve\psi_\ve+\widetilde{\y}_\ve$, where $\psi_\ve\doteq\zeta+VL_\ve$ and $\widetilde{\y}_\ve\doteq R L_\ve$, and $\widetilde{\y}_\ve$ satisfies \eqref{eq:ytys}, with $\widetilde{\y}_\ve$ and $\z_\ve$ in place of $\y$ and $\z$, respectively. It follows that $(\z_\ve,\widetilde{\y}_\ve)$ solves the SP $\{(d_i,n_i,c_i),i\in\allN\}$ for $\x+\ve\psi_\ve$. Below, we prove that $\psi_\ve$ converges to $\psi$ in $\cts$ as $\ve\downarrow0$, where we recall that $\psi=\zeta+VL$. Since $\nabla_\psi\sm(\x)$ exists on $[0,T)$ by assumption, Lemma \ref{prop:psive} will then imply that \eqref{eq:directionslimit} holds.

We are left to show that $\psi_\ve$ converges to $\psi$ in $\cts$ as $\ve\downarrow0$. Let $t\in(0,\infty)$. By the definitions of $\psi_\ve$ and $\psi$, and \eqref{eq:locallip}, 
	$$\norm{\psi_\ve-\psi}_t\leq\norm{V}\norm{L_\ve-L}_t\leq\ve\kappa_L\norm{V}\norm{\psi_\ve}_t\leq\ve\kappa_L\norm{V}(\norm{\psi}_t+\norm{\psi_\ve-\psi}_t),$$
where $\norm{V}$ denotes the finite operator norm of the matrix $V$ when viewed as a linear operator from $\R^N$ to $\R^J$. Rearranging (for $\ve>0$ sufficiently small), we see that
	$$\limsup_{\ve\downarrow0}\norm{\psi_\ve-\psi}_t\leq\lim_{\ve\downarrow0}\frac{\ve\lip_L\norm{V}\norm{\psi}_t}{1-\ve\lip_L\norm{V}}=0,$$
which completes the proof of the lemma.
\end{proof}
	
\subsection{Pathwise differentiability of reflected Brownian motion}\label{sec:pathwise}

In this section we establish what we refer to as pathwise differentiability of an RBM in the nonnegative quadrant $G\doteq\R_+^2$. Fix a filtered probability space $(\Omega,\F,\{\F_t\},\P)$ satisfying the usual conditions (see, e.g., \cite[Chapter II, Definition 67.1]{Rogers2000}). To avoid notational confusion with other sections in this work, we emphasize that \emph{throughout this subsection $W,\Z,\X,L,\nabla\Z,\psi,\phi,\zeta$ denote stochastic processes on the filtered probability space $(\Omega,\F,\{\F_t\},\P)$}. Fix a two-dimensional standard Brownian motion $W$ on $(\Omega,\F,\{\F_t\},\P)$. An RBM on the nonnegative quadrant is defined as follows.

\begin{defn}\label{def:rbm}
Given an initial condition $x\in G$, drift vector $b\in\R^2$, dispersion matrix $\sigma\in\R^{2\times 2}$ and directions of reflection $\{d_i,i=1,2\}$ in $\R^2$, suppose the SP $\{(d_i,e_i,0),i=1,2\}$ satisfies Assumption \ref{ass:setB} and Assumption \ref{ass:projection}. Let $\sm$ denote the associated SM and define the reflection matrix $R\in\R^{2\times 2}$ by $R\doteq\begin{pmatrix}d_1&d_2\end{pmatrix}$. Then the RBM on $G$ associated with $(x,b,\sigma,R)$ is an $\{\F_t\}$-adapted two-dimensional process $\Z$ defined by 
	\be\label{eq:Zxbsigma}\Z\doteq\sm(\X),\ee
where $\X=\{\X_t,t\in[0,\infty)\}$ is the $\{\F_t\}$-adapted two-dimensional process defined by 
	\be\label{eq:X}\X_t\doteq x+bt+\sigma W_t,\qquad t\in[0,\infty).\ee
\end{defn}

See Figure \ref{fig:rbm} for an example of a sample path of an RBM.

\begin{remark}\label{rmk:local}
By Lemma \ref{lem:local}, given the RBM $\Z$ associated with $(x,b,\sigma,R)$, there exists an $\{\F_t\}$-adapted two-dimensional process $L=\{L_t,t\in[0,\infty)\}$ such that 
	\be\label{eq:Zlocal} \Z_t=x+bt+\sigma W_t+RL_t,\qquad t\in[0,\infty),\ee
and for $i=1,2$, $L^i$ is nondecreasing and can only increase when $\Z$ lies on face $F_i$.
\end{remark}

The following result states that under a nondegeneracy condition, the RBM, along with its constraining term, $\P$ a.s.\ satisfies the boundary jitter property. The proof of Proposition \ref{prop:rbmjitter} is given in Section \ref{sec:rbmjitter}.

\begin{prop}\label{prop:rbmjitter}
Given an SP $\{(d_i,e_i,0),i=1,2\}$ satisfying Assumption \ref{ass:setB} and Assumption \ref{ass:projection}, $x\in G$, $b\in\R^2$ and $\sigma\in\R^{2\times 2}$, suppose that $\sigma\sigma'$ is a nondegenerate covariance matrix. Set $R\doteq\begin{pmatrix}d_1&d_2\end{pmatrix}$. Let $\Z$ be the RBM associated with $(x,b,\sigma,R)$ and define $\X$ as in \eqref{eq:X}. Then $\P$ a.s.\ $(\Z,\Z-\X)$ satisfies the boundary jitter property (Definition \ref{def:jitter}).
\end{prop}

We now present our main result on the existence of pathwise derivatives of a two-dimensional RBM on the nonnegative quadrant.

\begin{theorem}\label{thm:pathwise}
Fix an SP $\{(d_i,e_i,0),i=1,2\}$ satisfying Assumption \ref{ass:SPperturb} and set $R\doteq\begin{pmatrix}d_1&d_2\end{pmatrix}$. Let $\Z$ be the RBM associated with $(x,b,\sigma,R)$, define $\X$ as in \eqref{eq:X} and $L$ as in Remark \ref{rmk:local}. Then $\P$ a.s.\ the following holds: given $y\in\R^2$, $c\in\R^2$, $\theta\in\R^{2\times 2}$ and $V\in\R^{2\times 2}$ such that $x+\ve y\in\R_+^2$ for all $\ve>0$ sufficiently small and $V$ has zeros along its diagonal, let $\ve_0\in(0,\infty)$ be as in Assumption \ref{ass:SPperturb}. For each $\ve\in(0,\ve_0)$ such that $x+\ve y\in\R_+^2$, let $\Z^\ve$ denote the RBM associated with $(x+\ve y,b+\ve c,\sigma+\ve\theta,R+\ve V)$. Then
\begin{itemize}
	\item[1.] the process $\nabla\Z=\{\nabla\Z_t,t\in[0,\infty)\}$, defined by
		\be\label{eq:nablaZnablaesm}\nabla\Z_t\doteq\lim_{\ve\downarrow0}\frac{\Z_t^\ve-\Z_t}{\ve},\qquad t\in[0,\infty),\ee
	exists, takes values in $\dlr$ and satisfies $\nabla\Z=\nabla_\psi\sm(\X)$, where $\psi=\{\psi_t,t\in[0,\infty)\}$ is given by
		\be\label{eq:psiyct}\psi_t\doteq y+ct+\theta W_t+VL_t,\qquad t\in[0,\infty);\ee
	\item[2.] if $\phi=\{\phi_t,t\in[0,\infty)\}$ denotes the right continuous regularization of $\nabla\Z=\{\nabla\Z_t,t\in[0,\infty)\}$, then $(\phi,\phi-\psi)$ solves the DP associated with $\Z$ for $\psi$.
\end{itemize}
\end{theorem}

See Figure \ref{fig:rbmpathwise} for an illustration of a sample path of an RBM along with an associated pathwise derivative.

\begin{remark}
Given an SP $\{(d_i,e_i,0),i=1,2\}$ satisfying Assumption \ref{ass:linearlyind}, define $Q\in\R^{2\times 2}$ as in \eqref{eq:Q}. Then by Lemma \ref{lem:SPperturb}, Assumption \ref{ass:SPperturb} holds if and only if $\varrho(Q)<1$.
\end{remark}

\begin{figure}[h!]
	\centering
	\begin{subfigure}{.5\textwidth}
		\centering
		\includegraphics[width=.7\textwidth]{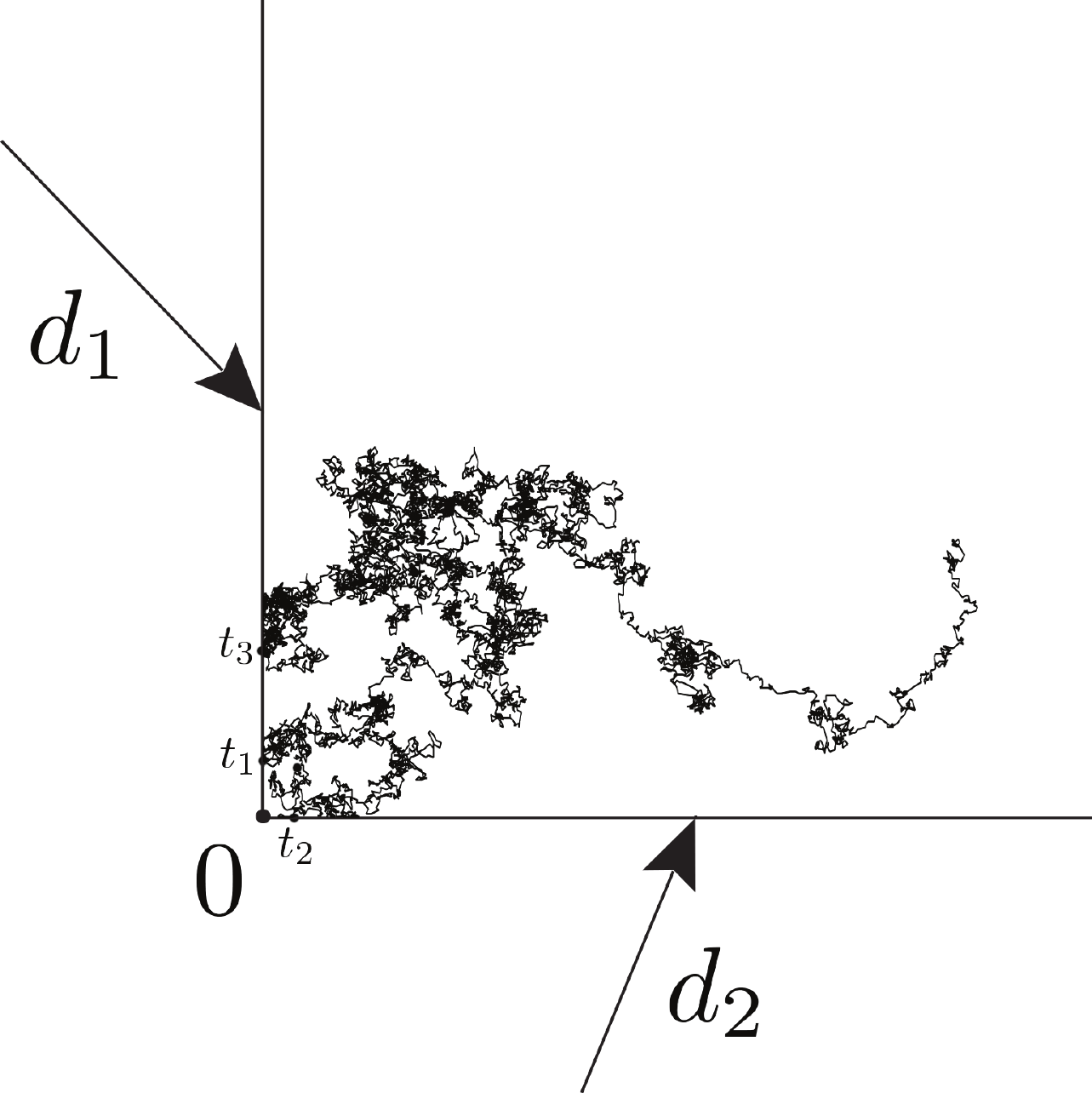}
		\caption{A sample path of an RBM.}
		\label{fig:rbm}
	\end{subfigure}%
	\begin{subfigure}{.5\textwidth}
		\centering
		\includegraphics[width=.7\textwidth]{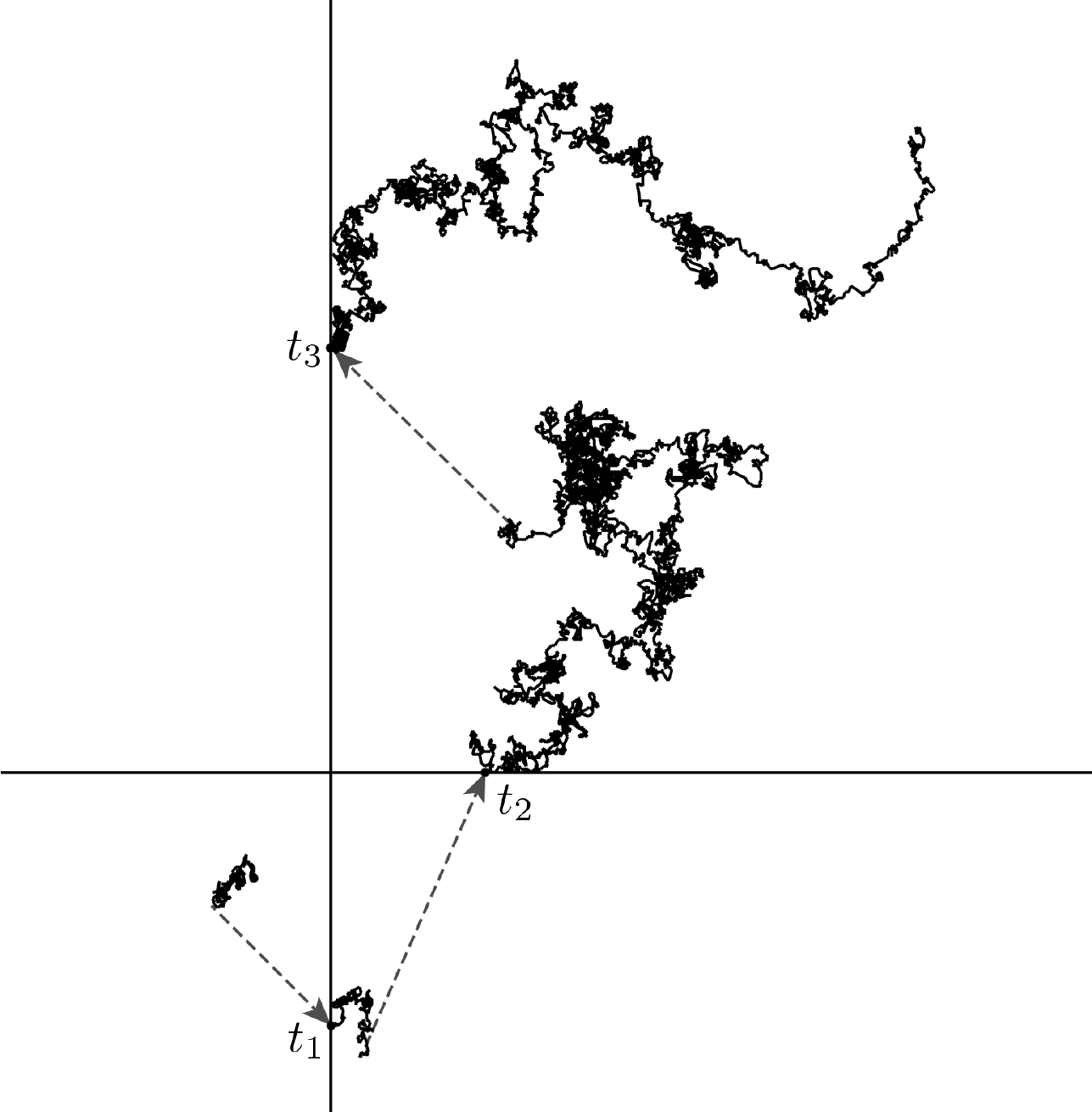}
		\caption{An associated pathwise derivative.}
		\label{fig:phi}
	\end{subfigure}
	\caption{(A). A sample path of an RBM $\Z$ on the nonnegative orthant associated with the SP depicted in Figure \ref{fig:G}. (B) The sample path of an associated pathwise derivative $\nabla\Z$. The gray dashed lines denote (some, though not all) discontinuity points of $\nabla\Z$. In both plots, $t_1\doteq\inf\{t>0:Z_t\in \partial G\}$, $t_2\doteq\inf\{t>t_1:Z_t\in F_2\}$ and $t_3\doteq\inf\{t>t_2:Z_t\in F_1\}$.}\label{fig:rbmpathwise}
\end{figure}

\begin{proof}
By Proposition \ref{prop:rbmjitter}, the event
	\be\label{eq:ZZXjitter}\{(\Z,\Z-\X)\text{ satisfies the boundary jitter property}\},\ee
has $\P$-measure one. Let $y\in\R^2$, $c\in\R^2$, $\theta\in\R^{2\times 2}$ and $V\in\R^{2\times 2}$ be such that $x+\ve y\in\R_+^2$ for all $\ve>0$ sufficiently small and $V$ has zeros along its diagonal. For each $\ve\in(0,\ve_0)$, let $\sm_\ve$ denote the SM associated with SP $\{(d_i+\ve v_i,e_i,0),i=1,2\}$. Define $\zeta=\{\zeta_t,t\in[0,\infty)\}$ by 
	$$\zeta_t=y+ct+\theta W_t,\qquad t\in[0,\infty),$$
and define $\psi=\{\psi_t,t\in[0,\infty)\}$ as in \eqref{eq:psiyct}. Then for each $\ve\in(0,\ve_0)$ such that $x+\ve y\in\R_+^2$, $\Z^\ve=\sm_\ve(\X+\ve\zeta)$. Furthermore, Assumption \ref{ass:linearlyind} implies $\W=\emptyset$, where $\W$ is defined as in \eqref{eq:Wset}. It follows that $\tau\doteq\inf\{t\in[0,\infty):\Z_t\in\W\}=\infty$. By part 1 of Theorem \ref{thm:main} and Theorem \ref{thm:directions}, on the set \eqref{eq:ZZXjitter}, $\nabla_\psi\sm(\X)$ exists, and
\begin{align*}
	\nabla\Z_t=\lim_{\ve\downarrow0}\frac{\sm_\ve(\X+\ve\zeta)(t)-\sm(\X)(t)}{\ve}&=\nabla_\psi\sm(\X)(t),\qquad t\in[0,\infty).
\end{align*}
This, along with Theorem \ref{thm:main}, completes the proof of the theorem.
\end{proof}

In forthcoming work we prove a result that is analogous to Theorem \ref{thm:pathwise}, but for a larger class of reflected diffusions in convex polyhedral domains, and study stability properties of the joint process $\{(\Z_t,\phi_t),t\in[0,\infty)\}$.

\section{The boundary jitter property}\label{sec:jitter}

In this section we verify that the boundary jitter property (Definition \ref{def:jitter}) holds for certain two-dimensional RBMs in the nonnegative quadrant and discuss some useful ramifications of the boundary jitter property.

\subsection{Verification for reflected Brownian motion in the nonnegative quadrant}\label{sec:rbmjitter}

In this section we prove Proposition \ref{prop:rbmjitter}. Fix an SP $\{(d_i,e_i,0),i=1,2\}$, $b\in\R^2$, $\sigma\in\R^{2\times 2}$ as in the statement of Proposition \ref{prop:rbmjitter}. Let $\sm$ denote the associated SM and $R\doteq\begin{pmatrix}d_1&d_2\end{pmatrix}$. Observe that $\U=\{0\}$ and $F_i=\{x\in G:\ip{x,e_i}=0\}$ for $i=1,2$. Since $\sigma\sigma'$ is symmetric and positive definite, the square matrix $\sigma$ is nondegenerate. Recall that $\sigma^{-1}\in\R^{2\times 2}$ denotes the matrix inverse of $\sigma$. 

Throughout this section, we fix a filtered probability space $(\Omega,\F,\{\F_t\},\P)$ satisfying the usual conditions. To avoid confusion with other sections of this work, we emphasize that \emph{throughout this subsection, $W,\Z^x,\X^x,\Y^x,L,\zeta$ denote stochastic processes on the filtered probability space $(\Omega,\F,\{\F_t\},\P)$}. Fix a two-dimensional standard Brownian motion $W$ on $(\Omega,\F,\{\F_t\},\P)$. Also, for $u>0$, define 
	\be\label{eq:Cu}C_u\doteq\lcb W_t\geq0\;\forall\;t\in[0,u]\rcb\ee 
and recall the well known property that $\P(C_u)=0$ for every $u>0$. For each $x\in G$, define $\X_t^x\doteq x+bt+\sigma W_t$ for $t\geq0$, $\Z^x\doteq\sm(\X^x)$ and $\Y^x\doteq\Z^x-\X^x$. We begin by establishing that $(\Z^x,\Y^x)$ satisfies condition 1 of the boundary jitter property.

\begin{lem}\label{lem:jitter1}
For each $x\in G$, $\P$ a.s.\ $(\Z^x,\Y^x)$ satisfies condition 1 of the boundary jitter property.
\end{lem}

\begin{proof}
For $x\in G$, define
	$$A_0^x\doteq\bigcap_{u\in\Q\cap(0,\infty)}\lcb\Y^x\text{ is nonconstant on }(0,u)\rcb,$$
and
	$$A^x\doteq\bigcap_{s\in\Q\cap(0,\infty)}\bigcap_{u\in\Q\cap(s,\infty)}\lcb\Y^x\text{ is nonconstant on }(s,u)\rcb\cup\lcb\Z_t^x\in G^\circ\;\forall\;t\in(s,u)\rcb.$$
Then for $x\in\S$, $A_0^x\cap A^x$ is contained in the event that $(\Z^x,\Y^x)$ satisfies condition 1 of the boundary jitter property, and for $x\in G^\circ\cup\U$, $A^x$ is contained in the event that $(\Z^x,\Y^x)$ satisfies condition 1 of the boundary jitter property. We show that $\P(A_0^x)=1$ for all $x\in\partial G$ and $\P(A^x)=1$ for all $x\in G$, which will complete the proof of the lemma.	
	
Suppose $x\in\partial G$ and let $i\in\{1,2\}$ be such that $i\in\allN(x)$. Fix $u\in\Q\cap(0,\infty)$. By conditions 1 and 2 of the ESP,
\begin{align*}
	\{\Y^x\text{ is constant on }[0,u]\}&\subseteq\{\ip{\X_t^x,e_i}=\ip{\Z_t^x,e_i}\geq0\;\forall\;t\in[0,u]\}.
\end{align*}
Due to Girsanov's theorem (see, e.g., \cite[Chapter IV, Theorem 38.5]{Rogers2000a}), there is a probability measure $\widetilde{\P}$ on $(\Omega,\F)$ that is equivalent to $\P$ such that under $\widetilde{\P}$,
 	$$\widetilde{\X}\doteq\lcb\frac{1}{|\sigma^i|}\ip{\X_t^x,e_i},t\in[0,u]\rcb=\lcb\frac{1}{|\sigma^i|}\ip{b,e_i}t+\frac{1}{|\sigma^i|}\sigma^iW_t,t\in[0,u]\rcb.$$
is a standard one-dimensional Brownian motion on $[0,u]$. Thus, by \eqref{eq:Cu},
	\be\label{eq:Yxconstant}\widetilde{\P}(\Y^x\text{ is constant on }[0,u])=\widetilde{\P}(\widetilde{\X}_t\geq0\;\forall\;t\in[0,u])=\P(C_u)=0.\ee
Since $\P$ and $\widetilde{\P}$ are equivalent, $\P(A_0^x)=1$. 

Now suppose $x\in G$. Fix $s\in\Q\cap(0,\infty)$ and $u\in\Q\cap(s,\infty)$. Define the stopping time $\rho^s$ as follows:
	$$\rho^s\doteq\inf\{t>s:\Z_t^x\in\partial G\}.$$
Note that $\{\rho^s\geq u\}=\{\Z_s^x\in G^\circ\;\forall\;t\in(s,u)\}$, so we are left to consider the set $\{\rho^s\in[s,u)\}$. On the set $\{\rho^s\in[s,u)\}$, define $\X^{\rho^s},\Z^{\rho^s},\Y^{\rho^s}$ as in \eqref{eq:xS}--\eqref{eq:yS}, but with $\rho^s,\X^x,\Z^x,\Y^x$ in place of $S,\x,\z,\y$, respectively. By the time-shift property of the ESP (Lemma \ref{lem:esmshift}), $(\Z^{\rho^s},\Y^{\rho^s})$ solves the ESP for $\X^{\rho^s}$. This, along with the strong Markov property for $W$ and the fact that $b$ and $\sigma$ are constant, imply that for $y\in\partial G$, $\Y^{\rho^s}$, conditioned on $\{\rho^s\in[s,u),\rho^s=y\}$, is equal in distribution to $\Y^y$. Hence, by \eqref{eq:Yxconstant} and the fact that $\P$ and $\widetilde{\P}$ are equivalent, for $r\in\Q\cap(0,\infty)$,
\begin{align*}
	&\P(\Y^{\rho^s}\text{ is nonconstant on }[0,r]|\rho^s\in[s,u))\\
	&\qquad=\int_{\partial G}\P(\Y^y\text{ is nonconstant on }[0,r])\P(\Z_{\rho^s}^x\in dy|\rho^s\in[s,u))\\
	&\qquad=1.
\end{align*}
Since the above holds for every $r\in\Q\cap(0,\infty)$, it follows that $\P(A^x)=1$, which completes the proof.
\end{proof}

We now turn to the proof that condition 2 of the boundary jitter property holds.

\begin{lem}\label{lem:jitter2}
For each $x\in G$, 
	\be\label{eq:int1zero}\P\lb\int_0^\infty 1_{\partial G}(\Z_t^x)dt=0\rb=1.\ee
Consequently, $\P$ a.s.\ $\Z^x$ satisfies condition 2 of the boundary jitter property.
\end{lem}

\begin{proof}
It is readily verified that $\Z^x$ is an SRBM associated with $(G,b,\sigma\sigma',R)$ that starts from $x$ (see \cite[Definition 1.1]{Taylor1993}, with $X=\X^x$ and $Y=L^x$). Then by \cite[Lemma 2.1]{Taylor1993}, the lemma follows.
\end{proof}

Next, we consider condition 4 of the boundary jitter property. For the following lemma, given $x\in G$, let $L^x=\{L_t^x,t\in[0,\infty)\}$ be as in Remark \ref{rmk:local}, but with $\Z^x$ and $L^x$ in place of $\Z$ and $L$, respectively.

\begin{lem}\label{lem:jitter4}
For each $x\in G$, $\P$ a.s.\ $\Z^x$ satisfies condition 4 of the boundary jitter property.
\end{lem}

\begin{proof}
For $x\in G$ define
	$$\rho_0^{s,x}\doteq\inf\lcb t>s:\Z_t^x\in\U\rcb,\qquad s\in[0,\infty),$$
and for $i=1,2$, define
	$$\rho_i^{s,x}\doteq\inf\lcb t>s:\Z_t^x\in F_i\rcb,\qquad s\in[0,\infty).$$
Clearly $\rho_i^{s,x}\leq\rho_0^{s,x}$ for $i=1,2$ and $s\in(0,\infty)$. For each $x\in G$,
\begin{align*}
	&\lcb \Z^x\text{ satisfies condition 4 of the boundary jitter property}\rcb\\
	&\qquad=\bigcap_{s\in\Q\cap(0,\infty)}\bigcap_{i=1,2}\lcb\rho_0^{s,x}=\rho_i^{s,x}<\infty\rcb^c.
\end{align*}
Fix $x\in G$, $s\in\Q\cap(0,\infty)$ and $i\in\{1,2\}$. It clearly suffices to show that $\P(\rho_0^{s,x}=\rho_i^{s,x}<\infty)=0$. By the Markov property and the fact that $\P(\Z_s^x\in\partial G)=0$ (see, e.g., \cite[Lemma 5.7]{Budhiraja2007} and its proof), we have
	$$\P(\rho_0^{s,x}=\rho_i^{s,x}<\infty)=\int_{G^\circ}\P(\rho_0^{0,y}=\rho_i^{0,y}<\infty)\P(\Z_s^x\in dy).$$
Thus, we are left to show that for each $y\in G^\circ$ and $i=1,2$, $\P(\rho_0^{0,y}=\rho_i^{0,y}<\infty)=0$; or equivalently, $\P(\Z_{\rho_i^{0,y}\wedge T}^y\in\U)=0$ for all $T<\infty$. Fix $y\in G^\circ$ and $T<\infty$. By Girsanov's theorem, there exists a probability measure $\widetilde{\P}$ on $(\Omega,\F)$ equivalent to $\P$ such that under $\widetilde{\P}$, $\widetilde{W}=\{\sigma^{-1}bt+W_t,t\in[0,T]\}$ is a standard two-dimensional Brownian motion on $[0,T]$. By Remark \ref{rmk:local}, $[L^y]^i$ is constant on $[0,\rho_i^{0,y}]$. Let $j\doteq 3-i$. Then under $\widetilde{\P}$, the stopped process 
	$$\widetilde{\Z}\doteq\lcb\sigma^{-1}\Z_{t\wedge\rho_i^{0,y}}^y,t\in[0,T]\rcb=\lcb\sigma^{-1}y+\widetilde{W}_t+\lsb L_{t\wedge\rho_i^{0,y}}^y\rsb^j(\sigma^{-1}d_j),t\in[0,T]\rcb$$
is a two-dimensional RBM on $[0,T]$ with zero drift, identity covariance matrix on the half plane $\{z\in\R^2:\ip{z,\sigma'e_j}\geq0\}$ with oblique direction of reflection $\sigma^{-1}d_j$ at the boundary $\{z\in\R^2:\ip{z,\sigma'e_j}=0\}$, starting at $\sigma^{-1}y\neq0$ and stopped at the first time it reaches the line $\{z\in\R^2:\ip{z,\sigma'e_i}=0\}$. Additionally, $\widetilde{\Z}_T\in\U$ if and only if $\Z_{\rho_i^{0,y}\wedge T}^y\in\U$. By \cite[Theorem 2.2]{Varadhan1984}, with $\xi=\pi$, $\theta_1=-\theta_2$ and $\alpha=0$, an RBM on the half plane with zero drift, identity covariance matrix and nonzero starting position a.s.\ does not reach the origin in finite time. Therefore, $\widetilde{\P}(Z_{\rho_i^{0,y}\wedge T}\in\U)=0$. Since $\P$ and $\widetilde{\P}$ are equivalent, this completes the proof.
\end{proof}

In preparation for the proof that $\Z^x$ satisfies condition 3 of the boundary jitter property, we state the following useful lemma. Recall the definition of a Cauchy process (see, e.g., \cite[Section 3]{Spitzer1958}).

\begin{lem}\label{lem:cauchy}
Let $\zeta=\{\zeta_t,t\in[0,\infty)\}$ be a one-dimensional Cauchy process on $(\Omega,\F,\{\F_t\},\P)$. Then
	\be\label{eq:cauchylimsup}\P\lb\limsup_{t\downarrow0}\frac{\zeta_t}{t}=\infty\rb=\P\lb\liminf_{t\downarrow0}\frac{\zeta_t}{t}=-\infty\rb=1.\ee
\end{lem}

\begin{proof}
See, for example, \cite[Chapter VIII, Theorem 5(iii)]{Bertoin1996}.
\end{proof}

We now prove condition 3 of the boundary jitter property. Since $\U=\{0\}$, we only need to consider the RBM starting at the origin.

\begin{lem}\label{lem:jitter3}
$\P$ a.s. $\Z^0$ satisfies condition 3 of the boundary jitter property.
\end{lem}

\begin{proof}
Define
	\be\label{eq:rhoi}\rho_i^0\doteq\inf\lcb t>0:\Z_t^0\in F_i\rcb,\qquad i=1,2.\ee
According to Lemma \ref{lem:jitter4}, $\P$ a.s.\ the following statement holds: given $i\in\{1,2\}$, if $\Z_t^0\in\U$ for some $t>0$, then $\allN(\Z_s^0)=\{i\}$ for some $s\in[0,t)$. Thus, it suffices to show that $\P(\rho_1^0=0)=\P(\rho_2^0=0)=1$. We prove $\P(\rho_1^0=0)$ with the proof of $\P(\rho_2^0=0)$ being identical. 

Define $\X_t^0\doteq bt+\sigma W_t$ for $t\geq0$ and set $\Y^0\doteq\Z^0-\X^0$ so that $(\Z^0,\Y^0)$ solves the SP for $\X^0$. On the set $\{\rho_1^0>0\}$, $\ip{\Z_t^0,e_1}>0$ for all $t\in(0,\rho_1^0)$, or equivalently, $\allN(\Z_t^0)\subseteq\{2\}$ for all $t\in(0,\rho_1^0)$. Therefore, by condition 3 of the ESP and the fact $\ip{d_2,e_2}>0$, $\{\ip{\Y_t^0,e_2},t\in[0,\rho_1^0)\}$ is nondecreasing and can only increase when $\ip{\Z_t^0,e_2}=0$. Since $\ip{\Z_t^0,e_2}\geq0$ for all $t\in[0,\rho_1^0)$, it follows that $(\ip{\Z^0,e_2},\ip{\Y^0,e_2})$ solves the one-dimensional SP for $\ip{\X^0,e_2}$ on $[0,\rho_1^0)$. This, combined with the explicit form of the one-dimensional SM given in \eqref{eq:sm1}, \eqref{eq:rhoi}, condition 3 of the ESP and the fact that $\ip{d_2,e_2}=1$, implies that, on the set $\{\rho_1^0>0\}$,
	\be\label{eq:Zt0}\Z_t^0=\X_t^0+\sup_{s\in[0,t]}(-\ip{\X_s^0,e_2}) d_2,\qquad t\in[0,\rho_1^0).\ee
On the set $\{\rho_1^0>0\}\cap\{\ip{\X_t^0,e_2}<0\text{ i.o. as }t\downarrow0\}$, we rearrange \eqref{eq:Zt0} to obtain
	$$\frac{\ip{\X_t^0,e_1}}{\sup_{s\in[0,t]}(-\ip{\X_s^0,e_2})}>-\ip{d_2,e_1},\qquad t\in(0,\rho_1^0).$$
However, we will show that in fact $\widetilde{\P}(\ip{\X_t^0,e_2}<0\text{ i.o. as }t\downarrow0)=1$ and
	\be\label{eq:liminfXe2}\widetilde{\P}\lb\liminf_{t\downarrow0}\frac{\ip{\X_t^0,e_1}}{\sup_{s\in[0,t]}(-\ip{\X_s^0,e_2})}=-\infty\rb=1.\ee
This will imply that $\P(\rho_1^0=0)=1$.

Fix $T<\infty$. By Girsanov's theorem, there exists a probability measure $\widetilde{\P}$ on $(\Omega,\F)$ equivalent to $\P$ such that under $\widetilde{\P}$,
	$$\widetilde{W}\doteq\lcb\sigma^{-1}\X_t^0,t\in[0,T]\rcb=\lcb\sigma^{-1}bt+W_t,t\in[0,T]\rcb$$
is a standard two-dimensional Brownian motion on $[0,T]$. It follows that, under $\widetilde{\P}$, $\{\ip{\X_t^0,e_2},t\in[0,T]\}=\{\ip{\widetilde{W}_t,\sigma'e_2},t\in[0,T]\}$ is a one-dimensional Brownian motion (with zero drift and infinitesimal variance $|\sigma'e_2|^2$) on $[0,T]$. By \eqref{eq:Cu}, we have 
	\be\label{eq:Wsigmae2io}\widetilde{\P}(\ip{\X_t^0,e_2}<0\text{ i.o. as }t\downarrow0)=\P(\cap_{u\in\Q\cap(0,T)}C_u^c)=1.\ee
Now let $v_2\doteq|\sigma'e_2|^{-1}\sigma'e_2\in\mathbb{S}^1$ and choose $v_1\in\mathbb{S}^1$ such that $\ip{v_1,v_2}=0$. Since $\sigma$ is nondegenerate, there exist $\alpha,\beta\in\R$ with $\alpha\neq0$ such that $\sigma'e_1=\alpha v_1+\beta v_2$. Observe that under $\widetilde{\P}$, $B^1\doteq\{\ip{\widetilde{W}_t,v_1},t\in[0,T]\}$ and $B^2\doteq\{\ip{\widetilde{W}_t,v_2},t\in[0,T]\}$ are independent one-dimensional Brownian motions on $[0,T]$ and, for $t\in[0,T]$,
\begin{align*}
	\frac{\ip{\X_t^0,e_1}}{\sup_{s\in[0,t]}(-\ip{\X_s^0,e_2})}&=\frac{\ip{\widetilde{W}_t,\sigma'e_1}}{\sup_{s\in[0,t]}(-\ip{\widetilde{W}_s,\sigma'e_2})}=\frac{\alpha B_t^1+\beta B_t^2}{|\sigma'e_2|\sup_{s\in[0,t]}(-B_s^2)}.
\end{align*}
For $c>0$, let $T_c\doteq\inf\{t>0:-B_t^2=c\}$. Then, to show \eqref{eq:liminfXe2}, it suffices to show that
	\be\label{eq:Btaucinfty}\widetilde{\P}\lb\liminf_{c\downarrow0}\frac{\alpha B_{T_c}^1-\beta c}{|\sigma'e_2|c}=-\infty\rb=\widetilde{\P}\lb\liminf_{c\downarrow0}\frac{\alpha B_{T_c}^1}{c}=-\infty\rb=1.\ee
Since $\{B_{T_c}^1,c\in[0,T]\}$ is a Cauchy process on $[0,T]$ (see, e.g., \cite[Lemma 3]{Spitzer1958}), \eqref{eq:Btaucinfty} follows from Lemma \ref{lem:cauchy}. This proves \eqref{eq:liminfXe2}, which along with \eqref{eq:Wsigmae2io}, completes the proof of the lemma.
\end{proof}

\begin{proof}[Proof of Proposition \ref{prop:rbmjitter}]
Let $x\in G$. By Lemma \ref{lem:jitter1}, Lemma \ref{lem:jitter2}, Lemma \ref{lem:jitter4}, Lemma \ref{lem:jitter3} and the fact that $\U=\{0\}$, $\P$ a.s.\ $(\Z^x,\Y^x)$ satisfies the boundary jitter property.
\end{proof}

\subsection{Ramifications of the boundary jitter property}\label{sec:ramificationsjitter}

The following two lemmas state consequences of the boundary jitter property that are used in the proof that the directional derivative of the ESM exists at times $t\in[0,\infty)$ that $\z(t)\in\U$, the nonsmooth part of the boundary defined in \eqref{eq:Uset}.

\begin{lem}\label{lem:xij}
Given $U\in(0,\infty]$ and $\z\in\cts([0,U):G)$, suppose $\z$ satisfies condition 4 of the boundary jitter property (Definition \ref{def:jitter}) on $[0,U)$ and $0\leq S<T<U$ are such that $\z(T)\in\U$ and $\allN(\z(t))\subsetneq\allN(\z(T))$ for all $t\in[S,T)$. Then there is a nested increasing sequence 
	\be\label{eq:Sxi0} S\doteq\xi_0<s_1\leq\xi_1<\cdots<s_j\leq\xi_j<\cdots<T\ee
such that $\xi_j\to T$ as $j\to\infty$ and for each $j\in\N$, $\z(\xi_j)\in\partial G$ and 
	\be\label{eq:xijsj}\bigcup_{t\in[\xi_{j-1},s_j)}\allN(\z(t))\subseteq\allN(\z(\xi_{j-1}))\qquad\text{and}\qquad\bigcup_{t\in[s_j,\xi_j]}\allN(\z(t))\subseteq\allN(\z(\xi_j)).\ee
\end{lem}

See Figure \ref{fig:jitter} (in Section \ref{sec:projdp}) for an illustration of a path $\z$ that satisfies condition 4 of the boundary jitter property with the sequence of times in \eqref{eq:Sxi0} marked.

\begin{proof}
To construct the nested increasing sequences \eqref{eq:Sxi0}, recursively define, for $j\in\N$, $s_j$ to be the first time after $\xi_{j-1}$ such that $\z(s_j)\in F_i$ for some $i\not\in\allN(\z(\xi_{j-1}))$; that is,
	\be\label{eq:sj}s_j\doteq\inf\{t\in(\xi_{j-1},T]:\allN(\z(t))\not\subseteq\allN(\z(\xi_{j-1}))\},\ee
and let
	\be\label{eq:tj}\xi_j\doteq\sup\{t\in[s_j,T):\allN(\z(s))\subseteq\allN(\z(t))\;\forall\;s\in[s_j,t]\}.\ee
We claim that for each $j\in\N$, 
	\be\label{eq:Sxij} S\leq\xi_{j-1}<s_j\leq\xi_j<T\qquad\text{and}\qquad\z(\xi_j)\in\partial G.\ee 
This, along with \eqref{eq:sj} and \eqref{eq:tj}, will establish \eqref{eq:Sxi0} and \eqref{eq:xijsj}. 

To prove the claim \eqref{eq:Sxij}, we use the principle of mathematical induction. By definition, $\xi_0\doteq S\in[S,T)$. Now suppose $\xi_{j-1}\in[S,T)$ for some $j\in\N$. By \eqref{eq:sj}, the continuity of $\z$ and the upper semicontinuity of $\allN(\cdot)$ (Lemma \ref{lem:allNusc}), we have $s_j>\xi_{j-1}$. To see that $s_j<T$, first choose $i\in\allN(\z(T))\setminus\allN(\z(\xi_{j-1}))$, where the set is nonempty because, by assumption, the strict inclusion $\allN(\z(t))\subsetneq\allN(\z(T))$ holds for all $t\in[S,T)$. According to condition 4 of the boundary jitter property, since $\z(T)\in\U$, there exists $t\in(\xi_{j-1},T)$ such that $\allN(\z(t))=\{i\}$. Thus, \eqref{eq:sj} implies that $s_j\leq t<T$. This, together with the previously established lower bound $s_j>\xi_{j-1}$, implies $s_j\in(\xi_{j-1},T)$. By definition \eqref{eq:tj}, $s_j\leq\xi_j$.

Proceeding, we show that $\xi_j<T$. Since $\z(T)\in\U$, condition 4 of the boundary jitter property, with $t=T$ and $\delta=T-s_j$, implies that for $i\in\allN(\z(T))$,
	$$t_0^i\doteq\inf\{t\in[s_j,T):\allN(\z(t))=\{i\}\}<T.$$
Consequently,
	\be\label{eq:t0maxiT}t_0\doteq\max_{i\in\allN(\z(T))}t_0^i<T\qquad\text{and}\qquad\bigcup_{u\in[s_j,t_0]}\allN(\z(u))=\allN(\z(T)).\ee
Combining \eqref{eq:t0maxiT} with the strict inclusion $\allN(\z(t))\subsetneq\allN(\z(T))$ for all $t\in[t_0,T)$, and the definition \eqref{eq:tj} for $\xi_j$, we have $\xi_j<t_0<T$. To see that $\z(\xi_j)\in\partial G$, first observe that \eqref{eq:sj} clearly implies that $\z(s_j)\in\partial G$. Since \eqref{eq:tj} and the continuity of $\z$ imply that $\allN(\z(s_j))\subseteq\allN(\z(\xi_j))$, this ensures $\z(\xi_j)\in\partial G$. Thus, we have proved the claim in \eqref{eq:Sxij}, and hence, that \eqref{eq:Sxi0} and \eqref{eq:xijsj} hold.

We are left to show that $\xi_j\to T$ as $j\to\infty$. Since \eqref{eq:Sxij} implies $\xi_j$ is increasing and bounded above by $T$, there exists $\xi_\infty\leq T$ such that $\xi_j<\xi_\infty$ for all $j\in\N$ and $\xi_j\to\xi_\infty$ as $j\to\infty$. By the continuity of $\z$ and the upper semicontinuity of $\allN(\cdot)$, there exists $j_0\in\N$ such that $\allN(\z(\xi_j))\subseteq\allN(\z(\xi_\infty))$ holds for all $j\geq j_0$. The inclusions in \eqref{eq:xijsj} imply that for each $j\in\N$, $\allN(\z(t))\subseteq\allN(\z(\xi_j))$ for all $t\in[s_j,s_{j+1})$. Combining these properties, we have 
	$$\bigcup_{t\in[s_{j_0},\xi_\infty]}\allN(\z(t))=\bigcup_{j\geq j_0}\bigcup_{t\in[s_j,s_{j+1})}\allN(\z(t))=\bigcup_{j\geq j_0}\allN(\z(\xi_j))\subseteq\allN(\z(\xi_\infty)).$$
Thus, $\allN(\z(t))\subseteq\allN(\z(\xi_\infty))$ for all $t\in[s_0,\xi_\infty]$. If $\xi_\infty<T$, then \eqref{eq:tj} would imply the contradiction $\xi_\infty\leq\xi_{j_0}$. Therefore, $\xi_\infty=T$.
\end{proof}

The following result describes the behavior of a path that starts at the nonsmooth part of the boundary and satisfies the boundary jitter property. The proof uses a time-reversal argument in conjunction with Lemma \ref{lem:xij}.

\begin{lem}\label{lem:chij}
Given $\z\in\cts([0,\infty):G)$, suppose $\z$ satisfies conditions 3 and 4 of the boundary jitter property (Definition \ref{def:jitter}), $\z(0)\in\U$ and $T\in(0,\infty)$ is such that $\allN(\z(t))\subsetneq\allN(\z(0))$ for all $t\in(0,T)$. Then there is a nested decreasing sequence
	\be\label{eq:chi0u1} T>\chi_0>u_1\geq\chi_1>\cdots>u_j\geq\chi_j>\cdots>0\ee
such that $\chi_0\in G^\circ$, $\chi_j\to0$ as $j\to\infty$ and for each $j\in\N$, $\z(\chi_j)\in\partial G$ and
	\be\label{eq:chijuj}\bigcup_{t\in[\chi_j,u_j]}\allN(\z(t))\subseteq\allN(\z(\chi_j))\qquad\text{and}\qquad\bigcup_{t\in(u_j,\chi_{j-1}]}\allN(\z(t))\subseteq\allN(\z(\chi_{j-1})).\ee
\end{lem}

\begin{proof}
Let $x\in G^\circ$ be arbitrary. Define the path $\widetilde{\z}\in\cts([0,\infty):G)$ to be the time reversal of $\z$ on $[0,T]$ concatenated with the line segment connecting $\z(0)$ and $x$ as follows:
\begin{equation}\label{eq:zreversed}
	\widetilde{\z}(t)\doteq
	\begin{cases}
		\z(T-t)&\text{for }t\in[0,T],\\
		\z(0)e^{T-t}+(x-\z(0))(1-e^{T-t})&\text{for }t\in(T,\infty).
	\end{cases}
\end{equation}
Note that the definition of $\widetilde{\z}(t)$ for $t\in[T,\infty)$ and the fact that $x\in G^\circ$ and $G$ is convex together ensure that $\widetilde{\z}(t)$ remains in $G^\circ$ for all $t\in(T,\infty)$. 

We now show that $\widetilde{\z}$ satisfies condition 4 of the boundary jitter property. Suppose $t\in(0,T]$ is such that $\allN(\widetilde{\z}(t))=\allN(\z(T-t))\in\U$. By condition 3 of the boundary jitter property, for each $i\in\allN(\widetilde{\z}(t))=\allN(\z(T-t))$ and every $\delta\in(0,t)$, there exists $u\in(T-t,T-t+\delta)\subsetneq(T-t,T)$ such that $\allN(\z(u))=\{i\}$, which implies that $\tilde{u}\doteq T-u\in(t-\delta,t)$ and $\allN(\widetilde{\z}(\tilde{u}))=\{i\}$. Since $\widetilde{\z}(t)\in G^\circ$ for all $t>T$, this proves that $\widetilde{\z}$ satisfies condition 4 of the boundary jitter property.

By condition 3 of the boundary jitter property, there exist $0<s<t<T$ and $i,j\in\allN(\z(0))$ such that $i\neq j$, $\allN(\z(s))=\{i\}$ and $\allN(\z(t))=\{j\}$. Then by condition 4 of the boundary jitter property, the continuity of $\z$ and the upper semicontinuity of $\allN(\cdot)$, there exists $S\in(0,T)$ such that $T-S\in(s,t)$ and $\z(T-S)\in G^\circ$. It follows that $\widetilde{\z}(t)=\z(T-t)\subsetneq\allN(\z(0))=\allN(\widetilde{\z}(T))$ for all $t\in[S,T)$. By Lemma \ref{lem:xij}, there is a nested increasing sequence as in \eqref{eq:Sxi0} such that $\xi_j\to T$ as $j\to\infty$ and for each $j\in\N$, $\widetilde{\z}(\xi_j)\in\partial G$ and \eqref{eq:xijsj} holds with $\widetilde{\z}$ in place of $\z$. For each $j\in\N$, set $\chi_j\doteq T-\xi_j$ and $u_j\doteq T-s_j$. It is then a straightforward to verify that these properties along with the definition \eqref{eq:zreversed} of $\widetilde{\z}$ imply that \eqref{eq:chi0u1} holds, $\chi_j\to0$ as $j\to\infty$ and for each $j\in\N$, $\z(\chi_j)=\widetilde{\z}(\xi_j)\in\partial G$ and \eqref{eq:chijuj} holds.
\end{proof}

\section{The derivative problem}\label{sec:dp}

In this section we discuss some useful properties of the DP and the associated DM, which were introduced in Definition \ref{def:dp}. Some of these properties are analogous to properties satisfied by the ESP that are stated in \cite[Section 2.1 and Section 3.1]{Ramanan2006}. Due to the similarity between the axiomatic framework for the DP and the axiomatic framework for the ESP, we are able to leverage arguments used to prove properties of the ESP to prove analogous properties for the DP. Throughout this section, fix an ESP $\{(d_i,n_i,c_i),i\in\allN\}$ and a solution $(\z,\y)$ to the ESP for $\x\in\cts_G$.

\subsection{Basic properties}

In this section we discuss basic properties of solutions to the DP. Our first result states that the DM is linear. The result is a consequence of the fact that, for $x\in G$, $H_{x}$ and $\spaan[d(x)]$ are (closed) linear subspaces of $\R^J$. Since the proof is a straightforward verification argument, we omit it.

\begin{lem}\label{lem:linear}
Suppose $(\phi_1,\eta_1)$ solves the DP associated with $\z$ for $\psi_1\in\dr$ and $(\phi_2,\eta_2)$ solves the DP associated with $\z$ for $\psi_2\in\dr$. Then for all $\alpha,\beta\in\R$, $(\alpha\phi_1+\beta\phi_2,\alpha\eta_1+\beta\eta_2)$ solves the DP associated with $\z$ for $\alpha\psi_1+\beta\psi_2$.
\end{lem}

Given a solution $(\phi,\eta)$ to the DP for $\psi\in\dr$ and $S\in[0,\infty)$, define $\psi^S,\phi^S,\eta^S\in\dr$ by
\begin{align}\label{eq:psiS}
	\psi^S(\cdot)&\doteq \phi(S)+\psi(S+\cdot)-\psi(S),\\ \label{eq:phiS}
	\phi^S(\cdot)&\doteq \phi(S+\cdot),\\\label{eq:etaS}
	\eta^S(\cdot)&\doteq \eta(S+\cdot)-\eta(S).
\end{align}
The following lemma states a useful time-shift property of the DM that is analogous to time-shift property of the ESP stated in Lemma \ref{lem:esmshift}.

\begin{lem}\label{lem:DManticipate}
Suppose $(\phi,\eta)$ solves the DP associated with $\z$ for $\psi\in\dr$. For $S\in[0,\infty)$, define $\z^S$ as in \eqref{eq:zS} and $\psi^S,\phi^S,\eta^S\in\dr$ as in \eqref{eq:psiS}--\eqref{eq:etaS}. Then $(\phi^S,\eta^S)$ solves the DP associated with $\z^S$ for $\psi^S$. Moreover, if $(\phi,\eta)$ is the unique solution to the DP associated with $\z$ for $\psi$, then for any $0\leq S<T<\infty$, $\phi(T)$ depends only on $\{\z(t),t\in[S,T]\}$, $\phi(S)$ and $\{\psi(S+t)-\psi(S),t\in[0,T-S]\}$. 
\end{lem}

\begin{proof}
Fix $S\in[0,\infty)$ and let $\z^S,\psi^S,\phi^S,\eta^S$ be as in the statement of the lemma. Fix $t\in[0,\infty)$. We first show that $(\phi^S,\eta^S)$ satisfies condition 1 of the DP associated with $\z^S$. By \eqref{eq:phiS}, the fact that $(\phi,\eta)$ satisfies condition 1 of the DP associated with $\z$, \eqref{eq:psiS} and \eqref{eq:etaS}, we have for $t\geq0$,
	$$\phi^S(t)=\psi(S+t)+\eta(S+t)+\phi(S)-\psi(S)-\eta(S)=\psi^S(t)+\eta^S(t).$$
Next, we show that $(\phi^S,\eta^S)$ satisfies condition 2 of the DP associated with $\z^S$, which is equivalent, by \eqref{eq:phiS} and \eqref{eq:zS}, to showing that $\ip{\phi(S+t),n_i}=0$ for all $i\in\allN(\z(S+t))$, which follows because $(\phi,\eta)$ satisfies condition 2 of the DP associated with $\z$. We now turn to the proof that $(\phi^S,\eta^S)$ satisfies condition 3 of the DP associated with $\z^S$. Fix $s\in[0,t)$. By \eqref{eq:etaS}, the fact that $(\phi,\eta)$ satisfies condition 3 of the DP associated with $\z$ and \eqref{eq:zS},
	$$\eta^S(t)-\eta^S(s)\in\spaan\lsb\cup_{u\in(S+s,S+t]}d(\z(u))\rsb=\spaan\lsb\cup_{u\in(s,t]}d(\z^S(u))\rsb.$$
This completes the proof of the lemma.
\end{proof}

\subsection{Lipschitz continuity and closure property}\label{sec:desmlip}

In this section we establish a Lipschitz continuity property for solutions of the DP that follows from Assumption \ref{ass:setB}. We also prove a closure property for the DM. We first state a useful consequence of Assumption \ref{ass:setB}.

\begin{lem}
Under Assumption \ref{ass:setB},
	\be\label{eq:zninudi}z\in\partial B,\;\nu\in\nu_B(z)\qquad\Rightarrow\qquad\ip{z,n_i}\ip{\nu,d_i}\leq0\qquad\forall\; i\in\allN(x).\ee
In particular, for all $i\in\allN$,
	\be\label{eq:zni1}z\in\partial B,\;\nu\in\nu_B(z),\;\ip{\nu,d_i}<0\qquad\Rightarrow\qquad\ip{z,n_i}\geq1.\ee
\end{lem}

\begin{proof}
The implication \eqref{eq:zninudi} follows from \cite[Lemma 2.1]{Dupuis1991}, and \eqref{eq:zni1} is a straightforward consequence of \eqref{eq:setB} and \eqref{eq:zninudi}.
\end{proof}

We now state the Lipschitz continuity that follows from Assumption \ref{ass:setB}.

\begin{theorem}\label{thm:dmlip}
Suppose the ESP $\{(d_i,n_i,c_i),i\in\allN\}$ satisfies Assumption \ref{ass:setB}. Then there exists $\lip_\dm\in(0,\infty)$ such that if $(\phi_1,\eta_1)$ solves the DP associated with $\z$ for $\psi_1\in\dr$ and $(\phi_2,\eta_2)$ solves the DP associated with $\z$ for $\psi_2\in\dr$, then for all $T\in[0,\infty)$,
\begin{align}\label{eq:philip}
	\norm{\phi_1-\phi_2}_t&\leq\lip_\dm\norm{\psi_1-\psi_2}_T.
\end{align}
As a consequence, for every $\psi\in\dr$, there is at most one solution to the DP.
\end{theorem}

\begin{proof}
For $f=\psi,\phi,\eta$, define $\Delta f\doteq f_1-f_2$. Fix $T\in[0,\infty)$ and let
	\be\label{eq:cnormpsi}c\doteq\norm{\Delta\psi}_T.\ee
We will show that
	\be\label{eq:DeltaetaB}\eta(t)\in cB\;\text{ for all }t\in[0,T].\ee
Since $B$ is compact and $\phi_j=\psi_j+\eta_j$ for $j=1,2$, \eqref{eq:DeltaetaB} implies the desired Lipschitz continuity result. To show \eqref{eq:DeltaetaB}, it suffices to show that for any $a>c$,
	\be\label{eq:DeltaetaaB}\Delta\eta(t)\in aB\;\text{ for all }t\in[0,T].\ee
Fix $a>c$ and define 
	$$\tau\doteq\inf\{t\geq 0:\Delta\eta(t)\not\in(aB)^\circ\},$$
where $\tau$ is possibly infinite. Then
	\be\label{eq:DeltaetaaBo}\Delta\eta(t)\in(aB)^\circ\;\text{ for all }t\in[0,\tau).\ee
Note that if $\tau>T$ then \eqref{eq:DeltaetaaB} follows. We will now argue by contradiction to show that \eqref{eq:DeltaetaaB} must hold. Suppose \eqref{eq:DeltaetaaB} is false. Then $\tau\leq T$. We consider two mutually exclusive and exhaustive cases. In the following we adopt the convention $\Delta\eta(0-)\doteq0$.

\emph{Case 1:} $\tau\in(0,T]$, $\Delta\eta(\tau-)\in\partial(aB)$. \\
For this case, let $z\doteq\Delta\eta(\tau-)$ and fix $\nu\in\nu(z/a)$. By \eqref{eq:DeltaetaaBo} and the fact that $\nu(z/a)$ is the set of inward normals to $aB$ at $z\in\partial (aB)$, it follows that for all $t\in(0,\tau)$,
	$$\label{eq:zDeltaetat}\ip{z-\Delta\eta(t),\nu}=\ip{\eta_1(\tau-)-\eta_1(t),\nu}-\ip{\eta_2(\tau-)-\eta_2(t),\nu}<0.$$
This implies there exists a sequence $\{t_k\}_{k\in\N}$ with $t_k\uparrow\tau$ such that either
	\be\label{eq:eta1tauminus}\ip{\eta_1(\tau-)-\eta_1(t_k),\nu}<0\;\text{ for all }k\in\N,\ee
or
	\be\label{eq:eta2tauminus}\ip{\eta_2(\tau-)-\eta_2(t_k),\nu}>0\;\text{ for all }k\in\N.\ee
Suppose \eqref{eq:eta1tauminus} holds. Recall the definition of $d(x)$ given in \eqref{eq:dx}. By condition 3 of the DP, there exists $i\in\allN$ and a sequence $\{u_k\}_{k\in\N}$ in $(0,\tau)$ with $u_k\uparrow\tau$ as $k\to\infty$ such that
	$$\ip{d_i,\nu}\neq0\qquad\text{and}\qquad \ip{\z(u_k),n_i}=c_i\qquad\text{for all }k\in\N.$$
Then, since $\nu\in\nu(z/a)$ and $z/a\in\partial B$, it follows from condition 1 of the DP and the geometric property \eqref{eq:setB} of the set $B$ that
	\be\label{eq:ipza}\ip{\frac{z}{a},n_i}=\frac{1}{a}\ip{\Delta\phi(\tau-)-\Delta\psi(\tau-),n_i}\not\in(-1,1).\ee
Condition 2 of the DP and the fact that $\ip{\z(u_k),n_i}=c_i$ imply that $\ip{\Delta\phi(u_k),n_i}=0$ for all $k\in\N$. Taking limits as $k\to\infty$ yields 
	\be\ip{\Delta\phi(\tau-),n_i}=0.\ee 
When combined with \eqref{eq:ipza}, it follows that $|\ip{\Delta\psi(\tau-),n_i}|\geq a$. Since $a>c$ and $\tau\leq T$, this contradicts \eqref{eq:cnormpsi}. Therefore \eqref{eq:eta1tauminus} does not hold. In an analogous fashion, it can be shown that \eqref{eq:eta2tauminus} cannot hold and therefore, Case 1 cannot hold.

\emph{Case 2:} $\tau\in[0,T]$, $\Delta\eta(\tau-)\in(aB)^\circ$ and $\Delta\eta(\tau)\not\in(aB)^\circ$.\\
The proof in this case is similar, with minor differences, but we fill in the details for completeness. For this case, let $z\doteq\Delta\eta(\tau)$. Then there exists $r\geq a$ such that $z\in\partial (rB)$. Fix $\nu\in\nu(z/r)$. Since $\nu(z/r)$ is the set of inward normals to $rB$ at $z\in\partial(rB)$ and $\Delta\eta(\tau-)\in(rB)^\circ$,
	$$\ip{z-\Delta\eta(\tau-),\nu}=\ip{\eta_1(\tau)-\eta_1(\tau-),\nu}-\ip{\eta_2(\tau)-\eta_2(\tau-),\nu}<0.$$
This implies that either
	\be\label{eq:eta1tauminus1}\ip{\eta_1(\tau)-\eta_1(\tau-),\nu}<0\ee
or
	\be\label{eq:eta2tauminus1}\ip{\eta_2(\tau)-\eta_2(\tau-),\nu}>0.\ee
Suppose that \eqref{eq:eta1tauminus1} holds. By \eqref{eq:etatetatminus}, there exists $i\in\{1,\dots,N\}$ such that
	$$\ip{d_i,\nu}\neq0\qquad\text{and}\qquad\ip{\z(\tau),n_i}=c_i.$$
Since $\ip{d_i,\nu}\neq0$, $\nu\in\nu(z/r)$ and $z/r\in\partial B$, it follows from the geometric property \eqref{eq:setB} of the set $B$ that
	\be\label{eq:ipznitau}\ip{\frac{z}{r},n_i}=\frac{1}{r}\ip{\Delta\phi(\tau)-\Delta\psi(\tau),n_i}\not\in(-1,1).\ee
Now by condition 1 of the DP and the fact that $\ip{\z(\tau),n_i}=c_i$, $\ip{\Delta\phi(\tau),n_i}=0$ and so
	\be|\ip{\Delta\psi(\tau),n_i}|\geq r,\ee
which contradicts \eqref{eq:cnormpsi} and the fact that $c<r$. Therefore \eqref{eq:eta1tauminus1} does not hold. In an analogous fashion, it can be shown that \eqref{eq:eta2tauminus1} does not hold and therefore, Case 2 cannot hold.
\end{proof}

The next lemma states a closure property for the DM, which follows from the Lipschitz continuity of the DM stated in Theorem \ref{thm:dmlip} and the fact that for each $x\in G$, $H_{x}$ and $\spaan(d(x))$ are closed subsets of $\R^J$. The closure property is similar to the closure property of the ESM (see \cite[Lemma 2.5]{Ramanan2006}).

\begin{lem}\label{lem:dmclosure}
Suppose the ESP $\{(d_i,n_i,c_i),i\in\allN\}$ satisfies Assumption \ref{ass:setB}. Let $\{\psi_k\}_{k\in\N}$ be a sequence in $\dr$ such that $\psi_k$ converges to $\psi\in\dr$ as $k\to\infty$. Fix $T\in(0,\infty)$. Suppose that for each $k\in\N$, $(\phi_k,\eta_k)$ solves the DP associated with $\z$ for $\psi_k$ on $[0,T)$. Then there exists $(\phi,\eta)\in\dr([0,T):\R^J)\times\dr([0,T):\R^J)$ such that $(\phi_k,\eta_k)\to(\phi,\eta)$ in $\dr([0,T):\R^J)\times\dr([0,T):\R^J)$ as $k\to\infty$ and $(\phi,\eta)$ solves the DP associated with $\z$ for $\psi$.
\end{lem}

\begin{proof}
Fix $T\in(0,\infty)$. By the Lipschitz continuity property stated in Theorem \ref{thm:dmlip}, $\{(\phi_k,\eta_k)\}_{k\in\N}$ is a Cauchy sequence in $\dr([0,T):\R^J)\times\dr([0,T):\R^J)$. Since $\dr([0,T):\R^J)$ is a complete metric space under the topology of uniform convergence (see, e.g., \cite[Chapter 3]{Billingsley1999}), there exists $(\phi,\eta)\in\dr([0,T):\R^J)\times\dr([0,T):\R^J)$ such that $(\phi_k,\eta_k)\to(\phi,\eta)$ in the uniform norm as $k\to\infty$. We are left to show that $(\phi,\eta)$ solves the DP for $\psi$ on $[0,T)$.

Let $t\in[0,T)$. Taking limits as $k\to\infty$ in $\phi_k(t)=\psi_k(t)+\eta_k(t)\in H_{\z(t)}$, we see that $\phi(t)=\psi(t)+\eta(t)\in H_{\z(t)}$, where we have used the fact that $H_{\z(t)}$ is a closed linear subspace. Thus, conditions 1 and 2 of the DP hold. Now let $0\leq s<t<T$. By condition 3 of the DP, for each $k\in\N$, 
	$$\eta_k(t)-\eta_k(s)\in\spaan\lsb\cup_{u\in(s,t]}d(\z(u))\rsb.$$
Since the right-hand side is a closed linear subspace, letting $k\to\infty$ in the above, we see that condition 3 of the DP holds. This completes the proof of the lemma.
\end{proof}

\begin{remark}
The closure property stated in Lemma \ref{lem:dmclosure} requires that $\z$ be fixed. Indeed, the closure property above does not generally hold if $\z$ is replaced by a convergent sequence $\{\z_k\}_{k\in\N}$. For example, let $J=1$ and consider the one-dimensional SP $\{(e_1,e_1,0)\}$. For each $k\in\N$, define $\z_k,\psi_k,\phi_k,\eta_k\in\cts$, by $\z_k(t)\doteq 1/k$, $\psi_k(t)\doteq\phi_k(t)\doteq1$, $\eta_k(t)\doteq0$ for all $t\in[0,\infty)$. It is readily verified that for each $k\in\N$, $(\phi_k,\eta_k)$ solves the DP associated with $\z_k$ for $\psi_k$. Define $\z,\psi,\phi,\eta\in\cts$ by $\z(t)\doteq0$, $\psi(t)\doteq\phi(t)\doteq1$, $\eta(t)\doteq0$ for all $t\in[0,\infty)$. Then $\z_k,\psi_k,\phi_k,\eta_k$ converge to $\z,\psi,\phi,\eta$, respectively, in $\cts$ as $k\to\infty$. However $(\phi,\eta)$ does not solve the (one-dimensional) DP associated with $\z$ for $\psi$.
\end{remark}

\section{Directional derivatives: The one-dimensional setting}\label{sec:nablasm1}

In this section we review prior results on directional directions of the one-dimensional SM and also present new results that relate the directional derivatives to solutions of the DP. Mandelbaum and Massey \cite[Lemma 5.2]{Mandelbaum1995} were the first to establish existence of and obtain an explicit characterization for directional derivatives $\nabla_\psi\sm_1(\x)$ when $\x,\psi\in\cts$, $\x(0)=0$ and the explicit representation for $\nabla_\psi\sm_1(\x)$ has a finite number of discontinuities in any compact interval of $[0,\infty)$.  The last two restrictions were removed and the result generalized  to  $\x, \psi \in \dr$ and $\x, \psi \in \mathcal{D}_{\text{lim}}$, the space of functions on $[0,\infty)$ that have finite left limits on $(0,\infty)$ and finite right limits on $[0,\infty)$, by Whitt \cite[Corollary 9.5.1]{Whitt2002} and Mandelbaum and Ramanan \cite[Theorem 3.2]{Mandelbaum2010}, respectively. In the following proposition, we summarize the results of \cite{Mandelbaum1995,Mandelbaum2010,Whitt2002} when $\x$ and $\psi$ are continuous. 

For $f,g\in\cts$, define $F(f,g):[0,\infty)\to\R$ by
\be\label{eq:F}
	F(f,g)(t)\doteq
	\begin{cases}
		0&\text{if }\sup_{s\in[0,t]}(-f(s))<0,\\
		\sup_{s\in\Phi_{-f}(t)}(-g(s))\vee0&\text{if }\sup_{s\in[0,t]}(-f(s))=0,\\
		\sup_{s\in\Phi_{-f}(t)}(-g(s))&\text{if }\sup_{s\in[0,t]}(-f(s))>0, 
	\end{cases}
\ee
where
	\be\label{eq:Phif}\Phi_{-f}(t)\doteq\lcb u\in[0,t]:-f(u)=\sup_{s\in[0,t]}(-f(s))\rcb.\ee
	
\begin{prop}\label{prop:1dderivative}
Given $\x,\psi\in\cts$, the directional derivative $\nabla_\psi\Gamma_1(\x)$ exists, is upper semicontinuous, lies in $\dlr$ and is given by
	\be\label{eq:nablasm1}\nabla_\psi\sm_1(\x)(t)=\psi(t)+F(\x,\psi)(t),\qquad t\in[0,\infty).\ee
Consequently, $F(\x,\psi)$ is upper semicontinuous and lies in $\dlr$.
\end{prop}

\begin{proof}
By \cite[Theorem 1.1]{Mandelbaum2010} $\nabla_\psi\sm_1(\x)$ exists and is upper semicontinuous. By \eqref{eq:sm1} and \cite[Theorem 3.2]{Mandelbaum2010}, it follows that $\nabla_\psi\sm_1(\x)$ is given by \eqref{eq:nablasm1}. By \cite[Theorem 1.2]{Mandelbaum2010} and because there are no \emph{chains} (see \cite[Definition 1.5]{Mandelbaum2010}) in the one-dimensional setting, $\nabla_\psi\sm_1(x)$ lies in $\dlr$.
\end{proof}

In the next proposition we characterize directional derivatives of $\sm_1$ via solutions of the (one-dimensional) DP when $(\z,\y)$ satisfies condition 1 of the boundary jitter property (conditions 2--4 of the boundary jitter property are automatic in the one-dimensional setting). Recall from Example \ref{ex:1dsp} that $G=\R_+$ and $\pi_1(x)=x\vee0$. It is straightforward to check that $\nabla_v\pi_1(x)$, defined as in \eqref{eq:pixv} for $(x,v)\in\R_+\times\R$, is given by
	\be\label{eq:nablavpix}\nabla_v\pi_1(x)=\begin{cases}v&\text{if }x>0,\\ v\vee0&\text{if }x=0.\end{cases}\ee
In addition, observe that $\S=\partial G=\{0\}$ and $G_0$, defined as in \eqref{eq:Gx} with $x=0$, is given by $G_0=\R_+$; $H_x$, defined as in \eqref{eq:Hx}, is given by
	\be\label{eq:Hx1d}H_x=\begin{cases}\R&\text{if }x>0,\\ \{0\}&\text{if }x=0,\end{cases}\ee
and $d(x)$, defined as in \eqref{eq:dx}, is given by $d(0)=\R_+$ and $d(x)=0$ for all $x>0$.

\begin{prop}\label{prop:1dsmderivative}
Given $\x\in\cts_G$, let $(\z,\y)$ denote the solution of the one-dimensional SP for $\x$. Then for all $\psi\in\cts$, 
\begin{itemize}
	\item[1.] $\nabla_\psi\sm_1(\x)$ exists and lies in $\dlr$; 
	\item[2.] $\nabla_\psi\sm_1(\x)(0)=\nabla_{\psi(0)}\pi_1(\x(0))$ and if $\nabla_\psi\sm_1(\x)$ is discontinuous at $t\in(0,\infty)$, then $\z(t)\in\S$ and $\nabla_\psi\sm_1(\x)$ is left continuous at $t$ if and only if $\nabla_\psi\sm_1(\x)(t-)\in G_{\z(t)}$;
	\item[3.] if $(\z,\y)$ satisfies condition 1 of the boundary jitter property (Definition \ref{def:jitter}), then there is a unique solution $(\phi,\eta)$ of the DP associated with $\z$ for $\psi$, and $\phi(t)=\nabla_\psi\sm_1(\x)(t+)$ for all $t\in[0,\infty)$.
\end{itemize}
\end{prop}

Before proving Proposition \ref{prop:1dsmderivative}, we first prove the following useful lemma.

\begin{lem}\label{lem:F}
Given $f,g\in\cts$, $F(f,g)$ is upper semicontinuous and lies in $\dlr$. Moreover, if $t\in(0,\infty)$ is a discontinuity point of $F(f,g)$, then the following properties hold:
\begin{itemize}
	\item[(i)] $\sm_1(f)(t)=0$;
	\item[(ii)] $F(f,g)$ is left continuous at $t$ if and only if $F(f,g)(t-)\geq-g(t)$.
\end{itemize}
\end{lem}

\begin{proof}
By Proposition \ref{prop:1dderivative}, $F(f,g)$ is upper semicontinuous and lies in $\dlr$. Fix a discontinuity point $t\in(0,\infty)$ of $F(f,g)$. By \eqref{eq:F}--\eqref{eq:Phif}, $t$ lies at an endpoint of the closed interval 
	$$I_f\doteq\{t\in[0,\infty):\sup_{s\in[0,t]}(-f(s))=0\}$$ 
and/or $\sup_{s\in\Phi_{-f}(\cdot)}(-g(s))$ is discontinuous at $t$. In either case, $t\in\Phi_{-f}(t)$, which along with the explicit formula for $\sm_1$ given in \eqref{eq:sm1} implies that $\sm_1(f)(t)=0$, so (i) holds. Suppose $F(f,g)$ is left continuous at $t$. Then by \eqref{eq:F} and the fact that $t\in\Phi_{-f}(t)$, $F(f,g)(t-)=F(f,g)(t)\geq-g(t)$. Alternatively, suppose $F(f,g)(t-)\geq-g(t)$. If $t<a\doteq\min I_f$, then \eqref{eq:F} implies $F(f,g)$ is continuous at $t$. Next, if $t=a$, then $F(f,g)(t-)=0$, $\Phi_{-f}(t)=\{t\}$ and $F(f,g)(t)=(-g(t))\vee0=0$, so $F(f,g)$ is left continuous at $t$. Proceeding, if $t\in I_f\setminus\{a\}$, then \eqref{eq:F} implies
	$$F(f,g)(t-)=\sup_{s\in[0,t)\cap I_f}(-g(s))\vee0,\qquad F(f,g)(t)=\sup_{s\in[0,t]\cap I_f}(-g(s))\vee0,$$
so the fact that $F(f,g)(t-)\geq-g(t)$ implies $F(f,g)$ is left continuous at $t$. Finally, suppose $t>\max I_f$. If $\Phi_{-f}(t)=\{t\}$, then the continuity of $f$ implies there is a sequence $\{s_k\}_{k\in\N}$ such that $s_k\uparrow t$ as $k\to\infty$ and for each $k\in\N$, $s_k>\max I_f$ and $\Phi_{-f}(s_k)=\{s_k\}$. In this case, 
	$$F(f,g)(t-)=\lim_{k\to\infty}F(f,g)(s_k)=\lim_{k\to\infty}(-g(s_k))=-g(t-)=F(f,g)(t).$$ 
On the other hand, if $\Phi_{-f}(t)\neq\{t\}$, then set $\Phi_{-f}(t-)\doteq\Phi_{-f}(t)\setminus\{t\}$. Then for all $u\in[0,t)$ sufficiently large, $\Phi_{-f}(u)=\Phi_{-f}(t-)\cap[0,u]$. Thus, by the continuity of $g$,
\begin{align*}
	F(f,g)(t-)&=\lim_{u\uparrow t}\sup_{s\in\Phi_{-f}(t-)\cap[0,u]}(-g(s))=\sup_{s\in\Phi_{-f}(t-)}(-g(s)),
\end{align*}
and, using that $F(f,g)(t-)\geq-g(t)$, we have 
	$$F(f,g)(t)=\sup_{s\in\Phi_{-f}(t-)}(-g(s))\vee(-g(t))=F(f,g)(t-).$$ 
The completes the proof of (ii).
\end{proof}

\begin{proof}[Proof of Proposition \ref{prop:1dsmderivative}]
Fix $\psi\in\cts$. By \cite[Theorem 3.2]{Mandelbaum2010}, $\nabla_\psi\sm_1(\x)$ exists, lies in $\dlr$ and is characterized by \eqref{eq:nablasm1}--\eqref{eq:Phif}. It follows from \eqref{eq:nablavpix} that $\nabla_\psi\sm_1(\x)(0)=\nabla_{\psi(0)}\pi_1(\x(0))$. Let $t\in(0,\infty)$ be a discontinuity point of $\nabla_\psi\sm_1(\x)$. Then by \eqref{eq:nablasm1}, the continuity of $\psi$ and Lemma \ref{lem:F}, $\z(t)=\sm_1(\x)(t)=0$, or equivalently, $\z(t)\in\S$, and $\nabla_\psi\sm_1(\x)$ is left continuous at $t$ if and only if $\nabla_\psi\sm_1(\x)(t)\geq0$, or equivalently, $\nabla_\psi\sm_1(\x)(t-)\in G_{\z(t)}$. This proves of parts 1 and 2 of Proposition \ref{prop:1dsmderivative}.

Now suppose $(\z,\y)$ satisfies condition 1 of the boundary jitter property. We show that $(\hat{\phi},\hat{\phi}-\psi)$, where $\hat{\phi}(t)\doteq\nabla_\psi\sm_1(\x)(t+)$ for all $t\in[0,\infty)$, solves the DP associated with $\z$ for $\psi$. Since solutions of the DP are unique under Assumption \ref{ass:setB}, by Theorem \ref{thm:dmlip}, this will complete the proof of part 3 of the proposition. Condition 1 of the DP holds automatically. Let $t\in[0,\infty)$. In view of \eqref{eq:Hx1d}, we need to show that if $\z(t)=0$, then $\hat{\phi}(t)=0$. Suppose $\z(t)=0$. By condition 1 of the boundary jitter property and the fact that $\y$ is nondecreasing, \emph{at least} one of the following holds:
\begin{itemize} 
	\item[(i)] $t\geq0$ and $\y(u)>\y(t)$ for all $u>t$;
	\item[(ii)] $t>0$ and $\y(s)<\y(t)$ for all $s<t$.
\end{itemize} 
First consider case (i). Since $\y$ is nondecreasing, there is a sequence $\{u_\ell\}_{\ell\in\N}$ such that $u_\ell\downarrow t$ as $\ell\to\infty$ and for each $\ell\in\N$, $0\leq\y(s)<\y(u_\ell)$ for all $s<u_\ell$. Then for each $\ell\in\N$, \eqref{eq:y} and \eqref{eq:nablasm1}--\eqref{eq:Phif} imply that $\Phi_{-\x}(u_\ell)=\{u_\ell\}$ and $\nabla_\psi\sm_1(\x)(u_\ell)=0$. Letting $\ell\to\infty$ yields $\hat{\phi}(t)=0$. Next, suppose case (ii) holds and case (i) does \emph{not} hold. Then $-\x(t)=\y(t)>0$ and by \eqref{eq:y} and \eqref{eq:nablasm1}--\eqref{eq:Phif}, $\Phi_{-\x}(t)=\{t\}$ and $\nabla_\psi\sm_1(\x)(t)=0$. Since case (i) does not hold, \eqref{eq:Phif} implies $\Phi_{-\x}(u)\subseteq[t,u]$ for all $u\in[t,t+\delta)$ for $\delta>0$ sufficiently small. Upon substituting the last relation into \eqref{eq:F}, we see that $F(\x,\psi)(u)=\sup_{r\in[t,u]}(-\psi(r))$ for all $u\in[t,t+\delta)$. Since $\psi$ is continuous, this implies that $F(\x,\psi)$ is right continuous at $t$, so $\hat{\phi}(t)=0$. This proves that $\hat{\phi}$ satisfies condition 2 of the DP.

We are left to show that $(\hat{\phi},\hat{\phi}-\psi)$ satisfies condition 3 of the DP. By \eqref{eq:nablasm1}, $\hat{\phi}(t)-\psi(t)=F(\x,\psi)(t+)$ for all $t\in[0,\infty)$. Fix $0\leq s<t<\infty$. In order to prove condition 3, due to the fact that $\spaan[d(0)]=\R$, it suffices to show that if $\z$ is positive on $(s,t]$, then $F(\x,\psi)(t+)-F(\x,\psi)(s+)=0$. By the continuity of $\z$, if $\z$ is positive on $(s,t]$, there exists $u>t$ such that $\z$ is positive on $(s,u)$. By \eqref{eq:sm1}--\eqref{eq:y} and \eqref{eq:Phif}, $\Phi_{\x}(\cdot)$ must be constant on $(s,u)$, which, along with \eqref{eq:F}, implies the desired conclusion $F(\x,\psi)(t+)-F(\x,\psi)(s+)=0$.
\end{proof}

\section{Directional derivatives: Up to the first hitting time of the nonsmooth part of the boundary}\label{sec:theta2}

In this section we prove existence of and characterize directional derivatives of the ESM up until the first time that the constrained path reaches $\U$, the nonsmooth part of the boundary. In order to prove our main result, we introduce the following statement, which will be referred to multiple times for different values of $T$. Recall the definition of $G_x$, for $x\in\S$, given in \eqref{eq:Gx} and that $\nabla_v\pi(x)$, defined in \eqref{eq:pixv}, was shown in Lemma \ref{lem:projxv} to exist for all $(x,v)\in G\times\R^J$.

\begin{statement}\label{state:A}
For all $\psi\in\cts$, the following hold:
\begin{itemize}
	\item[1.] $\nabla_\psi\esm(\x)$ exists on $[0,T)$ and lies in $\dlr([0,T):\R^J)$. 
	\item[2.] $\nabla_\psi\esm(\x)(0)=\nabla_{\psi(0)}\pi(\x(0))$ and if $t\in(0,T)$ is a discontinuity point of $\nabla_\psi\esm(\x)$, then $\z(t)\in\S$ and $\nabla_\psi\esm(\x)$ is left continuous at $t$ if and only if $\nabla_\psi\esm(\x)(t-)\in G_{\z(t)}$.
	\item[3.] If $(\z,\y)$ satisfies condition 1 of the boundary jitter property on $[0,T)$, then there exists a unique solution $(\phi,\eta)$ to the DP associated with $\z$ for $\psi$ on $[0,T)$ and $\phi(t)=\nabla_\psi\esm(\x)(t+)$ for $t\in[0,T)$.
\end{itemize}
\end{statement}

Given a solution $(\z,\y)$ to the ESP for $\x\in\cts_G$, let $\theta_2$ be the first time $\z$ reaches the nonsmooth part of the boundary $\U$; that is,
	\be\label{eq:theta2}\theta_2\doteq\inf\{t\in[0,\infty):\z(t)\in\U\}.\ee
The following proposition is the main result of this section.

\begin{prop}\label{prop:theta2}
Fix an ESP $\{(d_i,n_i,c_i),i\in\allN\}$ satisfying Assumption \ref{ass:setB} and Assumption \ref{ass:projection}. Given $\x\in\cts_G$, let $(\z,\y)$ denote the solution to the ESP for $\x$ and define $\theta_2$ as in \eqref{eq:thetan}. Then Statement \ref{state:A} holds with $T=\theta_2$.
\end{prop}

In Section \ref{sec:preliminary} we introduce some standard notation that will be used in this section and present some useful lemmas. In Section \ref{sec:theta2a} we prove Proposition \ref{prop:theta2}. 

\subsection{Preliminary setup and results}\label{sec:preliminary}

In this section we introduce our standard setup and present a useful lemma. Fix an ESP $\{(d_i,n_i,c_i)\}$ satisfying Assumption \ref{ass:setB} and Assumption \ref{ass:projection}. Given $\x\in\cts_G$ and $\psi\in\cts$, let $\z\doteq\esm(\x)$ and for $\ve>0$, let $\z_\ve\doteq\esm(\x+\ve\psi)$. Given $S\in(0,\infty)$, define $\x^S,\z^S$ as in \eqref{eq:xS}--\eqref{eq:xS}, so by the time-shift property of the ESP (Lemma \ref{lem:esmshift}), $\z^S=\esm(\x^S)$ and for $\ve>0$, define $\x_\ve^S,\z_\ve^S\in\cts$ by
\begin{align}\label{eq:xveS}
	\x_\ve^S(\cdot)&=\z_\ve(S)+\x(S+\cdot)+\ve\psi(S+\cdot)-\x(S)-\ve\psi(S)\\ \label{eq:zveS}
	\z_\ve^S(\cdot)&=\z_\ve(S+\cdot).
\end{align}
so by the time-shift property of the ESP, $\z_\ve^S=\esm(\x_\ve^S)$. By \eqref{eq:xveS} and \eqref{eq:xS},
\begin{align}\label{eq:xveSxS}
	\x_\ve^S(\cdot)&=\x^S(\cdot)+\z_\ve(S)-\z(S)+\ve\psi(S+\cdot)-\ve\psi(S)\\ \notag
	&=\x^S(\cdot)+\ve\hat{\psi}_\ve^S(\cdot),
\end{align}
where $\hat{\psi}_\ve^S\in\cts$ is given by
\begin{align}\label{eq:psiveS}
	\hat{\psi}_\ve^S(\cdot)&\doteq\frac{\z_\ve(S)-\z(S)}{\ve}+\psi(S+\cdot)-\psi(S)\\ \notag
	&=\nabla_\psi^\ve\esm(\x)(S)+\psi(S+\cdot)-\psi(S),
\end{align}
where the last equality uses \eqref{eq:nablapsive} and the definitions of $\z_\ve$ and $\z$. Suppose that $T\in(S,\infty)$ and $\nabla_\psi\esm(\x)$ exists on $[0,T)$. Then $\hat{\psi}_\ve^S\to\hat{\psi}^S$ uniformly on $[0,\infty)$ as $\ve\downarrow0$, where $\hat{\psi}^S\in\cts$ is given by
	\be\label{eq:psiSlimit}\hat{\psi}^S(\cdot)\doteq\nabla_\psi\esm(\x)(S)+\psi(S+\cdot)-\psi(S).\ee
The proof of Proposition \ref{prop:theta2} will proceed by showing the existence of and characterizing $\nabla_\psi\esm(\x)$ on intervals where $\z$ only hits a single face. The following lemma will allow us to piece together these results to establish existence on $[0,\theta_2)$.

\begin{lem}\label{lem:stitchA}
Fix $0\leq S<T<U<\infty$. Let $(\z,\y)$ be the solution of the ESP for $\x\in\cts_G$. Assume that Statement \ref{state:A} holds. Define $\x^S,\z^S,\y^S,\hat{\psi}^S$ as in \eqref{eq:xS}--\eqref{eq:yS} and \eqref{eq:psiSlimit}, and assume that Statement \ref{state:A} holds with $\x^S,\z^S,\y^S,\hat{\psi}^S,U-S$ in place of $\x,\z,\y,\psi,T$, respectively. Then Statement \ref{state:A} holds with $U$ in place of $T$. Furthermore, for $t\in[S,U)$, $\nabla_\psi\esm(\x)(t)=\nabla_{\hat{\psi}^S}\esm(\x^S)(t-S)$.
\end{lem}

\begin{proof}
For $\ve>0$, let $\z_\ve\doteq\esm(\x+\ve\psi)$ and define $\x_\ve^S,\z_\ve^S,\hat{\psi}_\ve^S$ as in \eqref{eq:xveS}, \eqref{eq:zveS} and \eqref{eq:psiveS}. Since Statement \ref{state:A} holds by assumption, $\nabla_\psi\esm(\x)$ exists on $[0,T)$, which, along with the fact that $S\in[0,T)$, implies $\hat{\psi}_\ve^S\to\hat{\psi}^S$ uniformly on $[0,\infty)$ as $\ve\downarrow0$. We show that $\nabla_\psi\esm(\x)$ exists on $[T,U)$. By our assumption that Statement \ref{state:A} holds with $\x^S,\hat{\psi}^S,\z^S,\y^S,U-S$ in place of $\x,\psi,\z,\y,T$, respectively, Proposition \ref{prop:psive}, and the fact that $\hat{\psi}_\ve^S\to\hat{\psi}^S$ uniformly on $[0,\infty)$ as $\ve\downarrow0$, we have
	\be\label{eq:nablahatpsive}\lim_{\ve\downarrow0}\nabla_{\hat{\psi}_\ve^S}^\ve\esm(\x^S)(t)=\nabla_{\hat{\psi}^S}\esm(\x^S)(t),\qquad t\in[0,U-S).\ee
By \eqref{eq:SMderivative}, \eqref{eq:nablapsive}, the time-shift property of the ESP (Lemma \ref{lem:esmshift}), \eqref{eq:xveSxS} and \eqref{eq:nablahatpsive}, for $t\in[S,U)$, we have
\begin{align}\label{eq:nablaTU}
	\nabla_\psi\esm(\x)(t)&=\lim_{\ve\downarrow0}\frac{\esm(\x+\ve\psi)-\esm(\x)}{\ve}\\ \notag
	&=\lim_{\ve\downarrow0}\frac{\esm(\x^S+\ve\hat{\psi}_\ve^S)(t-S)-\esm(\x^S)(t-S)}{\ve}=\nabla_{\hat{\psi}^S}\esm(\x^S)(t-S).
\end{align}
This establishes the existence of $\nabla_\psi\esm(\x)$ on $[0,U)$ as well as the final assertion of the lemma.

Proceeding, by assumption, $\nabla_\psi\esm(\x):[0,T)\mapsto\R^J$ lies in $\dlr([0,T):\R^J)$ and $\nabla_{\hat{\psi}^S}\esm(\x^S):[0,U-S)\mapsto\R^J$ lies in $\dlr([0,U-S):\R^J)$. Since $S\in[0,T)$, it follows from \eqref{eq:nablaTU} that $\nabla_\psi\esm(\x):[0,U)\mapsto\R^J$ lies in $\dlr([0,U):\R^J)$. By assumption, if $t\in(0,T)$ is a discontinuity point of $\nabla_\psi\esm(\x)$ then $\z(t)\in\S$ and  $\nabla_\psi\esm(\x)$ is left continuous at $t$ if and only if $\nabla_\psi\esm(\x)(t-)\in G_{\z(t)}$. Now suppose $t\in[T,U)$ is a discontinuity point of $\nabla_\psi\esm(\x)$. Then by \eqref{eq:nablaTU}, $t-S$ is a discontinuity point of $\nabla_{\hat{\psi}^S}\esm(\x^S)$. Since Statement \ref{state:A} holds with $U-S$ in place of $T$ for the time shifted paths and $S<T$, it follows from \eqref{eq:nablaTU} that $\nabla_\psi\esm(\x)(t)=\nabla_{\hat{\psi}^S}\esm(\x^S)(t-S)\in\S$ and $\nabla_\psi\esm(\x)(\cdot)=\nabla_{\hat{\psi}^S}\esm(\x^S)(\cdot-S)$ is left continuous at $t$ if and only if
	$$\nabla_\psi\esm(\x)(t-)=\nabla_{\hat{\psi}^S}\esm(\x^S)((t-S)-)\in G_{\z^S(t-S)}=G_{\z(t)},$$
where the last equality uses \eqref{eq:zS}. This proves that part 2 of Statement \ref{state:A} holds with $U$ in place of $T$.

We are left to show part 3 of Statement \ref{state:A} with $U$ in place of $T$. Suppose $(\z,\y)$ satisfies condition 1 of the boundary jitter property on $[0,U)$. Then it is readily verified (using the relations \eqref{eq:zS}--\eqref{eq:yS}) that $(\z^S,\y^S)$ satisfies condition 1 of the boundary jitter property on $[0,U-S)$. Therefore, by assumption, there exists a unique solution $(\phi,\eta)$ of the DP associated with $\z$ for $\psi$ on $[0,T)$, $\phi(t)=\nabla_\psi\esm(\x)(t+)$ for all $t\in[0,T)$, there exists a unique solution $(\phi^S,\eta^S)$ of the DP associated with $\z^S$ for $\psi^S$ on $[0,U-S)$ and $\phi^S(t)=\nabla_{\hat{\psi}^S}\esm(\x^S)(t+)$ for all $t\in[0,U-S)$. Let $\hat{\phi}(t)\doteq\nabla_\psi\esm(\x)(t+)$ on $[0,U)$. Note that by condition 2 of the DP,
	\be\label{eq:hatphit0T}\hat{\phi}(t)=\phi(t)\in H_{\z(t)},\qquad t\in[0,T),\ee 
and, by \eqref{eq:nablaTU}, condition 2 of the DP and \eqref{eq:zS}, for all $t\in[S,U)$,
\begin{align}\label{eq:hatphitSU}
	\hat{\phi}(t)\doteq\nabla_\psi\esm(\x)(t+)&=\nabla_{\hat{\psi}^S}\esm(\x^S)((t-S)+)=\phi^S(t-S)\in H_{\z^S(t-S)}=H_{\z(t)}.
\end{align}
Let $\hat{\eta}\doteq\hat{\phi}-\psi$ on $[0,U)$. We prove that $(\hat{\phi},\hat{\eta})$ solves the DP associated with $\z$ for $\psi$ on $[0,U)$, which along with the uniqueness of solutions implied by the Lipschitz continuity of the DM (Theorem \ref{thm:dmlip}) will prove part 3 of Statement \ref{state:A}.

Condition 1 of the DP holds automatically. Condition 2 of the DP follows from \eqref{eq:hatphit0T} and \eqref{eq:hatphitSU}. It remains to prove that $\hat{\eta}$ satisfies condition 3 of the DP on $[0,U)$. Let $0\leq s<t<U$. If $0\leq s< t<T$, then by \eqref{eq:hatphit0T}, \eqref{eq:psiSlimit} and because $(\phi,\eta)$ solves the DP for $\psi$ on $[0,T)$, 
	\be\label{eq:hatetathatetas0T}\hat{\eta}(t)-\hat{\eta}(s)=\phi(t)-\psi(t)-(\phi(s)-\psi(s))\in\spaan[\cup_{u\in(s,t]}d(\z(u))].\ee
If $S\leq s<t<U$, then by \eqref{eq:hatphitSU} and because $(\phi^S,\eta^S)$ solves the DP for $\hat{\psi}^S$ on $[0,U-S)$,
\begin{align}\label{eq:hatetathatetasSU}
	\hat{\eta}(t)-\hat{\eta}(s)&\doteq\hat{\phi}^S(t-S)-\hat{\psi}^S(t)-(\hat{\phi}^S(s-S)-\hat{\psi}^S(s))\\ \notag
	&\in\spaan[\cup_{u\in(s-S,t-S]}d(\z^S(u-S))]=\spaan[\cup_{u\in(s,t]}d(\z(u))]
\end{align}
That leaves the remaining case when $0\leq s<S<T\leq t< U$. By \eqref{eq:hatetathatetas0T}--\eqref{eq:hatetathatetasSU},
\begin{align*}
	\hat{\eta}(t)-\hat{\eta}(s)&=\hat{\eta}(t)-\hat{\eta}(S)+\hat{\eta}(S)-\hat{\eta}(s)\\
	&\in\spaan\lsb\lcb\cup_{u\in(S,t]}d(\z(u))\rcb\cup\lcb\cup_{u\in(s,S]}d(\z(u))\rcb\rsb\\
	&=\spaan\lsb\cup_{u\in(s,t]}d(\z(u))\rsb,
\end{align*}
which completes the proof of condition 3.
\end{proof}

\subsection{Existence and characterization}\label{sec:theta2a}

In this section we prove Proposition \ref{prop:theta2}. If $\theta_2=0$ the lemma is trivial. Alternatively, if $\z(t)\in G^\circ$ for all $t\in[0,\infty)$, then it is readily verified that for all $\psi\in\cts$, $\nabla_\psi\esm(\x)=\psi$ and $(\psi,0)$ solves the DP associated with $\z$ for $\psi$. For the remainder of this section we assume that $\theta_2\in(0,\infty]$ and $\z(t)\in\partial G$ for some $t\in[0,\infty)$. 

Recursively define the increasing sequence $\{t_k\}_{k=1,\dots,K}$, $K\in\N_\infty$, in $[0,\theta_2)$ as follows: first, set 
	\be\label{eq:t1}t_1\doteq\inf\{t\in[0,\theta_2):\z(t)\in\partial G\}.\ee
Since $\z(t)\in\partial G$ for some $t\in[0,\infty)$ by assumption, $t_1<\infty$. Given $k\in\N$ for which $t_k$ is defined, if $t_k=\theta_2$, set $K=k$, whereas if $t_k<\theta_2$, then $\z(t_k)$ lies in the relative interior of some $(J-1)$-dimensional face $F_{i_k}$ of $\partial G$ and we recursively define $t_{k+1}$ to be the first time after $t_k$ that $\z$ hits $\cup_{j\neq i_k}F_j$; that is,
	\be\label{eq:tk}t_{k+1}\doteq\inf\{t\in(t_k,\theta_2]:\z(t)\in\partial G,\;\allN(\z(t))\neq\allN(\z(t_k))\}.\ee
If $t_k<\theta_2$ for all $k\in\N$, set $K=\infty$. In other words,
	\be\label{eq:K}K\doteq\inf\{k\in\N:t_k=\theta_2\}.\ee
If $K=\infty$, then \eqref{eq:tk} and the continuity of $\z$ imply that $t_k\to\theta_2$ as $k\to\infty$.  

If $t_1=\theta_2$, then $\z(t)\in G^\circ$ for all $t\in[0,\theta_2)$ and it is straightforward to prove that for all $\psi\in\cts$, $\nabla_\psi\esm(\x)=\psi$ and $(\psi,0)$ solves the DP associated with $\z$ for $\psi$. We assume that $t_1\in[0,\theta_2)$. Using induction, we prove that for $2\leq k<K+1$, Statement \ref{state:A} holds with $T=t_k$. Since $t_K=\theta_2$ if $K<\infty$ and $t_k\to\theta_2$ if $K=\infty$, Proposition \ref{prop:theta2} will then follow. In the next lemma, we establish the base case of the induction hypothesis. Since $\Z(t)$ lies on at most one face of $G$ for $t\in[0,t_2)$, we are able to reduce the problem to the one-dimensional setting and invoke Proposition \ref{prop:1dsmderivative}.

\begin{lem}\label{lem:t2}
Given a solution $(\z,\y)$ to the ESP for $\x\in\cts_G$, define $t_2$ as in \eqref{eq:tk}. Then Statement \ref{state:A} holds with $T=t_2$.
\end{lem}

\begin{proof}
Fix $\psi\in\cts$. It suffices to prove that Statement \ref{state:A} holds for all $T\in(t_1,t_2)$. Fix $T\in(t_1,t_2)$. Let $i\in\allN$ denote the unique index such that $\z(t_1)\in F_i$, or equivalently, $\allN(\z(t_1))=\{i\}$. Define $f,g\in\cts([0,\infty):\R)$ by
\begin{align*}
	f(t)&\doteq\ip{\x(t),n_i}-c_i,&&t\in[0,\infty),\\
	g(t)&\doteq\ip{\psi(t),n_i},&&t\in[0,\infty).
\end{align*}
Since $\x(0)\in G$, $f(0)\geq0$ holds. According to \eqref{eq:t1}--\eqref{eq:tk}, $\z(t)\not\in\cup_{j\neq i} F_j$ for all $t\in[0,T]$, so by condition 3 of the ESP and the normalization $\ip{d_i,n_i}=1$,
	\be\label{eq:ynd}\y(t)=\ip{\y(t),n_i}d_i,\qquad t\in[0,T],\ee
and $\ip{\y(\cdot),n_i}$ is nondecreasing on $[0,T]$ and can only increase when $\ip{\z(t),n_i}=c_i$. When combined with conditions 1 and 2 of the ESP, it is readily verified that 
	\be\label{eq:zn}\ip{\z(t),n_i}-c_i=\sm_1(f)(t),\qquad t\in[0,T).\ee
As usual, let $\z_\ve\doteq\esm(\x+\ve\psi)$. By \eqref{eq:Zlip}, \eqref{eq:allNusc} and the fact that $\z$ and $\psi$ are bounded on $[0,T]$, we have, for $\ve>0$ sufficiently small, $\z_\ve(t)\not\in\cup_{j\neq i} F_j$ for all $t\in[0,T]$. For such $\ve>0$, we can follow an argument analogous to the one above to obtain that 
	\be\label{eq:yvend}\y_\ve(t)=\ip{\y_\ve(t),n_i}d_i,\qquad t\in[0,T],\ee
and, with $f$ and $g$ defined as above,
	\be\label{eq:zven}\ip{\z_\ve(t),n_i}-c_i=\sm_1(f+\ve g)(t),\qquad t\in[0,T).\ee
Therefore, by Proposition \ref{prop:1dderivative}, for $t\in[0,T)$,
\begin{align}\label{eq:limzvez}
	\ip{\nabla_\psi\esm(\x)(t),n_i}=\nabla_g\sm_1(f)(t),
\end{align}
which along with condition 1 of the ESP implies
\begin{align}\label{eq:limyvey}
	\lim_{\ve\downarrow0}\frac{\ip{\y_\ve(t),n_i}-\ip{\y(t),n_i}}{\ve}&=F(f,g)(t),
\end{align}
where $F(f,g)$ is the function defined in \eqref{eq:F}. By condition 1 of the ESP, \eqref{eq:ynd}, \eqref{eq:yvend} and \eqref{eq:limyvey}, for all $t\in[0,T)$,
	\be\label{eq:derivativet2}\nabla_\psi\esm(\x)(t)=\psi(t)+\lim_{\ve\downarrow0}\frac{\y_\ve(t)-\y(t)}{\ve}=\psi(t)+F(f,g)(t)d_i.\ee
By Lemma \ref{lem:F}, $F(f,g)$ lies in $\dlr([0,T):\R)$, and so, $\psi$ being continuous, $\nabla_\psi\esm(\x)$ also lies in $\dlr([0,T):\R^J)$. This proves part 1 of Statement \ref{state:A}.

The fact that $\nabla_\psi\esm(\x)(0)$ exists and is equal to $\nabla_{\psi(0)}\pi(\x(0))$ follows from Theorem \ref{thm:speu} and Lemma \ref{lem:projxv}. Now suppose $t\in(0,T)$ is a discontinuity point of $\nabla_\psi\esm(\x)$. Then \eqref{eq:derivativet2} and the continuity of $\psi$ imply that $F(f,g)$ is discontinuous. By part (i) of Lemma \ref{lem:F} and \eqref{eq:zn}, we have $\ip{\z(t),n_i}=c_i$, and since $t<\theta_2$, $\z(t)\in F_i\cap\S$. By \eqref{eq:Gx}, $G_{\z(t)}=\{x\in\R^J:\ip{x,n_i}\geq0\}$. By \eqref{eq:limzvez} and Proposition \ref{prop:1dsmderivative}, $\ip{\nabla_\psi\esm(\x)(\cdot),n_i}$ is left continuous at $t$ if and only if $\ip{\nabla_\psi\esm(\x)(t-),n_i}\geq0$, or equivalently, $\nabla_\psi\esm(\x)(t-)\in G_{\z(t)}$. This proves part 2 of Statement \ref{state:A}.

Proceeding, suppose $(\z,\y)$ satisfies condition 1 of the boundary jitter property. Then, using \eqref{eq:zn}, it is readily verified that $(\sm_1(f),\sm_1(f)-f)$ satisfies condition 1 of the (one-dimensional) boundary jitter property on $[0,T)$. Define $\hat{\phi}\in\dr([0,T):\R^J)$ by $\hat{\phi}(t)\doteq\nabla_\psi\esm(\x)(t+)$ for $t\in[0,T)$. We show that $(\hat{\phi},\hat{\phi}-\psi)$ solves the DP for $\psi$ on $[0,T)$. Condition 1 of the DP holds automatically. By \eqref{eq:limzvez}, for $t\in[0,T)$,
	\be\label{eq:ht}h(t)\doteq\ip{\hat{\phi}(t),n_i}=\nabla_g\sm_1(f)(t+).\ee
Then by Proposition \ref{prop:1dsmderivative}, $(h,h-g)$ solves the (one-dimensional) DP associated with $\sm_1(f)$ for $g$ on $[0,T)$. Let $t\in[0,T)$. If $\z(t)\in G^\circ$, then $H_{\z(t)}=\R^J$ so $\hat{\phi}(t)\in H_{\z(t)}$ clearly holds. Alternatively, if $\z(t)\in F_i$, then by \eqref{eq:Hx} and \eqref{eq:zn}, $H_{\z(t)}=\{y\in\R^J:\ip{y,n_i}=0\}$ and $\sm_1(f)(t)=0$. This combined with condition 2 of the (one-dimensional) DP, implies that $h(t)=0$; or equivalently, by \eqref{eq:ht}, $\hat{\phi}(t)\in H_{\z(t)}$. This proves that $\hat{\phi}$ satisfies condition 2 of the DP. Now let $0\leq s<t<T$. By the definition of $\hat{\phi}$, \eqref{eq:derivativet2} and the continuity of $\psi$, 
	\be\label{eq:phipsiFF}\hat{\phi}(t)-\psi(t)-(\hat{\phi}(s)-\psi(s))=(F(f,g)(t+)-F(f,g)(s+))d_i\in\spaan(d_i).\ee
By \eqref{eq:Hx} and the fact that $\allN(\z(u))\subseteq\{i\}$ for all $u\in(s,t]$, we have,
	$$\spaan\lsb\cup_{u\in(s,t]}d(\z(u))\rsb=\begin{cases}\spaan(d_i)&\text{if }\cup_{u\in(s,t]}\allN(\z(u))=\{i\},\\ \{0\}&\text{if }\cup_{u\in(s,t]}\allN(\z(u))=\emptyset.\end{cases}$$
In view of \eqref{eq:phipsiFF}, we are left to show that if $\cup_{u\in(s,t]}\allN(\z(u))=\emptyset$, or equivalently, $\z(u)\in G^\circ$ for all $u\in(s,t]$, then $F(f,g)(t+)-F(f,g)(s+)=0$. Suppose $\z(u)\in G^\circ$ for all $u\in(s,t]$. By \eqref{eq:zn}, $\sm_1(f)(u)>0$ for all $u\in(s,t]$. Then, by \eqref{eq:ht}, \eqref{eq:nablasm1} and because $(h,h-g)$ satisfies condition 3 of the (one-dimensional) DP associated with $\sm_1(f)$ for $g$,
\begin{align*}
	F(f,g)(t+)-F(f,g)(s+)&=h(t)-g(t)-(h(s)-g(s))=0.
\end{align*}
This completes the proof that $(\hat{\phi},\hat{\phi}-\psi)$ solves the DP associated with $\z$ for $\psi$. Thus, part 3 of Statement \ref{state:A} holds.
\end{proof}

In the following lemma we establish the induction step. The proof relies on performing a certain time shift and then applying Lemma \ref{lem:t2}.

\begin{lem}\label{lem:tk}
Given a solution $(\z,\y)$ to the ESP for $\x\in\cts_G$, define $\{t_k\}_{k=1,\dots,K}$, $K\in\N_\infty$, as in \eqref{eq:tk}--\eqref{eq:K}. Let $2\leq k<K$. Assume that Statement \ref{state:A} holds with $T=t_k$. Then Statement \ref{state:A} holds with $T=t_{k+1}$.
\end{lem}

\begin{proof}
Fix $2\leq k<K$ and $\psi\in\cts$. By the continuity of $\z$, the upper semicontinuity of $\allN(\cdot)$ (Lemma \ref{lem:allNusc}) and \eqref{eq:tk}, we can choose $S\in[0,t_k)$ such that $\z(t)\in G^\circ$ for all $t\in[S,t_k)$. Define $\x^S,\z^S,\hat{\psi}^S$ as in \eqref{eq:xS}--\eqref{eq:zS} and \eqref{eq:psiSlimit}. Define $\theta_2^S,t_1^S,t_2^S$ as in \eqref{eq:theta2}, \eqref{eq:t1}--\eqref{eq:tk}, but with $\z^S,\theta_2^S,t_1^S,t_2^S$ in place of $\z,\theta_2,t_1,t_2$, respectively. Since $\z(t)\in G^\circ$ for all $t\in[S,t_k)$ and $S<\theta_2$, we have $t_1^S=t_k-S$ and $t_2^S=t_{k+1}-S$. By the time-shift property of the ESP, $\z^S=\esm(\x^S)$. Thus Lemma \ref{lem:t2} implies that Statement \ref{state:A} holds with $\x^S,\z^S,\hat{\psi}^S,t_2^S$ in place of $\x,\z,\psi,T$, respectively. Therefore, by Lemma \ref{lem:stitchA} (with $t_k$ and $t_{k+1}$ in place of $T$ and $U$, respectively), Statement \ref{state:A} holds with $T=t_{k+1}$.
\end{proof}

\begin{proof}[Proof of Proposition \ref{prop:theta2}]
By Lemma \ref{lem:t2}, Lemma \ref{lem:tk} and the principle of mathematical induction, for $k=1,\dots,K$, Statement \ref{state:A} holds with $T=t_k$. Since either $K<\infty$ and $t_K=\theta_2$, or $K=\infty$ and $t_k\to\theta_2$ as $k\to\infty$, Statement \ref{state:A} holds with $T=\theta_2$.
\end{proof}

\section{Derivative projection operators}\label{sec:derivativeprojection}

In Section \ref{sec:psiconst} we show that when $\x\in\cts_G$ satisfies the boundary jitter property, in order to characterize directional derivatives of the ESM at $\x$, is suffices to only consider perturbations $\psi$ that are constant in neighborhoods of times that $\z$ lies in $\U$, the nonsmooth part of the boundary (see Lemma \ref{lem:CZdense} and the proof of Theorem \ref{thm:main}). The study of directional derivatives for such perturbations $\psi$ is largely governed by properties of a family of (oblique) projection operators, which we introduce in this section. 

Fix an ESP $\{(d_i,n_i,c_i),i\in\allN\}$ satisfying Assumption \ref{ass:setB}. \emph{For the remainder of this section fix a compact, convex, symmetric set $B$ with $0\in B^\circ$ satisfying \eqref{eq:setB} as in Assumption \ref{ass:setB}}. A useful interpretation of $B$ is in terms of an associated norm on $\R^J$, denoted $\norm{\cdot}_{B}$, defined as follows:
\begin{equation}\label{eq:normB}
	\norm{y}_{B}\doteq\min\{r\geq0:y\in rB\},\qquad y\in\R^J.
\end{equation}
For more on this norm, as well as an in-depth discussion of the set $B$, see \cite[Section 2]{Dupuis1999a}. We will write $(\R^J,\norm{\cdot}_B)$ to denote $\R^J$ equipped with the norm $\norm{\cdot}_B$. 

Recall that $\mathcal{W}\doteq\{x\in\U:\spaan(H_{x}\cup d(x))\neq\R^J\}$. In the following lemma we identify a certain decomposition of $\R^J$ that is associated with each $x\in \partial G\setminus\W$. 

\begin{lem}\label{lem:Wdecomp}
For each $x\in\partial G$, $H_{x}\cap\spaan[d(x)]=\{0\}$. In addition, if $x\in \partial G\setminus\W$, then for each $y\in\R^J$, there exist a unique pair of vectors $v_y\in H_{x}$ and $w_y\in\spaan[d(x)]$ such that $y=v_y+w_y$.
\end{lem}

\begin{proof}
Suppose $H_x\neq\{0\}$ and $z\in H_{x}\setminus\{0\}$, so that $\ip{z,n_i}=0$ for $i\in\allN(x)$. Since $0\in B^\circ$, there exist $r>0$ such that $z\in\partial(rB)$ and $\nu\in\nu_{rB}(z)$ such that $\ip{z,\nu}<0$. Then \eqref{eq:setB} and Remark \ref{rmk:setB} imply that $\ip{\nu,d_i}=0$ for all $i\in\allN(x)$, so $z\not\in\spaan[d(x)]$. Thus, $H_{x}\cap\spaan[d(x)]=\{0\}$. The last assertion of the lemma follows because $x\in \partial G\setminus\W$ by assumption, so $\spaan(H_{x}\cup d(x))=\R^J$.
\end{proof}

For the following lemma, given a linear subspace $A$ of $\R^J$, let $\text{dim}(A)$ denote the dimension of $A$.

\begin{lem}\label{lem:Wempty}
Under Assumption \ref{ass:linearlyind}, $\W=\emptyset$.
\end{lem}

\begin{proof}
Let $x\in\U$. Since $H_x$, defined in \eqref{eq:Hx}, is equal to the intersection of $|\allN(x)|$ $(J-1)$-dimensional hyperplanes, it follows that $\text{dim}(H_x)\geq J-|\allN(x)|$. Under Assumption \ref{ass:linearlyind}, it follows from \eqref{eq:dx} that $\text{dim}(\spaan[d(x)])=|\allN(x)|$. By Lemma \ref{lem:Wdecomp}, $H_x\cap\spaan[d(x)]=\{0\}$, from which it follows that $\text{dim}(\spaan[H_x\cup d(x)])=\text{dim}(H_x)+\text{dim}(\spaan[d(x)])\geq J$, so $x\not\in\W$.
\end{proof}

\subsection{Derivative projection operator and its adjoint operator}\label{sec:derivativeprojectionoperator}

In the following lemma we associate with each $x\in \partial G\setminus\W$ a linear operator $\proj_x$, which we refer to as the \emph{derivative projection operator at $x$}. The derivative projection operator $\proj_x$ projects points in $\R^J$ to the linear subspace $H_{x}$, defined in \eqref{eq:Hx}, along a direction that lies in the span of $d(x)$, defined in \eqref{eq:dx}. 

\begin{lem}\label{lem:projx}
For each $x\in \partial G\setminus\W$, there exists a unique operator
	\be\proj_x:(\R^J,\norm{\cdot}_B)\mapsto(\R^J,\norm{\cdot}_B)\ee 
such that for each $y\in\R^J$,
	\be\label{eq:Hxdx}\proj_x y\in H_{x}\qquad\text{and}\qquad\proj_x y-y\in\spaan[d(x)].\ee
Furthermore, $\proj_x$ is linear and its operator norm, denoted $\norm{\proj_x}$, satisfies
	\be\norm{\proj_x}\doteq\sup_{y\neq 0}\frac{\norm{\proj_xy}_B}{\norm{y}_B}\leq1.\ee
In other words, the derivative projection operator $\proj_x$ is a contraction on $(\R^J,\norm{\cdot}_B)$ that maps $B$ into $B\cap H_{x}$.
\end{lem}

\begin{remark}
The derivative projection operator $\proj_x$ is characterized by $H_{x}$ and $d(x)$, which depend on the sets $\{n_i,i\in\allN(x)\}$ and $\{d_i,i\in\allN(x)\}$, respectively. Therefore, given $x,\tilde{x}\in\partial G\setminus\W$ such that $\allN(x)=\allN(\tilde{\x})$, then the projection operators $\proj_x$ and $\proj_{\tilde{x}}$ are equal.
\end{remark}

\begin{proof}
Fix $x\in \partial G\setminus\W$. Given $y\in\R^J$, Lemma \ref{lem:Wdecomp} implies that there exist unique $v_y\in H_{x}$ and $w_y\in\spaan[d(x)]$ such that $y=v_y+w_y$. Set $\proj_xy\doteq v_y$ for all $y\in\R^J$. The uniqueness of the decomposition $y=v_y+w_y$ implies that $\proj_x$ is well defined on $\R^J$. This proves \eqref{eq:Hxdx}. To verify that $\proj_x$ is linear, let $\alpha,\beta\in\R$ and $y,z\in\R^J$. Since $H_{x}$ and $\spaan[d(x)]$ are linear subspaces of $\R^J$, we have 
	$$\alpha\proj_xy+\beta\proj_xz\in H_{x}\qquad\text{and}\qquad\alpha(\proj_xy-y)+\beta(\proj_xz-z)\in \spaan[d(x)].$$ 
The uniqueness of $\proj_x$ then establishes the linear relation $\proj_x(\alpha y+\beta z)=\alpha\proj_xy+\beta\proj_xz$.

We now prove that $\norm{\proj_x}\leq1$. By the linearity of $\proj_x$, it suffices to show that given $y\in\partial B$, $\proj_xy$ lies in $B$. Fix $y\in\partial B$ and set
	\be\label{eq:c}c\doteq\norm{\proj_xy}_{B}=\min\{r\geq0:\proj_xy\in rB\}.\ee
We need to show that $c\leq1$. If $c=0$, then this inequality is automatically satisfied. For the remainder of the proof assume that $c>0$. We claim, and prove below, that the set
	\be\label{eq:supphyper}\lcb\proj_xy+\sum_{i\in\allN(x)}r_id_i,r_i\in\R\rcb\ee
lies in a supporting hyperplane to $cB$ at $\proj_xy$. Since \eqref{eq:Hxdx} implies $y$ lies in the set \eqref{eq:supphyper}, $y$ lies in a supporting hyperplane to $cB$. In particular, $y\not\in(cB)^\circ$ and so $\norm{y}_{B}\geq c$.

To prove the claim, first note that by \eqref{eq:Hxdx} and \eqref{eq:c}, 
	\be\label{eq:HxpartialcB}\proj_xy\in H_{x}\cap\partial(cB).\ee 
Thus, the definition \eqref{eq:Hx} of $H_{x}$ implies that $\ip{\proj_xy,n_i}=0$ for all $i\in\allN(x)$. Since \eqref{eq:setB} holds for the set $cB$, we have 
	\be\label{eq:dinupixy}\ip{d_i,\nu}=0\qquad\text{for all }\nu\in\nu_{cB}(\proj_xy),\;i\in\allN(x).\ee
For a proof by contradiction, suppose the set \eqref{eq:supphyper} does not lie in a supporting hyperplane to $cB$ at $\proj_xy$. Since the set \eqref{eq:supphyper} is an affine subspace and $cB$ is a convex set, there exist $r_i\in\R$, $i\in\allN(x)$, such that 
	$$\proj_xy+\sum_{i\in\allN(x)}r_id_i\in(cB)^\circ.$$
The above display and \eqref{eq:HxpartialcB} together imply that $\sum_{i\in\allN(x)}r_i\ip{d_i,\nu}>0$ for some inward normal $\nu\in\nu_{cB}(\proj_xy)$, which contradicts \eqref{eq:dinupixy}. Therefore, the claim must hold.
\end{proof}

Consider the dual closed convex set $B^\ast$, introduced in this context in \cite{Dupuis1999a}, defined as
	\be\label{eq:Bast}B^\ast\doteq\lcb y\in\R^J:\sup_{z\in B}\ip{y,z}\leq1\rcb.\ee
Then $B^\ast$ is compact, convex, symmetric with $0\in(B^\ast)^\circ$ (see, e.g., \cite[Section 3.2]{Dupuis1999a}), so, analogous to \eqref{eq:normB}, $B^\ast$ defines a norm $\norm{\cdot}_{B^\ast}$ on $\R^J$ as follows
	\be\label{eq:normBast}\norm{y}_{B^\ast}\doteq\min\{r\geq0:y\in rB^\ast\},\qquad y\in\R^J.\ee
Let $\proj_x^\ast$ denote the linear operator that is adjoint to the derivative projection operator $\proj_x$ defined in Lemma \ref{lem:projx}; that is, $\ip{\proj_xy,z}=\ip{y,\proj_x^\ast z}$ for every $y,z\in\R^J$. In the lemma below, we summarize some important properties of $\proj_x^\ast$. Recall that given a subset $A\subset\R^J$, we define $A^\perp\doteq\{y\in\R^J:\ip{x,y}=0\;\forall\;x\in A\}$.

\begin{lem}\label{lem:projxast}
For each $x\in \partial G\setminus\W$, 
	\be\label{eq:projxast}\proj_x^\ast:(\R^J,\norm{\cdot}_{B^\ast})\mapsto(\R^J,\norm{\cdot}_{B^\ast})\ee
is the unique linear operator such that for each $y\in\R^J$,
	\be\label{eq:dxperpHxperp}\proj_x^\ast y\in\spaan[d(x)]^\perp\qquad\text{and}\qquad\proj_x^\ast y-y\in H_{x}^\perp.\ee
Furthermore, the operator norm of $\proj_x^\ast$, denoted $\norm{\proj_x^\ast}$, satisfies
	\be\label{eq:Lastnorm}\norm{\proj_x^\ast}\doteq\sup_{y\neq 0}\frac{\norm{\proj_x^\ast y}_B}{\norm{y}_B}\leq1.\ee
In other words, the adjoint derivative projection operator $\proj_x^\ast$ is a contraction on $(\R^J,\norm{\cdot}_{B^\ast})$ that maps $B^\ast$ into $B^\ast\cap \spaan[d(x)]^\perp$.
\end{lem}

\begin{proof}
Fix $x\in \partial G\setminus\W$. Since $d_i\in\spaan[d(x)]$ for each $i\in\allN(x)$ and the derivative projection operator at $x$ is uniquely defined, \eqref{eq:Hxdx} implies $\proj_xd_i=0$ for each $i\in\allN(x)$. Thus, for all $y\in\R^J$ and $i\in\allN(x)$,
	$$\ip{\proj_x^\ast y,d_i}=\ip{y,\proj_xd_i}=0,$$
which implies that $\proj_x^\ast y\in\spaan[d(x)]^\perp$. Similarly, by \eqref{eq:Hxdx} and the uniqueness of $\proj_x$, $\proj_xz=z$ for all $z\in H_{x}$. Hence, given $y\in\R^J$ and $z\in H_{x}$, we have 
	$$\ip{\proj_x^\ast y-y,z}=\ip{y,\proj_xz-z}=0.$$
Since this holds for all $y\in\R^J$ and $z\in H_{x}$, it follows that $\proj_x^\ast y-y\in H_{x}^\perp$. 

We now establish that $\proj_x^\ast$ is the unique linear operator satisfying \eqref{eq:dxperpHxperp}. Suppose $A:\R^J\mapsto\R^J$ is a linear operator that satisfies \eqref{eq:dxperpHxperp} for all $y\in\R^J$, but with $A$ in place of $\proj_x^\ast$. Then one can readily verify that the adjoint of $A$, denoted $A^\ast$, satisfies \eqref{eq:Hxdx} for all $y\in\R^J$, but with $A^\ast$ in place of $\proj_x$. By the uniqueness of $\proj_x$ shown in Lemma \ref{lem:projx}, it must hold that $A^\ast=\proj_x$. Since the adjoint of $\proj_x$ is uniquely defined, this implies that $A=\proj_x^\ast$. Therefore, $\proj_x^\ast$ is well defined.

We are left to show that $\norm{\proj_x^\ast}_{B^\ast}\leq1$. By the linearity of $\proj_x^\ast$, it suffices to show that $\proj_x^\ast y\in B^\ast$ for all $y\in\partial B^\ast$. Let $y\in \partial B^\ast$. Then
\begin{align*}
	\sup_{z\in B}\ip{\proj_x^\ast y,z}&=\sup_{z\in B}\ip{y,\proj_xz}\leq\sup_{z\in B}\ip{y,z}=1,
\end{align*}
where the inequality follows since $\proj_x$ maps $B$ into $B$, as shown in Lemma \ref{lem:projx}, and the last equality follows from \eqref{eq:Bast} and because $y\in\partial B^\ast$. Again recalling \eqref{eq:Bast}, we see that $\proj_x^\ast y\in B^\ast$. Since $x\in \partial G\setminus\W$ was arbitrary, the proof is complete.
\end{proof}

\subsection{Contraction properties}\label{sec:contraction}

In this section we prove some key contraction properties for sequences of derivative projection operators.

\begin{lem}\label{lem:Bastprop}
For each $x\in \partial G\setminus\W$ and $y\in\R^J$,
	\be\label{eq:projxasty}\proj_x^\ast y=y\qquad\text{if }y\in\spaan[d(x)]^\perp\ee
and
	\be\label{eq:setBast}\norm{\proj_x^\ast y}_{B^\ast}<\norm{y}_{B^\ast}\qquad\text{if }y\not\in\spaan[d(x)]^\perp.\ee
\end{lem}

\begin{proof}
Fix $x\in \partial G\setminus\W$. By the linearity of $\proj_x^\ast$, the fact that $(B^\ast)^\circ$ is nonempty and the definition of $\norm{\cdot}_{B^\ast}$, it suffices to show that \eqref{eq:projxasty} and \eqref{eq:setBast} hold for all $y\in\partial B^\ast$. Let $y\in\partial B^\ast$. Suppose $y\in\spaan[d(x)]^\perp$. By \eqref{eq:dxperpHxperp} and the uniqueness of $\proj_x^\ast$ shown in Lemma \ref{lem:projxast}, $\proj_x^\ast y=y$, so \eqref{eq:projxasty} is satisfied. On the other hand, suppose $y\not\in\spaan[d(x)]^\perp$. We show that $\proj_x^\ast y\in(B^\ast)^\circ$, so \eqref{eq:setBast} holds. For a proof by contradiction, suppose that $\proj_x^\ast y\in\partial B^\ast$. Then, in view of \eqref{eq:Bast} and using the compactness of $B$, there exists $\tilde{z}\in\partial B$ such that 
	$$\ip{\proj_x\tilde{z},y}=\ip{\tilde{z},\proj_x^\ast y}=1.$$
By Lemma \ref{lem:projx}, $\proj_x\tilde{z}\in B$. This, together with the above display, the fact that $y\in\partial B^\ast$ and \eqref{eq:Bast}, implies that $\proj_x\tilde{z}\in\partial B$. Furthermore, given $z\in B$,
	$$\ip{-y,\proj_x\tilde{z}-z}=-1+\ip{y,z}\leq0,$$ 
so $-y\in\nu_{B}(\proj_x\tilde{z})$. Since $\ip{\proj_x\tilde{z},n_i}=0$ for all $i\in\allN(x)$, $\proj_x\tilde{z}\in\partial B$ and $-y\in\nu_{B}(\proj_x\tilde{z})$, \eqref{eq:setB} implies that $\ip{-y,d_i}=0$ for all $i\in\allN(x)$. However, this contradicts the fact that $y\not\in\spaan[d(x)]^\perp$, so \eqref{eq:setBast} must hold.
\end{proof}

\begin{lem}\label{lem:projxKastzero}
There exists $\delta\in[0,1)$ such that given $x\in \partial G\setminus\W$ and a finite sequence $\{x_k\}_{k=1,\dots,K}$ in $\partial G\setminus\W$ such that 
	\be\label{eq:allNcup1K}\allN(x)=\cup_{k=1,\dots,K}\allN(x_k),\ee 
the following inequality holds for all $y\in H_{x}^\perp$:
	\be\label{eq:normBdelta}\norm{\proj_{x_1}^\ast\cdots\proj_{x_K}^\ast y}_{B^\ast}\leq\delta\norm{y}_{B^\ast}.\ee
Consequently, given $x\in \partial G\setminus\W$ and a sequence $\{x_k\}_{k\in\N}$ in $\partial G\setminus\W$ such that $\allN(x)=\cup_{k\geq\ell}\allN(x_k)$ for all $\ell\in\N$, it follows that for any $y\in H_{x}^\perp$,
	\be\label{eq:normBlimit}\lim_{K\to\infty}\lsb\proj_{x_1}^\ast\cdots\proj_{x_K}^\ast y\rsb=0.\ee
\end{lem}

\begin{proof}
We claim, and prove below, that for each $x\in\partial G\setminus\W$, there exists $\delta_x\in[0,1)$, depending only on the set $\allN(x)$, such that given $K\in\N$ and a finite sequence $\{x_k\}_{k=1,\dots,K}$ in $\partial G\setminus\W$ satisfying \eqref{eq:allNcup1K}, then for all $y\in H_{x}^\perp$,
	\be\label{eq:projx1Kasty}\norm{\proj_{x_1}^\ast\cdots\proj_{x_K}^\ast y}_{B^\ast}\leq\delta_x\norm{y}_{B^\ast}.\ee
Then, due to the fact that $\delta_x$ depends only on $\allN(x)$ and there are only a finite number of distinct subsets of $\allN$, we have $\delta\doteq\sup_{x\in\partial G\setminus\W}\delta_x\in[0,1)$, which will complete the proof of \eqref{eq:normBdelta}.

We are left to establish the claim. Fix $x\in \partial G\setminus\W$. Let $y\in H_{x}^\perp\cap\partial B^\ast$. The property $\spaan(H_{x}\cup d(x))=\R^J$ implies that $H_{x}^\perp\cap\spaan[d(x)]^\perp=\{0\}$, so $y\not\in\spaan[d(x)]^\perp$. By Lemma \ref{lem:Bastprop}, given $\tilde{x}\in\partial G\setminus\W$ such that $y\not\in\spaan[d(\tilde{x})]^\perp$, we have $\norm{\proj_{\tilde{x}}^\ast y}_{B^\ast}<1$. Define
	\be\label{eq:deltaxy}\delta_{x,y}\doteq\sup\lcb\norm{\proj_{\tilde{x}}^\ast y}_{B^\ast}:\tilde{x}\in\partial G\setminus\W,\allN(\tilde{x})\subseteq\allN(x),y\not\in\spaan[d(\tilde{x})]^\perp\rcb.\ee
Since $\proj_{\tilde{x}}^\ast$ depends only on $\allN(\tilde{x})$ and there are only a finite number of distinct subsets of $\allN(x)$, the above supremum is in fact over a finite number of elements, so $\delta_{x,y}\in[0,1)$. In addition, $\delta_{x,y}$ depends only on $\allN(x)$ and $y$. Now let $K\in\N$ and $\{x_k\}_{k=1,\dots,K}$ be a finite sequence in $\partial G\setminus\W$ such that \eqref{eq:allNcup1K} holds. Observe that \eqref{eq:allNcup1K} implies $\spaan[d(x)]^\perp=\cap_{k=1,\dots,K}\spaan[d(x_k)]^\perp$, and since $y\not\in\spaan[d(x)]^\perp$, this ensures that $y\not\in\spaan[d(x_k)]^\perp$ for some $k\in\{1,\dots,K\}$. Let $\tilde{k}\doteq\max\{1\leq k\leq K:y\not\in\spaan[d(x_k)]^\perp\}$. Then by Lemma \ref{lem:Bastprop}, \eqref{eq:allNcup1K}, \eqref{eq:deltaxy} and the nonexpansive property of the operators $\proj_{x_k}^\ast$ shown in Lemma \ref{lem:projxast},
\begin{align*}
	\norm{\proj_{x_1}^\ast\cdots\proj_{x_K}^\ast y}_{B^\ast}&=\norm{\proj_{x_1}^\ast\cdots\proj_{x_{\tilde{k}}}^\ast y}_{B^\ast}\leq\delta_{x,y}.
\end{align*}
Since $y\mapsto\norm{\proj_{x_1}^\ast\cdots\proj_{x_K}^\ast y}_{B^\ast}$ is continuous and $H_{x}^\perp\cap\partial B^\ast$ is compact, we have
	$$\delta_x\doteq\sup_{y\in H_{x}^\perp\cap\partial B^\ast}\delta_{x,y}\in[0,1),$$
and $\delta_x$ depends only on $\allN(x)$. Finally, by the linearity of the operators $\proj_{x_k}^\ast$, $1\leq k\leq K$, \eqref{eq:projx1Kasty} holds for any $y\in H_{x}^\perp$. This proves the claim.
\end{proof}

\begin{lem}\label{lem:accumulate}
Suppose $x\in \partial G\setminus\W$ and $\{x_k\}_{k\in\N}$ is a sequence in $\partial G\setminus\W$ such that
	\be\label{eq:allNxk}\allN(x)=\cup_{k\geq\ell}\allN(x_k)\qquad\text{for all }\ell\in\N.\ee
Then for all $y\in\R^J$,
	\be\label{eq:limKprod}\lim_{k\to\infty}\lsb\proj_{x_k}\cdots\proj_{x_1}y\rsb=\proj_xy\ee
and
	\be\label{eq:limKprod1}\lim_{k\to\infty}\lsb\proj_{x_1}\cdots\proj_{x_k}y\rsb=\proj_xy.\ee
\end{lem}

\begin{proof}
Let $y\in\R^J$. We first prove \eqref{eq:limKprod}. By the nonexpansive property of $\proj_{x_k}$ stated in Lemma \ref{lem:projx}, there exists a subsequence $\{k_m\}_{m\in\N}$ and $\bar{y}\in\R^J$ such that $\bar{y}=\lim_{m\to\infty}[\proj_{x_{k_m}}\cdots\proj_{x_1}y]$. Since $\proj_xy$ does not depend on the subsequence, it suffices to establish that $\bar{y}=\proj_xy$.

For each $k\in\N$ and $\tilde{y}\in\R^J$, \eqref{eq:Hxdx} and \eqref{eq:allNxk} together imply $\proj_{x_k}\tilde{y}-\tilde{y}\in\spaan[d(x_k)]\subseteq\spaan[d(x)]$. Consequently,
\begin{align}\label{eq:baryydx}
	\bar{y}-y&=\lim_{m\to\infty}\lsb\proj_{x_{k_m}}\cdots\proj_{x_1}y\rsb-y\\ \notag
	&=\lim_{m\to\infty}\sum_{r=1}^{k_m}\lcb\proj_{x_r}\lsb\proj_{x_{r-1}}\cdots\proj_{x_1}y\rsb-\lsb\proj_{x_{r-1}}\cdots\proj_{x_1}y\rsb\rcb\\ \notag
	&\in \spaan[d(x)],
\end{align}
where we adopt the convention $\proj_{x_0}\cdots\proj_{x_1}y\doteq y$ and we have used the fact that $\spaan[d(x)]$ is a closed linear subspace. In addition, for any $z\in H_{x}^\perp$,
\begin{align*}
	|\ip{\bar{y},z}|=\lim_{m\to\infty}|\ip{\proj_{x_{k_m}}\cdots\proj_{x_1}y,z}|=\lim_{m\to\infty}|\ip{y,\proj_{x_1}^\ast\cdots\proj_{x_{k_m}}^\ast z}|=0,
\end{align*}
where the final equality is a consequence of \eqref{eq:normBlimit}. Therefore $\bar{y}\in H_{x}$. Combined with \eqref{eq:Hxdx} and \eqref{eq:baryydx}, this shows that $\bar{y}-\proj_xy\in H_{x}\cap\spaan[d(x)]$. Since Lemma \ref{lem:Wdecomp} implies $H_{x}\cap\spaan[d(x)]=\{0\}$, it follows that $\bar{y}=\proj_xy$, which completes the proof of \eqref{eq:limKprod}. The proof of \eqref{eq:limKprod1} is analogous, so we omit it. 
\end{proof}

The next result is a corollary to the previous lemma and states that the convergence described above holds uniformly for $y$ in compact subsets of $\R^J$.

\begin{cor}\label{lem:accumulateuniformly}
Suppose $x\in \partial G\setminus\W$ and $\{x_k\}_{k\in\N}$ is a sequence in $\partial G\setminus\W$ such that \eqref{eq:allNxk} holds. Then given any compact set $C\subseteq\R^J$,
	\be\label{eq:limKproduniform}\lim_{k\to\infty}\sup_{y\in C}\left|\lsb\proj_{x_k}\cdots\proj_{x_1}y\rsb-\proj_xy\right|=0\ee
and
	\be\label{eq:limKproduniform1}\lim_{k\to\infty}\sup_{y\in C}\left|\lsb\proj_{x_1}\cdots\proj_{x_k}y\rsb-\proj_xy\right|=0.\ee
\end{cor}

\begin{proof}
Fix a compact subset $C\subseteq\R^J$. We prove that \eqref{eq:limKproduniform} holds, with the proof of \eqref{eq:limKproduniform1} being analogous. It suffices to show that given any sequence $\{y_k\}_{k\in\N}$ in $C$,
	\be\label{eq:limsupLxk}\limsup_{k\to\infty}\left|\lsb\proj_{x_1}\cdots\proj_{x_k}y_k\rsb-\proj_xy_k\right|=0.\ee
Fix a sequence $\{y_k\}_{k\in\N}$ in $C$. By compactness, we can assume, by possibly taking a subsequence, that $\hat{y}\doteq\lim_{k\to\infty}y_k$ exists in $C$. Then, by the triangle inequality,
\begin{align*}
	\norm{\lsb\proj_{x_k}\cdots\proj_{x_1}y_k\rsb-\proj_xy_k}_B&\leq\norm{\proj_{x_k}\cdots\proj_{x_1}(y_k-\hat{y})}_B+\norm{\lsb\proj_{x_k}\cdots\proj_{x_1}\hat{y}\rsb-\proj_x\hat{y}}_B\\
	&\qquad+\norm{\proj_x(\hat{y}-y_k)}_B.
\end{align*}
Letting $k\to\infty$ in the above display and using the nonexpansive property of $\proj_{x_k}\cdots\proj_{x_1}$, the definition of $\hat{y}$ and \eqref{eq:limKprod} of Lemma \ref{lem:accumulate}, which is applicable since \eqref{eq:allNxk} holds, we see that each term on the right-hand side converges to zero. Since the convergence holds for any given subsequence, \eqref{eq:limsupLxk} holds.
\end{proof}

\subsection{Relation to solutions of the DP}\label{sec:projdp}

The main result of this section relates derivative projection operators to solutions of the DP associated with $\z$ at times $T\in(0,\infty)$ such that $\z(T)\in\U$, provided that $\z$ satisfies condition 4 of the boundary jitter property.

\begin{lem}\label{lem:projphirhok}
Let $(\z,\y)$ be a solution of the ESP for $\x\in\cts_G$ and define $\tau$ as in \eqref{eq:tau}. Suppose $\z$ satisfies condition 4 of the boundary jitter property on $[0,\tau)$ and $0\leq S<T<\tau$ are such that $\z(S)\in G^\circ$, $\z(T)\in\U$ and $\allN(\z(t))\subsetneq\allN(\z(T))$ for all $t\in[S,T)$. Suppose $\psi\in\cts$ is constant on $[S,T]$ and $(\phi,\eta)$ solves the DP associated with $\z$ for $\psi$ on $[0,T)$. Then $\phi(T-)$ exists and
	\be\label{eq:phixi0minus}\phi(T-)=\proj_{\z(T)}\phi(S).\ee
\end{lem}

\begin{figure}[h!]
	\centering
	\begin{subfigure}{.5\textwidth}
		\centering
		\includegraphics[width=.7\textwidth]{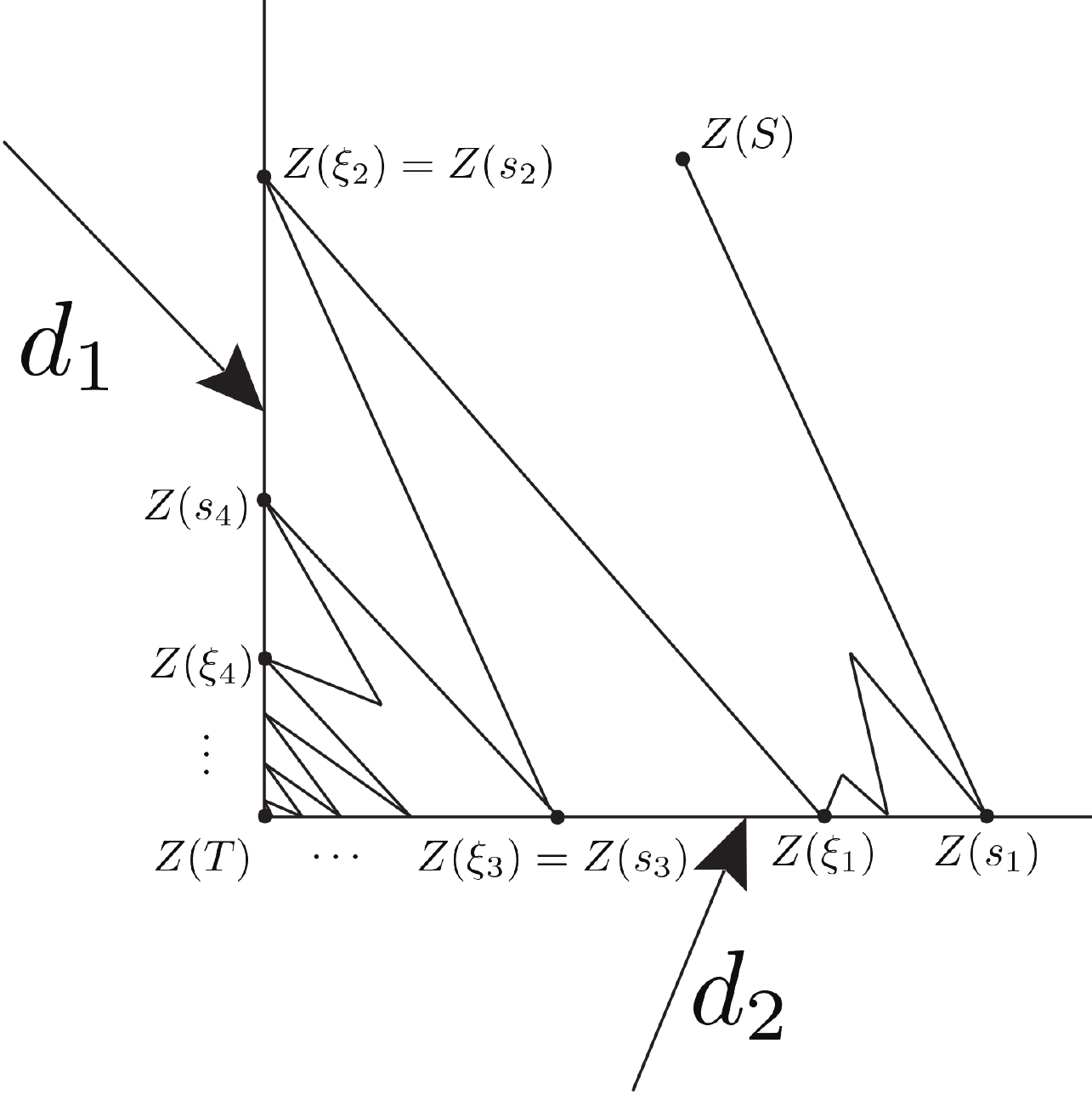}
		\caption{A constrained path $\z$.}
		\label{fig:jitter}
	\end{subfigure}%
	\begin{subfigure}{.5\textwidth}
		\centering
		\includegraphics[width=.7\textwidth]{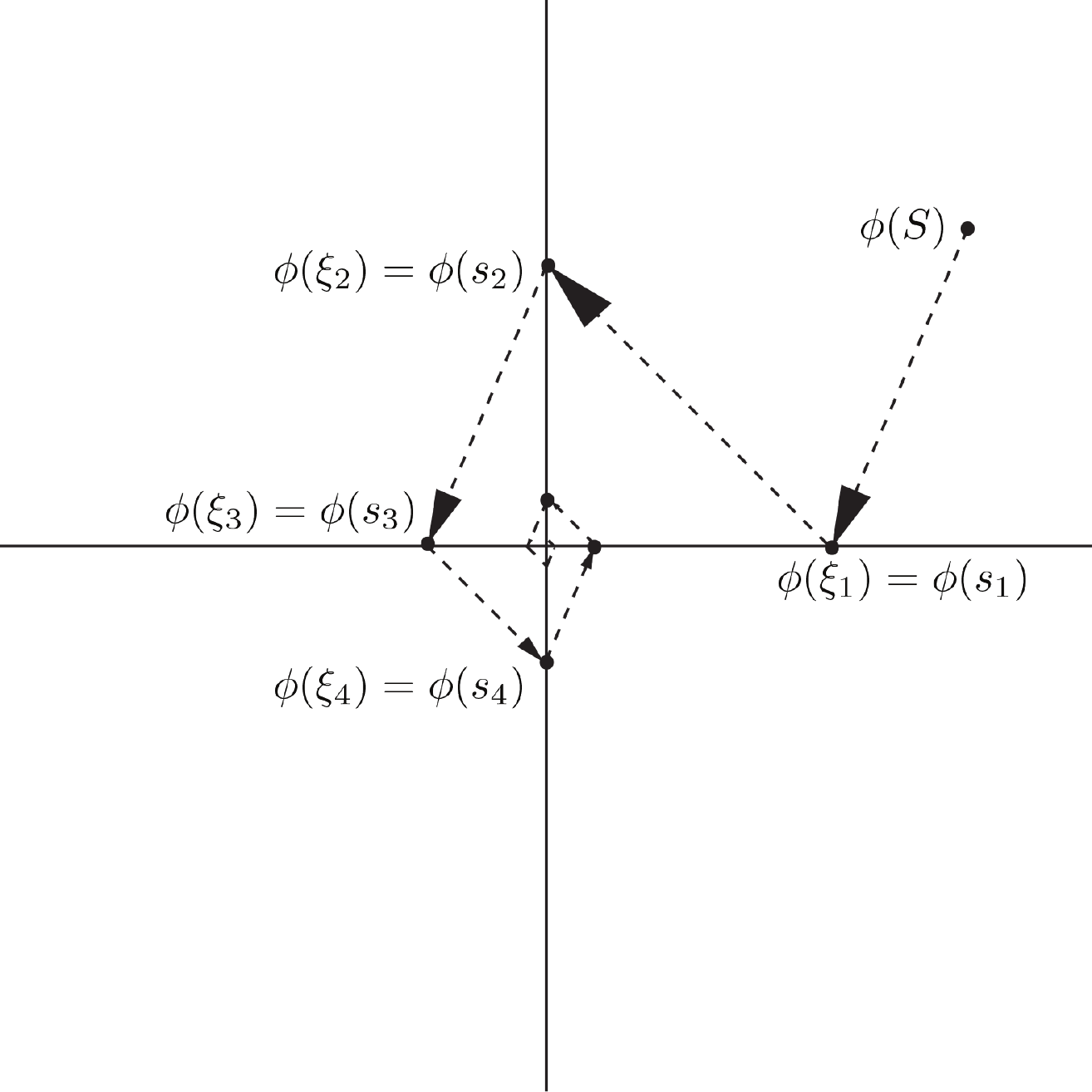}
		\caption{A path $\phi$.}
		\label{fig:dp}
	\end{subfigure}
	\caption{(A). A path $\z$ that satisfies condition 4 of the boundary jitter property on the interval $[S,T]$ and $\z(T)=0$. The times depicted correspond to the nested sequence \eqref{eq:Sxi0} defined in Lemma \ref{lem:xij}. (B). A solution $\phi$ of the DP associated with $\z$ for $\psi$ that is constant on $[S,T]$. Here $\phi$ is piecewise constant with jumps at times $s_1<s_2<\dots$ and $\phi(T-)=\proj_{\z(T)}\phi(S)=0$.}
	\label{fig:test}
\end{figure}

\begin{proof}
By Lemma \ref{lem:xij}, there is a nested sequence $S=\xi_0<s_1\leq\xi_1<\cdots<s_j\leq\xi_j<\cdots<T$ such that $\xi_j\to T$ as $j\to\infty$ and for each $j\in\N$, $\z(\xi_j)\in\partial G$ and \eqref{eq:xijsj} holds. We first prove that for each $j\in\N$,
	\be\label{eq:tjproj}\phi(\xi_j)=\proj_{\z(\xi_j)}\phi(\xi_{j-1}).\ee
See Figure \ref{fig:dp} for an illustration of $\phi$ when $\z$ satisfies condition 4 of the boundary jitter property. Fix $j\in\N$. By condition 1 of the DP and the fact that $\psi$ is constant on $[S,T]$, we have
	\be\label{eq:phisj1xij}\phi(t)=\phi(\xi_{j-1})+\eta(t)-\eta(\xi_{j-1}),\qquad t\in[\xi_{j-1},s_j),\ee
and
	\be\label{eq:phixijsj}\phi(t)=\phi(s_j-)+\eta(t)-\eta(s_j-),\qquad t\in[s_j,\xi_j].\ee
We first show that $(\phi,\eta)$ is constant on $[\xi_{j-1},s_j)$. By \eqref{eq:xijsj} and condition 2 of the DP, $\phi(\xi_{j-1})\in H_{\z(\xi_{j-1})}\subseteq H_{\z(t)}$ for all $t\in[\xi_{j-1},s_j)$. Thus, conditions 1--3 of the DP are satisfied with $\phi(t)=\phi(\xi_{j-1})$ and $\eta(t)=\eta(\xi_{j-1})$ for $t\in[\xi_{j-1},s_j)$. By uniqueness of solutions to the DP, $(\phi,\eta)$ must be constant on $[\xi_{j-1},s_j)$. Next, by \eqref{eq:etatetas}, \eqref{eq:etatetatminus}, the fact that $\eta$ is constant on $[\xi_{j-1},s_j)$ and \eqref{eq:xijsj}, we have
\begin{align}\label{eq:etaxijetaxij1a}
	\eta(\xi_j)-\eta(\xi_{j-1})&=\eta(\xi_j)-\eta(s_j-)\\ \notag
	&\in\spaan\lsb\cup_{u\in[s_j,\xi_j]}d(\z(u))\rsb\subseteq\spaan[d(\z(\xi_j))].
\end{align}
Combined with the fact that $\psi$ is constant on $[\xi_{j-1},\xi_j]$, this implies $\phi(\xi_j)-\phi(\xi_{j-1})\in\spaan[d(\z(\xi_j))]$. Moreover, by condition 2 of the DP, $\phi(\xi_j)\in H_{\z(\xi_j)}$. Relation \eqref{eq:tjproj} then follows from the characterization \eqref{eq:Hxdx} of $\proj_{\z(\xi_j)}$ established in Lemma \ref{lem:projx}.

Iterating the recursion relation \eqref{eq:tjproj} yields
	\be\label{eq:phixij}\phi(\xi_j)=\lsb\proj_{\z(\xi_j)}\cdots\proj_{\z(\xi_1)}\rsb\phi(\xi_0),\qquad j\in\N.\ee
By condition 4 of the boundary jitter property and the fact that $\allN(\z(t))\subsetneq\allN(\z(T))$ for all $t\in[S,T)$, we have $\cup_{t\in[\xi_j,T)}\allN(\z(t))=\allN(\z(T))$ for all $j\in\N$. Then by \eqref{eq:xijsj} and the fact that $\xi_j\to T$ as $j\to\infty$, we have
	$$\cup_{j\geq\ell}\allN(\z(\xi_j))=\cup_{t\in[\xi_\ell,T)}\allN(\z(t))=\allN(\z(T)),\qquad\ell\in\N.$$
Since $T<\tau$, Lemma \ref{lem:xij} implies that $\z(\xi_j)\in\partial G\setminus\W$ for all $j\in\N$. Therefore, Lemma \ref{lem:accumulate} with $x=\z(T)$ and $x_k=\Z(\xi_k)$ for $k\in\N$, and \eqref{eq:phixij} imply that
	\be\label{eq:limphixi0}\lim_{j\to\infty}\phi(\xi_j)=\proj_{\z(T)}\phi(\xi_0).\ee
To see that $\phi(T-)=\proj_{\z(T)}\phi(\xi_0)$, it suffices to show that 
	$$\lim_{j\to\infty}\sup_{u\in[s_j,\xi_j]}|\phi(u)-\phi(\xi_j)|=\lim_{j\to\infty}\sup_{u\in[s_j,s_{j+1})}|\phi(u)-\phi(\xi_j)|=0,$$
where the first equality follows because $\phi$ is constant on $[\xi_j,s_{j+1})$. 

Fix $j\in\N$ and define $\z^{s_j},\psi^{s_j},\phi^{s_j},\eta^{s_j}$ as in \eqref{eq:zS} and \eqref{eq:psiS}--\eqref{eq:etaS}, but with $s_j$ in place of $S$, so by Lemma \ref{lem:DManticipate}, $(\phi^{s_j},\eta^{s_j})$ solves the DP associated with $\z^{s_j}$ for $\psi^{s_j}$. By condition 1 of the DP, \eqref{eq:xijsj}, \eqref{eq:Hx} and \eqref{eq:zS}, $\phi(\xi_j)\in H_{\z(\xi_j)}\subseteq H_{\z(s_j+u)}=H_{\z^{s_j}(u)}$ for all $u\in[0,\xi_j-s_j]$. It is then readily verified that $(\tilde{\phi},0)$, where $\tilde{\phi}\equiv\phi(\xi_j)$, solves the DP associated with $\z^{s_j}$ for $\tilde{\psi}\equiv\phi(\xi_j)$ on $[0,\xi_j-{s_j}]$. By \eqref{eq:phiS}, the Lipschitz continuity of the DM (Theorem \ref{thm:dmlip}), \eqref{eq:psiS} and the fact that $\psi$ is constant on $[S,T]$, we have
\begin{align}\label{eq:phiuphixij}
	\sup_{u\in[s_j,\xi_j]}|\phi(u)-\phi(\xi_j)|&=\norm{\phi^{s_j}-\tilde{\phi}}_{\xi_j-s_j}\\ \notag
	&\leq\lip_\dm\norm{\psi^{s_j}-\tilde{\psi}}_{\xi_j-{s_j}}=\lip_\dm|\phi(s_j)-\phi(\xi_j)|.
\end{align}
By condition 2 of the DP, $\phi(s_j)\in H_{\z(s_j)}$. By conditions 1 and 3 of the DP, the facts that $\phi$ is constant on $[\xi_{j-1},s_j)$ and $\psi$ is constant on $[S,T]$, \eqref{eq:etatetatminus} and \eqref{eq:xijsj}, we have 
	$$\phi(s_j)-\phi(\xi_{j-1})=\phi(s_j)-\phi(s_j-)=\eta(s_j)-\eta(s_j-)\in\spaan[d(\z(s_j))].$$
It then follows from the characterization \eqref{eq:Hxdx} of $\proj_{\z(s_j)}$ established in Lemma \ref{lem:projx} that $\phi(s_j)=\proj_{\z(s_j)}\phi(\xi_{j-1})$. By condition 2 of the DP, \eqref{eq:Hx} and \eqref{eq:xijsj}, $\phi(\xi_j)\in H_{\z(\xi_j)}\subseteq H_{\z(s_j)}$. Again invoking the characterization \eqref{eq:Hxdx} of $\proj_{\z(s_j)}$ established in Lemma \ref{lem:projx}, we have $\phi(\xi_j)=\proj_{\z(s_j)}\phi(\xi_j)$. Combining these relations with \eqref{eq:phiuphixij} yields
\begin{align*}
	\sup_{u\in[s_j,\xi_j]}|\phi(u)-\phi(\xi_j)|\leq\lip_\dm|\proj_{\z(s_j)}[\phi(\xi_{j-1})-\phi(\xi_j)]|.
\end{align*}
Taking limits as $j\to\infty$, we see that the right-hand side converges to zero because of \eqref{eq:limphixi0} and the fact that for each $x\in\partial G\setminus\W$, $\proj_x$ is a linear operator. This completes the proof.
\end{proof}

\section{Directional derivatives: Proof of the main result}\label{sec:tau}

In this section we prove Theorem \ref{thm:main}, which is our main result. Fix an ESP $\{(d_i,n_i,c_i),i\in\allN\}$ satisfying Assumption \ref{ass:setB} and Assumption \ref{ass:projection}.

\subsection{Directional derivatives along a dense subset of paths}\label{sec:psiconst}

Given a solution $(\z,\y)$ of the ESP for $\x\in\cts_G$, define $\tau$ as in \eqref{eq:tau} and a subset of functions in $\cts$ that are constant in neighborhoods of times $t\in[0,\infty)$ that $\z(t)$ lies in $\U$, the nonsmooth part of the boundary. Specifically, for $\delta>0$, define
\begin{align}\label{eq:CZdelta}
	\cts^{\delta,\z}\doteq\lcb\psi\in\cts:\begin{array}{c}\forall\;t\in[0,\infty),\z(t)\in\U\Rightarrow \\ \psi\text{ is constant on }[(t-\delta)\vee0,t+\delta]\end{array}\rcb.
\end{align}
Set 
	$$\cts^\z\doteq\bigcup_{\delta>0}\cts^{\delta,\z}.$$ 
In the next lemma we provide sufficient conditions for $\cts^\z$ to be dense in $\cts$. The proof of Lemma \ref{lem:CZdense} is given in Appendix \ref{apdx:CZdense}.

\begin{lem}\label{lem:CZdense}
Suppose $\z$ satisfies condition 2 of the boundary jitter property (Definition \ref{def:jitter}). Then $\cts^\z$ is dense in $\cts$.
\end{lem}

We now introduce the following statement, which will be called upon repeatedly throughout this section for different values of $T$. Recall the definition of $G_x$, $x\in\S$, given in \eqref{eq:Gx}.

\begin{statement}\label{state:main}
For all $\psi\in\cts^{\delta,\z}$,
\begin{itemize}
	\item[1.] $\nabla_\psi\esm(\x)$ exists on $[0,T)$ and lies in $\dlr([0,T):\R^J)$.
	\item[2.] $\nabla_\psi\esm(\x)(0)=\nabla_{\psi(0)}\pi(\x(0))$ and if $t\in(0,T)$ is a discontinuity point of $\nabla_\psi\esm(\x)$, then $\z(t)\in\S$ and $\nabla_\psi\esm(\x)$ is left continuous at $t$ if and only if $\nabla_\psi\esm(\x)(t-)\in G_{\z(t)}$.
	\item[3.] There exists a unique solution $(\phi,\eta)$ to the DP associated with $\z$ for $\psi$ and $\phi(t)=\nabla_\psi\esm(\x)(t+)$ for $t\in[0,T)$.
\end{itemize}
\end{statement}

We have the following proposition.

\begin{prop}\label{prop:mainalldelta}
Given $\x\in\cts_G$, let $(\z,\y)$ denote the solution of the ESP for $\x$ and define $\tau$ as in \eqref{eq:tau}. Suppose $(\z,\y)$ satisfies the boundary jitter property on $[0,\tau)$. Then for each $\delta>0$, Statement \ref{state:main} holds with $T=\tau$.
\end{prop}

The proof of Proposition \ref{prop:mainalldelta} proceeds as follows. Given a solution $(\z,\y)$ of the ESP for $\x\in\cts_G$, define $\tau$ as in \eqref{eq:tau} and $\theta_2$ as in \eqref{eq:theta2}. If $N\geq 3$, then for each $n=3,\dots,N$, let $\theta_n\in[0,\tau]$ be the first time in the interval $[0,\tau)$ that $\z$ reaches the intersection of $n$ or more faces; that is,
	\be\label{eq:thetan}\theta_n\doteq\inf\{t\in[0,\tau):|\allN(\z(t))|\geq n\}\wedge\tau,\qquad n=2,\dots,N.\ee
Set $\theta_{N+1}\doteq\tau$. Observe that this definition of $\theta_2$ is consistent with \eqref{eq:theta2}. Using a proof by induction, we prove the following statement for $n=2,\dots,N+1$.

\begin{statement}\label{state:thetan}
Given $\x\in\cts_G$, let $(\z,\y)$ denote the solution of the ESP for $\x$ and define $\tau$ as in \eqref{eq:tau} and $\{\theta_n\}_{n=2,\dots,N+1}$, as in \eqref{eq:thetan}. Suppose $(\z,\y)$ satisfies the boundary jitter property (Definition \ref{def:jitter}) on $[0,\tau)$. Then for all $\psi\in\cts^{\delta,\z}$, Statement \ref{state:main} holds with $T=\theta_n$.
\end{statement}

The base case ($n=2$) follows from Proposition \ref{prop:theta2}. The following lemma states the induction step and is the main challenge in proving Proposition \ref{prop:mainalldelta}. The proof of Lemma \ref{lem:inductionstep} is given in Section \ref{sec:induction}.

\begin{lem}\label{lem:inductionstep}
Fix $\delta>0$. Let $2\leq n\leq N$. Assume that Statement \ref{state:thetan} holds. Then Statement \ref{state:thetan} holds with $\theta_{n+1}$ in place of $\theta_n$.
\end{lem}

We can now prove Proposition \ref{prop:mainalldelta}.

\begin{proof}[Proof of Proposition \ref{prop:mainalldelta}]
Let $\delta>0$ be arbitrary. By Proposition \ref{prop:theta2}, Lemma \ref{lem:inductionstep} and the principle of mathematical induction, Statement \ref{state:thetan} holds for $n=2,\dots,N+1$. In particular, Since $\theta_{N+1}\doteq\tau$, the proof is complete.
\end{proof}

As we now show, the proof of Theorem \ref{thm:main} is then a consequence of Proposition \ref{prop:mainalldelta}, Lemma \ref{lem:CZdense} and the closure property for the DM (Lemma \ref{lem:dmclosure}).

\begin{proof}[Proof of Theorem \ref{thm:main}]
Let $\psi\in\cts$ be arbitrary. By Lemma \ref{lem:CZdense} and the definition of $\cts^\z$, there is a Cauchy sequence $\{\psi_k\}_{k\in\N}$ in $\cts^\z$ such that $\psi_k\to\psi$ in $\cts$ as $k\to\infty$ and for each $k\in\N$, there exists $\delta_k>0$ such that $\psi_k\in\cts^{\delta_k,\z}$. Thus, by Proposition \ref{prop:mainalldelta}, for each $k\in\N$, 
\begin{itemize}
	\item[(i)] $\nabla_{\psi_k}\esm(\x)$ exists on $[0,\tau)$ and lies in $\dlr([0,\tau):\R^J)$; 
	\item[(ii)] $\nabla_{\psi_k}\esm(\x)(0)=\nabla_{\psi_k(0)}\pi(\x(0))$ and if $\nabla_{\psi_k}\esm(\x)$ is discontinuous at $t\in(0,\tau)$, then $\z(t)\in\S$ and $\nabla_{\psi_k}\esm(\x)$ is left continuous at $t\in(0,\tau)$ if and only if $\nabla_{\psi_k}\esm(\x)(t-)\in G_{\z(t)}$; 
	\item[(iii)] there is a unique solution $(\phi_k,\eta_k)$ to the DP associated with $\z$ for $\psi_k$ on $[0,\tau)$ and $\phi_k(t)=\nabla_{\psi_k}\esm(\x)(t+)$ for all $t\in[0,\tau)$.
\end{itemize}
Let $t\in(0,\tau)$. By (i) and (iii), we have $\nabla_{\psi_k}\esm(\x)(t-)=\phi_k(t-)$, $\nabla_{\psi_k}\esm(\x)(t)\in\{\phi_k(t-),\phi_k(t)\}$, and $\nabla_{\psi_k}\esm(\x)(t)=\phi_k(t)$ if and only if $\nabla_{\psi_k}\esm(\x)$ is right continuous at $t$. When combined with (ii) and the definition of the functional $\Theta_\z$ given in \eqref{eq:ThetaZ}, this shows that $\nabla_{\psi_k}\esm(\x)(t)=\Theta_\z(\phi_k)(t)$.

Since $\{\psi_k\}_{k\in\N}$ is Cauchy in $\cts$ and the DM is Lipschitz continuous (Theorem \ref{thm:dmlip}), there exists $(\phi,\eta)\in\dr([0,\tau):\R^J)\times\dr([0,\tau):\R^J)$ such that $(\phi_k,\eta_k)\to(\phi,\eta)$ in $\dr([0,\tau):\R^J)\times\dr([0,\tau):\R^J)$ as $k\to\infty$. The closure property of the DM (Lemma \ref{lem:dmclosure}) then shows that $(\phi,\eta)$ is the unique solution of the DP associated with $\z$ for $\psi$, so part 2 of Theorem \ref{thm:main} holds.

Let $t\in(0,\tau)$. We show that for all $k\in\N$ sufficiently large,
	\be\label{eq:Thetaphikphi}|\Theta_\z(\phi_k)(t)-\Theta_\z(\phi)(t)|\leq\norm{\phi_k-\phi}_t.\ee
First, suppose $\z(t)\in G\setminus\S$. Then (ii) and (iii) above imply that for each $k\in\N$, $\phi_k$ is continuous at $t$, and thus, $\phi$ is also continuous at $t$. Definition \eqref{eq:ThetaZ} of $\Theta_\z$ shows that $\Theta_\z(\phi_k)(t)=\phi_k(t)$ for each $k\in\N$ and $\Theta_\Z(\phi)(t)=\phi(t)$, so \eqref{eq:Thetaphikphi} holds. Alternatively, suppose $\z(t)\in\S$. Let $i\in\allN$ be the unique index such that $\allN(\z(t))=\{i\}$. Then \eqref{eq:Gx} and \eqref{eq:Hx} show that
	$$\partial G_{\z(t)}=\{x\in\R^J:\ip{x,n_i}=0\}=H_{\z(t)}.$$ 
If $\phi(t-)\not\in\partial G_{\z(t)}$, then for all $k\in\N$ sufficiently large, either $\phi(t-),\phi_k(t-)\in G_{\z(t)}$ or $\phi(t-),\phi_k(t-)\not\in G_{\z(t)}$. In both cases, \eqref{eq:Thetaphikphi} follows from the definition \eqref{eq:ThetaZ} of $\Theta_\z$. On the other hand, if $\phi(t-)\in\partial G_{\z(t)}=H_{\z(t)}$, then by condition 2 of the DP (Definition \ref{def:dp}), \eqref{eq:etatetatminus}, condition 1 of the DP and the continuity of $\psi$, we have 
	$$\phi(t)\in H_{\z(t)}\qquad\text{and}\qquad\phi(t)-\phi(t-)\in\spaan[d(\z(t))].$$ 
Thus, Lemma \ref{lem:projx} implies that $\phi(t)=\proj_{\z(t)}\phi(t-)=\phi(t-)$. In particular, $\phi$ is continuous at $t$ so \eqref{eq:Thetaphikphi} follows from the definition \eqref{eq:ThetaZ} of $\Theta_\z$. This exhausts all possible cases and so \eqref{eq:Thetaphikphi} must hold for all $t\in(0,\tau)$.

Let $t\in(0,\tau)$. By the triangle inequality, the fact that $\nabla_{\psi_k}\esm(\x)(t)=\Theta_\z(\phi_k)(t)$, the definition \eqref{eq:nablapsive} of $\nabla_{\psi_k}^\ve\esm(\x)$, \eqref{eq:Thetaphikphi} and the Lipschitz continuity of both the ESM (Theorem \ref{thm:esmlip}) and the DM, we have, for all $k\in\N$ sufficiently large,
\begin{align*}
	|\nabla_\psi^\ve\esm(\x)(t)-\Theta_{\z}(\phi)(t)|&\leq|\nabla_\psi^\ve\esm(\x)(t)-\nabla_{\psi_k}^\ve\esm(\x)(t)|\\
	&\qquad+|\nabla_{\psi_k}^\ve\esm(\x)(t)-\nabla_{\psi_k}\esm(\x)(t)|\\
	&\qquad+|\Theta_{\z}(\phi_k)(t)-\Theta_{\z}(\phi)(t)|\\
	&\leq\ve^{-1}|\esm(\x+\ve\psi)(t)-\esm(\x+\ve\psi_k)(t)|\\
	&\qquad+|\nabla_{\psi_k}^\ve\esm(\x)(t)-\nabla_{\psi_k}\esm(\x)(t)|\\
	&\qquad+\norm{\phi_k-\phi}_t\\
	&\leq |\nabla_{\psi_k}^\ve\esm(\x)(t)-\nabla_{\psi_k}\esm(\x)(t)|\\
	&\qquad+(\lip_\sm+\lip_{\dm})\norm{\psi-\psi_k}_t.
\end{align*}
Sending $\ve\downarrow0$ first and then $k\to\infty$ in the above display yields $\lim_{\ve\downarrow0}\nabla_\psi^\ve\esm(\x)(t)=\Theta_{\z}(\phi)(t)$. By Theorem \ref{thm:speu} and Lemma \ref{lem:projxv}, $\nabla_\psi\esm(\x)(0)=\nabla_{\psi(0)}\pi(\x(0))$. Thus, parts 1 and 4 of Theorem \ref{thm:main} hold. By the definition \eqref{eq:ThetaZ} of $\Theta_{\z}(\phi)$, we see that $\phi(t)=\nabla_\psi\esm(\x)(t+)$ for $t\in[0,\tau)$. This establishes part 3 of Theorem \ref{thm:main}. The final statement of the theorem follows from Lemma \ref{lem:Wempty} and the definition of $\tau$ given in \eqref{eq:tau}.
\end{proof}

We are left to prove Lemma \ref{lem:inductionstep}. The proof is given in Section \ref{sec:induction}. In the next section we state some useful lemmas that will be used in Section \ref{sec:induction}.

\subsection{Some useful lemmas}\label{sec:useful}

In this section we state some lemmas that will be needed in subsequent proofs. The first lemma is a trivial consequence of the Lipschitz continuity of the ESM. Recall the definition of $\nabla_\psi^\ve\esm(\x)$ given in \eqref{eq:nablapsive}.

\begin{lem}\label{lem:nablaesmtcompact}
Given $\x,\psi\in\cts$, $t\in[0,\infty)$ and a sequence $\{\ve_\ell\}_{\ell\in\N}$ such that $\ve_\ell\downarrow0$ as $\ell\to\infty$, there exists a subsequence, also denoted $\{\ve_\ell\}_{\ell\in\N}$, such that $\lim_{\ell\to\infty}\nabla_\psi^{\ve_\ell}\esm(\x)(t)$ exists.
\end{lem}

\begin{proof}
This is an immediate consequence of the compactness of $\{\nabla_\psi^{\ve_\ell}\esm(\x)(t)\}_{\ell\in\N}$ that follows from \eqref{eq:nablapsive} and the Lipschitz continuity of the ESM (Theorem \ref{thm:esmlip}).
\end{proof}

The following lemma will be useful for proving that Statement \ref{state:main} holds.

\begin{lem}\label{lem:stitchmain}
Fix $0\leq S<T<U<\infty$. Let $(\z,\y)$ be the solution of the ESP for $\x\in\cts_G$. Assume that Statement \ref{state:main} holds. Define $\x^S,\z^S$ and $\hat{\psi}^S$ as in \eqref{eq:xS}, \eqref{eq:zS} and \eqref{eq:psiSlimit}, respectively, and assume that Statement \ref{state:main} holds with $\x^S,\hat{\psi}^S,\z^S,U-S$ in place of $\x,\psi,\z,T$, respectively. Then Statement \ref{state:main} holds with $U$ in place of $T$.
\end{lem}

\begin{proof}
The statement and proof of Lemma \ref{lem:stitchmain} are analogous to the statement and proof of Lemma \ref{lem:stitchA}, so to avoid redundancy, we omit the proof. The main difference is the observation that if $\psi\in\cts^{\delta,\z}$, then $\hat{\psi}^S\in\cts^{\delta,\z^S}$.
\end{proof}

The next lemma states a Lipschitz continuity property for the orthogonal projections of solutions to the ESP. Given a subset $I\subseteq\allN$, we let $\Pi_I:\R^J\mapsto\spaan(\{n_i,i\in I\})$ denote the orthogonal projection with respect to $\ip{\cdot,\cdot}$. Observe that when $I=\allN(x)$ for some $x\in\partial G$, then $\Pi_I$ projects onto $H_{x}^\perp=\spaan(\{n_i,i\in\allN(x)\})$. For a function $f:[0,\infty)\mapsto\R^J$, define $\Pi_If:[0,\infty)\mapsto\spaan(\{n_i,i\in\allN\})$ by $(\Pi_If)(t)\doteq\Pi_I(f(t))$ for all $t\in[0,\infty)$. 

\begin{lem}\label{lem:projlip}
Given an ESP $\{(d_i,n_i,c_i),i\in\allN\}$ satisfying Assumption \ref{ass:setB} and $I\subseteq\allN$, there exists $\kappa_I<\infty$ such that if $(\z_1,\y_1)$ solves the ESP $\{(d_i,n_i,c_i),i\in\allN\}$ for $\x_1\in\cts$, $(\z_2,\y_2)$ solves the ESP $\{(d_i,n_i,c_i),i\in\allN\}$ for $\x_2\in\cts$, and $\allN(\z_1(t))\cup\allN(\z_2(t))\subseteq I$ for all $t\in[0,T)$, then for all $t\in[0,T)$,
\begin{align}\label{eq:PiZlip}
	\norm{\Pi_{I}\z_1-\Pi_{I}\z_2}_t&\leq \kappa_I\norm{\Pi_{I}\x_1-\Pi_{I}\x_2}_t.
\end{align}
\end{lem}

The proof of Lemma \ref{lem:projlip} is given in Appendix \ref{apdx:projlip}.

\subsection{Proof of the key induction step}\label{sec:induction}

In this section we prove Lemma \ref{lem:inductionstep}. \emph{Throughout this section we fix $\delta>0$, $2\leq n\leq N$ and assume that Statement \ref{state:thetan} holds}. Given $\x\in\cts_G$, let $(\z,\y)$ be the solution of the ESP for $\x$ and define $\tau$ as in \eqref{eq:tau} and $\{\theta_n\}_{n=2,\dots,N+1}$ as in \eqref{eq:thetan}. We need to show that if $(\z,\y)$ satisfies the boundary jitter property on $[0,\tau)$, then Statement \ref{state:main} holds with $T=\theta_{n+1}$. If $\theta_{n+1}=\theta_n$, the assertion is immediate. Suppose $\theta_{n+1}>\theta_n$. Set 
	\be\label{eq:t1n}t_1^{(n)}\doteq\theta_n.\ee 
Given $k\in\N$ for which $t_k^{(n)}$ is defined, if $t_k^{(n)}=\theta_{n+1}$, set $K_n=k$, where as if $t_k^{(n)}<\theta_{n+1}$, define $\rho_k^{(n)}$ to be the first time in the interval $(t_k^{(n)},t_k^{(n)}+\delta]$ that $\z$ reaches a face that is distinct from any of the faces that $\z(t_k^{(n)})$ lies on; that is,
	\be\label{eq:rhok}\rho_k^{(n)}\doteq\inf\lcb t\in(t_k^{(n)},t_k^{(n)}+\delta]:\allN(\z(t))\not\subseteq\allN(\z(t_k^{(n)}))\rcb\wedge (t_k^{(n)}+\delta),\ee
and define $t_{k+1}^{(n)}$ to be the first time in the interval $[\rho_k^{(n)},\theta_{n+1}]$ that $\z$ reaches the intersection of $n$ or more faces; that is,
	\be\label{eq:tk1n}t_{k+1}^{(n)}=\inf\lcb t\in[\rho_k^{(n)},\theta_{n+1}]:|\allN(\z(t))|\geq n\rcb.\ee
If $t_k^{(n)}<\theta_{n+1}$ for all $k\in\N$, set $K_n=\infty$ so that
	\be\label{eq:Kn} K_n\doteq\inf\lcb k\in\N:t_k^{(n)}=\theta_{n+1}\rcb.\ee 
If $K_n=\infty$, then \eqref{eq:thetan}, \eqref{eq:rhok}, \eqref{eq:tk1n} and the continuity of $\z$ imply that $t_k^{(n)}\to\theta_{n+1}$ as $k\to\infty$. Therefore, in order to prove Lemma \ref{lem:inductionstep}, we need to show that for all $1\leq k<K_n+1$, Statement \ref{state:main} holds with $T=t_k^{(n)}$. Since $t_1^{(n)}\doteq\theta_n$ and Statement \ref{state:thetan} holds by assumption, Statement \ref{state:main} holds with $T=t_1^{(n)}$. In Lemma \ref{lem:rhok} and Lemma \ref{lem:rho1} below, we show that if Statement \ref{state:main} holds with $T=t_k^{(n)}$, then Statement \ref{state:main} holds with $T=\rho_k^{(n)}$. In Lemma \ref{lem:sigma2}, we prove that if Statement \ref{state:main} holds with $T=t_k^{(n)}$, then Statement \ref{state:main} holds with $T=t_{k+1}^{(n)}$, which along with the principle of mathematical induction, will complete the proof that for all $1\leq k<K_n+1$, Statement \ref{state:main} holds with $T=t_k^{(n)}$.

We begin with the proof that Statement \ref{state:main} holds with $T=\rho_k^{(n)}$. The proof is split into two lemmas. In Lemma \ref{lem:rhok} we consider the case that $t_k^{(n)}>0$ and in Lemma \ref{lem:rho1} we consider the case that $k=1$ and $t_1^{(n)}\doteq\theta_n=0$. Since Lemma \ref{lem:rhok} is needed in the proof of Lemma \ref{lem:rho1}, we first consider the case that $t_k^{(n)}>0$. For the following lemma recall that $\delta>0$, $2\leq n\leq N$ are fixed and our assumption that Statement \ref{state:thetan} holds.

\begin{lem}\label{lem:rhok}
Given $\x\in\cts_G$, let $(\z,\y)$ denote the solution of the ESP for $\x$. Define $\tau$ as in \eqref{eq:tau} and assume that $(\z,\y)$ satisfies the boundary jitter property on $[0,\tau)$. Define $\theta_n$ and $\theta_{n+1}$ as in \eqref{eq:thetan} and assume that $\theta_{n+1}>\theta_n$. Define $\{t_k^{(n)}\}_{k=1,\dots,K_n}$, $\{\rho_k^{(n)}\}_{k=1,\dots,K_n-1}$ and $K_n\in\N$ are as in \eqref{eq:t1n}--\eqref{eq:Kn}. Let $1\leq k<K_n$. Assume that $t_k^{(n)}>0$ and Statement \ref{state:main} holds with $T=t_k^{(n)}$. Then Statement \ref{state:main} holds with $T=\rho_k^{(n)}$.
\end{lem}

\begin{proof}
For notational convenience, we drop the superscript $n$ notation and write $t_k$ and $\rho_k$ in place of $t_k^{(n)}$ and $\rho_k^{(n)}$, respectively.

Fix $\psi\in\cts^{\delta,\z}$. By \eqref{eq:rhok}, \eqref{eq:tk1n}, the continuity of $\z$ and the upper semicontinuity of $\allN(\cdot)$ (Lemma \ref{lem:allNusc}), we can choose $S\in[(t_k-\delta)\vee0,t_k)$ such that
\begin{align}\label{eq:allNzSrhok}
	\allN(\z(t))&\subsetneq\allN(\z(t_k)),&&t\in[S,t_k),\\ \label{eq:allNztkrhok}
	\allN(\z(t))&\subseteq\allN(\z(t_k)),&&t\in[t_k,\rho_k).
\end{align}
By \eqref{eq:rhok}, $\rho_k\leq t_k+\delta$. Thus, \eqref{eq:CZdelta} and the fact that $\z(t_k)\in\U$ imply that $\psi$ is constant on $[S,\rho_k]$. Since Statement \ref{state:main} holds with $T=t_k$ by assumption, $\nabla_\psi\esm(\x)(t)$ exists for all $t\in[0,t_k)$, there exists a unique solution $(\phi,\eta)$ of the DP for $\psi$ on $[0,t_k)$ and $\phi(t)=\nabla_\psi\esm(\x)(t+)$ for $t\in[0,t_k)$. It then follows from Lemma \ref{lem:projphirhok}, with $T=t_k$, that $\phi(t_k-)$ exists and $\phi(t_k-)=\proj_{\z(t_k)}\phi(S)$. Thus, $\nabla_\psi\esm(\x)(t_k-)$ exists and satisfies
	\be\label{eq:nablatknminus}\nabla_\psi\esm(\x)(t_k-)=\proj_{\z(t_k)}\phi(S).\ee

We claim that $\nabla_\psi\esm(\x)(t)$ exists for all $t\in[t_k,\rho_k)$ and satisfies 
	\be\label{eq:nablatknrhokn}\nabla_\psi\esm(\x)(t)=\proj_{\z(t_k)}\phi(S),\qquad t\in[t_k,\rho_k).\ee
We defer the proof of the claim and instead first derive some consequences of the claim. Parts 1 and 2 of Statement \ref{state:main} with $T=t_k$, together with \eqref{eq:nablatknminus} and \eqref{eq:nablatknrhokn}, imply that parts 1 and 2 of Statement \ref{state:main} hold with $T=\rho_k$. We now turn to part 3 of Statement \ref{state:main} holds with $T=\rho_k$. Define $\hat{\phi}(t)\doteq\nabla_\psi\esm(\x)(t+)$ and $\hat{\eta}(t)\doteq\hat{\phi}(t)-\psi(t)$ for $t\in[0,\rho_k)$. Observe that $\hat{\phi}(t)=\phi(t)$ for $t\in[0,t_k)$ and by \eqref{eq:nablatknrhokn}, $\hat{\phi}(t)=\proj_{\z(t_k)}\phi(S)$ for all $t\in[t_k,\rho_k)$. Due to the fact that $(\hat{\phi},\hat{\eta})$ solves the DP for $\psi$ on $[0,t_k)$, to prove part 3 of Statement \ref{state:main}, we are left to show that $(\hat{\phi},\hat{\eta})$ satisfies conditions 1 and 2 of the DP for $t\in[t_k,\rho_k)$ and condition 3 of the DP for $0\leq s<t<\rho_k$ whenever $t\in[t_k,\rho_k)$. Let $t\in[t_k,\rho_k)$. Condition 1 of the DP holds by definition. Lemma \ref{lem:projx}, the definition \eqref{eq:Hx} of $H_x$ and \eqref{eq:allNztkrhok} imply that $\proj_{\z(t_k)}\phi(S)\in H_{\z(t_k)}\subseteq H_{\z(t)}$, so condition 2 of the DP holds. Lastly, suppose $s\in[0,t)$. Since $\psi$ is constant on $[(t_k-\delta)\vee0,\rho_k)$ and $\hat{\phi}(t)=\hat{\phi}(t_k-)$ for all $t\in[t_k,\rho_k)$, we have $\hat{\eta}(t)=\hat{\eta}(t_k-)$ for all $t\in[t_k,\rho_k)$. When combined with the fact that $\hat{\eta}$ satisfies condition 3 of the DP for $\psi$ on $[0,t_k)$, we have
\begin{align*}
	\hat{\eta}(t)-\hat{\eta}(s)=\hat{\eta}(t_k-)-\hat{\eta}(s)&\in\cup_{r\in(s,t_k)}\spaan\lsb\cup_{u\in(s,r]}d(\z(u))\rsb\\
	&\subseteq\spaan\lsb\cup_{u\in(s,t]}d(\z(u))\rsb.
\end{align*}
Thus, condition 3 of the DP holds. This completes the proof that Statement \ref{state:main} holds with $T=\rho_k$.

It only remains to prove the claim \eqref{eq:nablatknrhokn}. Fix $t\in[t_k,\rho_k)$. By \eqref{eq:allNzSrhok}, \eqref{eq:allNztkrhok} the upper semicontinuity of $\allN(\cdot)$ (Lemma \ref{lem:allNusc}), the continuity of $\z$ and $\psi$, and the Lipschitz continuity of the ESM, for all $\ell\in\N$ sufficiently large,
	\be\label{eq:allNzStve}\allN(\z_{\ve_\ell}(u))\subseteq\allN(\z(t_k)),\qquad u\in[S,t].\ee 
By Lemma \ref{lem:nablaesmtcompact}, given a sequence $\{\ve_\ell\}_{\ell\in\N}$ with $\ve_\ell\downarrow0$ as $\ell\to\infty$, there exists a subsequence, also denoted $\{\ve_\ell\}_{\ell\in\N}$, such that $v\doteq\lim_{\ell\to\infty}\nabla_\psi^{\ve_\ell}\esm(\x)(t)$ exists. It suffices to show that $v=\proj_{\z(t_k)}\phi(S)$. Due to the uniqueness of the derivative projection operators stated in Lemma \ref{lem:projx}, this is equivalent to showing that $v\in H_{\z(t_k)}$ and $v-\phi(S)\in\spaan[d(\z(t_k))]$.

Now \eqref{eq:nablapsive}, condition 1 of the ESP and the fact that $\psi$ is constant on $[S,t]$ show that
\begin{align*}
	\nabla_\psi^{\ve_\ell}\esm(\x)(t)-\nabla_\psi^{\ve_\ell}\esm(\x)(S)&=\frac{1}{\ve_\ell}[\y_{\ve_\ell}(t)-\y_{\ve_\ell}(S)-(\y(t)-\y(S))].
\end{align*}
When combined with condition 3 of the ESP, \eqref{eq:allNzSrhok}, \eqref{eq:allNztkrhok} and \eqref{eq:allNzStve}, this implies that for all $\ve>0$ sufficiently small,
 \begin{align}\label{eq:nablaTnablaS}
	\nabla_\psi^{\ve_\ell}\esm(\x)(t)-\nabla_\psi^{\ve_\ell}\esm(\x)(S)\in\spaan[d(\z(t_k))].
\end{align}
Since $\spaan[d(\z(t_k))]$ is closed, taking limits as $\ell\to\infty$, we see that $v-\nabla_\psi\esm(\x)(S)\in\spaan[d(\z(t_k))]$. Then because $\phi(t)=\nabla_\psi\esm(\x)(t+)$ for all $t\in[0,t_k)$ and $\nabla_\psi\esm(\x)\in\dlr([0,t_k):\R^J)$, in order to show that $\phi(S)=\nabla_\psi\esm(\x)(S)$, it is enough to show that $\phi$ is continuous at $S$. Recall that $\z(S)\in G^\circ$. Therefore, by the continuity of $\z$, the fact that $(\phi,\phi-\psi)$ solves the DP associated with $\z$ for $\psi$ on $[0,t_k)$ and condition 3 of the DP, $\phi-\psi$ is constant in a neighborhood of $S$. Since $\psi$ is continuous, this implies $\phi$ is continuous at $S$, so $v-\phi(S)\in\spaan[d(\z(t_k))]$.

We are left to prove that $v\in H_{\z(t_k)}$. Let $s\in(S,t_k)$ be arbitrary. Define $\x^s,\z^s,\x_{\ve_\ell}^s$, $\z_{\ve_\ell}^s$ and $\psi_{\ve_\ell}^s$ as in \eqref{eq:xS}, \eqref{eq:zS}, \eqref{eq:xveS}, \eqref{eq:zveS} and \eqref{eq:psiveS}, but with $s$ and $\ve_\ell$ in place of $S$ and $\ve$, respectively. Then by the time-shift property of the ESP (Lemma \ref{lem:esmshift}) and \eqref{eq:xveSxS}, $\z^s=\esm(\x^s)$ and $\z_{\ve_\ell}^s=\esm(\x^s+\ve_\ell\psi_{\ve_\ell}^s)$. Thus, \eqref{eq:allNzSrhok}, \eqref{eq:allNztkrhok}, \eqref{eq:allNzStve}, \eqref{eq:zS} and \eqref{eq:zveS} imply that for $\ell\in\N$ sufficiently large, 
	$$\allN(\z^s(u))\cup\allN(\z_{\ve_\ell}^s(u))\subseteq\allN(\z(t_k))\qquad\text{for }u\in[0,t-s].$$
Let $I\doteq\allN(\z(t_k))$ and recall that $\Pi_I$ denotes the orthogonal projection from $\R^J$ onto $\spaan(\{n_i,i\in I\})=H_{\z(t_k)}^\perp$. By Lemma \ref{lem:projlip}, with $\z_{\ve_\ell}^s$, $\z^s$ and $t-s$ in place of $\z_1$, $\z_2$ and $T$, respectively, there is a constant $\lip_I<\infty$ such that for all $\ell\in\N$ sufficiently large,
\begin{align*}
	\frac{\norm{\Pi_I\z_{\ve_\ell}^s-\Pi_I\z^s}_{t-s}}{\ve_\ell}&\leq\lip_I\norm{\Pi_I\psi_{\ve_\ell}^s}_{t-s}=\lip_I|\Pi_I\nabla_\psi^{\ve_\ell}\esm(\x)(s)|,
\end{align*}
where the final equality uses \eqref{eq:psiveS} and the fact that $\psi$ is constant on $[s,t]$. Hence, by \eqref{eq:zS}, \eqref{eq:zveS}, \eqref{eq:nablapsive} and the assumption that $\nabla_\psi\esm(\x)$ exists on $[0,t_k)$,
\begin{align*}
	|\Pi_Iv|&=\lim_{\ell\to\infty}\frac{|\Pi_I\z_{\ve_\ell}^s(t-s)-\Pi_I\z^s(t-s)|}{\ve_\ell}\\
	&\leq\lim_{\ell\to\infty}\lip_I|\Pi_I\nabla_\psi^{\ve_\ell}\esm(\x)(s)|=\lip_I|\Pi_I\nabla_\psi\esm(\x)(s)|.
\end{align*}
Sending $s\uparrow t_k$ in the above display, invoking the identity \eqref{eq:nablatknminus} and noting that $\proj_{\z(t_k)}$ projects onto $H_{\z(t_k)}$ and $\Pi_I$ denotes orthogonal projection onto $H_{\z(t_k)}^\perp$, yields
	$$|\Pi_Iv|\leq\lip_I|\Pi_I\nabla_\psi\esm(\x)(t_k-)|=\lip_I|\Pi_I\proj_{\z(t_k)}\phi(S)|=0.$$
Thus, $v\in H_{\z(t_k)}$, completing the proof.
\end{proof}

In the following lemma we consider the case that $\theta_n=0$ and prove that Statement \ref{state:main} holds with $T=\rho_1^{(n)}$. This is relevant when the path $\z$ starts at the nonsmooth part of the boundary. The proof is much more involved than the proof of Lemma \ref{lem:rhok}. For the following lemma recall that $\delta>0$, $2\leq n\leq N$ are fixed and our assumption that Statement \ref{state:thetan} holds.

\begin{lem}\label{lem:rho1}
Given $\x\in\cts_G$, let $(\z,\y)$ denote the solution of the ESP for $\x$. Define $\tau$ as in \eqref{eq:tau} and assume that $(\z,\y)$ satisfies the boundary jitter property on $[0,\tau)$. Define $\theta_n$ and $\theta_{n+1}$ as in \eqref{eq:thetan} and assume that $\theta_{n+1}>\theta_n=0$. Define $t_1^{(n)}$ and $\rho_1^{(n)}$ as in \eqref{eq:t1n}--\eqref{eq:rhok}. Then Statement \ref{state:main} holds with $T=\rho_1^{(n)}$.
\end{lem}

\begin{proof}
For notational convenience, we drop the superscript $n$ notation and write $\rho_1$ in place of $\rho_1^{(n)}$.

Fix $\psi\in\cts^{\delta,\z}$. Recall the definition of $\nabla_v\pi(x)$ from \eqref{eq:pixv}, and note that by Theorem \ref{thm:speu} and Lemma \ref{lem:projxv}, 
	\be\label{eq:v0}v_0\doteq\nabla_\psi\esm(\x)(0)\doteq\lim_{\ve\downarrow0}\esm_\psi\esm(\x)(0)=\nabla_{\psi(0)}\pi(\x(0))\ee
and $v_0-\psi(0)\in\conv[d(\z(0))]$. By the definition of $\rho_1$ in \eqref{eq:rhok}, we have
	\be\label{eq:rho1}\allN(\z(t))\subseteq\allN(\z(0))\qquad\text{for all }t\in[0,\rho_1).\ee
We claim, and prove below, that for all $t\in(0,\rho_1)$, $\nabla_\psi\esm(\x)(t)$ exists and satisfies
\begin{equation}\label{eq:nablaproj}
	\nabla_\psi\esm(\x)(t)=\proj_{\z(0)}v_0,\qquad t\in(0,\rho_1).
\end{equation}
Given the claim, it follows that $\nabla_\psi\esm(\x)$ lies in $\dlr([0,\rho_1):\R^J)$ and is continuous on $(0,\rho_1)$, so parts 1 and 2 of Statement \ref{state:main} hold with $T=\rho_1$. Let $\hat{\phi}$ denote the right continuous regularization of $\nabla_\psi\esm(\x)$ on $[0,\rho_1)$; that is, $\hat{\phi}\equiv \proj_{\z(0)}v_0$ on $[0,\rho_1)$. By \eqref{eq:Hxdx}, $\proj_{\z(0)}v_0\in H_{\z(0)}$ and 
	$$\proj_{\z(0)}v_0-\psi(0)=(\proj_{\z(0)}v_0-v_0)+(v_0-\psi(0))\in\spaan\lsb d(\z(0))\rsb.$$
In view of the above display and \eqref{eq:rho1}, it is readily verified that $(\hat{\phi},\hat{\phi}-\psi)$ is the unique solution of the DP for $\psi$ on $[0,\rho_1)$, so part 3 of Statement \ref{state:main} holds.

We now turn to the proof of \eqref{eq:nablaproj}. Since $n\geq2$ and $\theta_n=0$, $\z(0)\in\U$. This, combined with the definition of $\cts^{\delta,\z}$ in \eqref{eq:CZdelta} and the fact that $\rho_k\leq\delta$ by \eqref{eq:rhok}, implies that $\psi$ is constant on $[0,\rho_1]$. Fix $T\in(0,\rho_1)$. By Lemma \ref{lem:nablaesmtcompact}, given a sequence $\{\ve_\ell\}_{\ell\in\N}$ with $\ve_\ell\downarrow0$ as $\ell\to\infty$, there exists a subsequence, also denoted $\{\ve_\ell\}_{\ell\in\N}$, such that the following limit exists
	\be\label{eq:vT}v_T\doteq\lim_{\ell\to\infty}\nabla_\psi^{\ve_\ell}\esm(\x)(T).\ee
It suffices to show that $v_T=\proj_{\z(0)}v_0$ must hold. Due to the uniqueness of the derivative projection operators stated in Lemma \ref{lem:projx}, this is equivalent to showing that $v_T\in H_{\z(0)}$ and $v_T-v_0\in\spaan[d(\z(0))]$.

We first show that $v_T-v_0\in\spaan[d(\z(0))]$. By \eqref{eq:rho1}, the Lipschitz continuity of the ESM and the upper semicontinuity of $\allN(\cdot)$ (Lemma \ref{lem:allNusc}), we have, for all $\ell\in\N$ sufficiently large:
	\be\label{eq:vet}\allN(\z_{\ve_\ell}(t))\subseteq\allN(\z(0))\qquad\text{for all }t\in[0,T].\ee
Due to \eqref{eq:nablapsive}, condition 1 of the ESP and the fact that $\psi$ is constant on $[0,T]$, it follows that
\begin{align*}
	\nabla_\psi^{\ve_\ell}\esm(\x)(T)-\nabla_\psi^{\ve_\ell}\esm(\x)(0)&=\frac{1}{\ve_\ell}\lsb\y_{\ve_\ell}(T)-\y_{\ve_\ell}(0)-(\y(T)-\y(0))\rsb.
\end{align*}
When combined with condition 3 of the ESP, \eqref{eq:rho1} and \eqref{eq:vet}, this implies that for all $\ell\in\N$ sufficiently large,
	$$\nabla_\psi^{\ve_\ell}\esm(\x)(T)-\nabla_\psi^{\ve_\ell}\esm(\x)(0)\in\spaan[d(\z(0))].$$
Since $\spaan[d(\z(0))]$ is a closed subspace, taking limits as $\ell\to\infty$ and using \eqref{eq:v0} and \eqref{eq:vT}, we see that $v_T-v_0\in\spaan[d(\z(0))]$.

We are left to prove that $v_T\in H_{\z(0)}$, which is the more complicated part of the proof. We consider two exhaustive and mutually exclusive cases separately. In Case 1, we show that there exists a time-shifted version of the path, which starts in $G\setminus\U$ and satisfies the assumptions of Lemma \ref{lem:rhok}. We then invoke Lemma \ref{lem:rhok} to prove our result. In Case 2, we similarly apply Lemma \ref{lem:rhok} to a time-shifted version of the path; however, that case is more complicated and will make crucial use of condition 3 of the boundary jitter property. \\

\emph{Case 1:} $\allN(\z(t))=\allN(\z(0))$ for some $t\in(0,T)$. \\
Let $U\in(0,T)$ be such that $\allN(\z(U))=\allN(\z(0))$. By \eqref{eq:thetan} and the assumption that $\theta_{n+1}>\theta_n=0$, 
	\be\label{eq:allNU}|\allN(\z(U))|=|\allN(\z(0))|=n.\ee 
By condition 2 of the boundary jitter property, there exists $S\in(0,U)$ such that $\z(S)\not\in\U$. Define $\x^S,\z^S,\y^S$ as in \eqref{eq:xS}--\eqref{eq:yS}, so $\x^S\in\cts_G$ and by the time-shift property of the ESP (Lemma \ref{lem:esmshift}), $(\z^S,\y^S)$ solves the ESP for $\x^S$. Define 
	$$\tau^S\doteq\inf\{t\in[0,\infty):\z^S\in\W\}=\tau-S,$$
where the last equality uses \eqref{eq:zS}, \eqref{eq:tau} and the fact that $S<\tau$. Since $(\z,\y)$ satisfies the boundary jitter property on $[0,\tau)$, it is straightforward to verify, using \eqref{eq:zS}--\eqref{eq:yS}, that $(\z^S,\y^S)$ satisfies the boundary jitter property on $[0,\tau^S)$. Define $\theta_n^S$ and $\theta_{n+1}^S$ as in \eqref{eq:thetan}, but with $\z^S$ and $\tau^S$ in place of $\z$ and $\tau$, respectively. Since $\z(S)\not\in\U$ and $S\in(0,U)$, it follows from \eqref{eq:thetan}, \eqref{eq:zS} and \eqref{eq:allNU} that $\theta_n^S\in(0,U-S]$. By \eqref{eq:zS} and \eqref{eq:rho1}, $\z^S(t)\subseteq\allN(\z(0))$ for all $t\in[0,\rho_1-S)$. Together, the previous two sentences, along with the fact that $U<T<\rho_1$, imply that 
	\be\label{eq:thetan1SthetanS}\allN(\z^S(\theta_n^S))=\allN(\z(0))\qquad\text{and}\qquad0<\theta_n^S\leq U-S<\theta_{n+1}^S.\ee
Set $t_1^S\doteq\theta_n^S$ and define 
\begin{align}\label{eq:rho1S}
	\rho_1^S&\doteq\inf\{t\in(\theta_n^S,\theta_n^S+\delta]:\allN(\z^S(t))\not\subseteq\allN(\z^S(\theta_n^S))\}\wedge(\theta_n^S+\delta)\\ \notag
	&=\lb\inf\{t\in(0,\theta_n^S+\delta]:\allN(\z(t))\not\subseteq\allN(\z(0))\}\wedge(\theta_n^S+\delta)\rb-S\\ \notag
	&\geq \rho_1-S,
\end{align}
where the second line uses \eqref{eq:zS}, \eqref{eq:rho1} and \eqref{eq:thetan1SthetanS}, and the final line follows from the definition \eqref{eq:rhok} of $\rho_1$ along with the fact that $t_1\doteq\theta_n=0$. In addition, by our assumption that Statement \ref{state:thetan} holds, Statement \ref{state:main} holds with $\x^S,\z^S$ and $t_1^S\doteq\theta_n^S$ in place of $\x,\z$ and $T$, respectively. Therefore, the conditions in Lemma \ref{lem:rhok} (when $k=1$) hold, so we can conclude that
\begin{itemize}
	\item[(i)] Statement \ref{state:main} holds with $\x^S,\z^S$ and $\rho_1^S$ in place of $\x,\z$ and $T$, respectively.
\end{itemize}
Proceeding, by Lemma \ref{lem:nablaesmtcompact}, we can choose a further subsequence of $\{\ve_\ell\}_{\ell\in\N}$, also denoted $\{\ve_\ell\}_{\ell\in\N}$, such that the following limit exists
	\be\label{eq:vS}v_S\doteq\lim_{\ell\to\infty}\nabla_\psi^{\ve_\ell}\esm(\x)(S).\ee
Define
	\be\label{eq:psiSvs}\hat{\psi}^S(\cdot)\doteq v_S+\psi(S+\cdot)-\psi(S).\ee
The above definition, along with \eqref{eq:CZdelta}, \eqref{eq:zS} and the fact that $\psi\in\cts^{\delta,\z}$, implies that $\hat{\psi}^S\in\cts^{\delta,\z^S}$. Thus, by (i) above, we have the following:
\begin{itemize}
	\item[(ii)] $\nabla_{\hat{\psi}^{S}}\esm(\x^{S})$ exists on $[0,\rho_1^{S})$; 
	\item[(iii)] there exists a unique solution $({\phi}^S,{\eta}^S)$ of the DP associated with $\z^S$ for $\hat{\psi}^S$ on $[0,\rho_1^{S})$, and ${\phi}^{S}(t)=\nabla_{\hat{\psi}^{S}}\esm(\x^{S})(t+)$ for $t\in[0,\rho_1^{S})$. 
\end{itemize}
We show that $\phi^S$ is constant on $[\theta_n^S,\rho_1^S)$. By \eqref{eq:thetan1SthetanS}--\eqref{eq:rho1S}, we have $\allN(\z^S(t))\subseteq\allN(\z(0))$ for all $t\in[\theta_n^S,\rho_1^S)$. Then \eqref{eq:Hx} implies that $H_{\z(0)}\subseteq H_{\z^S(t)}$ for all $t\in[\theta_n^S,\rho_1^S)$. When combined with condition 2 of the DP and \eqref{eq:thetan1SthetanS}, we have
	\be\label{eq:phiSthetanS}{\phi}^S(\theta_n^S)\in H_{\z^S(\theta_n^S)}=H_{\z(0)}\subseteq H_{\z^S(t)},\qquad t\in[\theta_n^S,\rho_1^S).\ee
Since $\hat{\psi}^S\in\cts^{\delta,\z^S}$ and $\z^S(\theta_n^S)\in\U$, $\hat{\psi}^S$ is constant on $[\theta_n^S,\theta_n^S+\delta]$, and in particular, by \eqref{eq:rho1S}, $\hat{\psi}^S$ is constant on $[\theta_n^S,\rho_1^S]$. It is readily verified that if $(\phi^S,\eta^S)$ are also constant on $[\theta_n^S,\rho_1^S)$, then conditions 1--3 of the DP associated with $\z^S$ for $\hat{\psi}^S$ are satisfied for all $s,t\in[\theta_n^S,\rho_1^S)$. Therefore, by the uniqueness of solutions to the DP, ${\phi}^S$ must be constant on $[\theta_n^S,\delta-S)$. This, together with (iii) above and \eqref{eq:phiSthetanS}, implies that 
	\be\label{eq:psiSvS}\nabla_{\hat{\psi}^S}\esm(\x^S)(t)={\phi}^S(\theta_n^S)\in H_{\z(0)},\qquad t\in(\theta_n^S,\rho_1^S).\ee

For each $\ell\in\N$, define $\x_{\ve_\ell}^S,\z_{\ve_\ell}^S,\hat{\psi}_{\ve_\ell}^S$ as in \eqref{eq:xveS}--\eqref{eq:zveS} and \eqref{eq:psiveS}, but with $\ve_\ell$ in place of $\ve$. Then by \eqref{eq:xveSxS} and \eqref{eq:vS}--\eqref{eq:psiSvs}, the following hold:
\begin{itemize}
	\item[(iv)] $\X_{\ve_\ell}^S=\X^S+\ve_\ell\hat{\psi}_{\ve_\ell}^S$;
	\item[(v)] $\hat{\psi}_{\ve_\ell}^S\to\hat{\psi}^S$ uniformly on $[0,\infty)$ as $\ell\to\infty$. 
\end{itemize}
Since $0<S<U<T<\rho_1$ by definition, $\theta_n^S\leq U-S$ by \eqref{eq:thetan1SthetanS}, and $\rho_1-S<\rho_1^S$ by \eqref{eq:rho1S}, it follows that $T-S\in(\theta_n^S,\rho_1^S)$. We can now conclude that
\begin{align*}
	v_T&=\lim_{\ell\to\infty}\frac{\esm(\x+\ve_\ell\psi)(T)-\esm(\x)(T)}{\ve_\ell}\\
	&=\lim_{\ell\to\infty}\frac{\esm(\x^S+\ve_\ell\hat{\psi}_{\ve_\ell}^S)(T-S)-\esm(\x^S)(T-S)}{\ve_\ell}\\
	&=\lim_{\ell\to\infty}\nabla_{\hat{\psi}_{\ve_\ell}^S}^{\ve_\ell}\esm(\x^S)(T-S)\\
	&=\nabla_{\hat{\psi}^S}\esm(\x^S)(T-S)\in H_{\z(0)},
\end{align*}
where the first equality follows from \eqref{eq:vT} and \eqref{eq:nablapsive}; the second equality is due to the time-shift property of the ESP and (iv); the third equality follows from \eqref{eq:nablapsive}; the final equality holds because of (ii), (v) and Proposition \ref{prop:psive}; and the inclusion is due to \eqref{eq:psiSvS}. This proves Case 1. \\

\emph{Case 2:} $\allN(\z(t))\subsetneq\allN(\z(0))$ for all $t\in(0,T)$. \\
Since $\z$ satisfies conditions 3 and 4 of the boundary jitter property, $\z(0)\in\U$ and $\allN(\z(t))\subsetneq\allN(\z(0))$ for all $t\in(0,T)$, there is a nested decreasing sequence
	\be\label{eq:Tchi0}T>\chi_0>u_1\geq\chi_1>\cdots>u_j\geq\chi_j>\cdots>0\ee
such that the conclusions of Lemma \ref{lem:chij} hold.  Condition 3 of the boundary jitter property, the fact that $T<\tau$ and \eqref{eq:chijuj} imply that $\allN(\z(0))=\cup_{j\geq m}\allN(\z(\chi_j))$ for all $m\in\N$. Let $\gamma>0$ be arbitrary. Define the compact set $C\subseteq\R^J$ by
	\be\label{eq:compactC}C\doteq\lcb y\in\R^J:|y|\leq\lip_\dm\lip_{\esm}|\psi(0)|\rcb.\ee
By \eqref{eq:limKproduniform} in Corollary \ref{lem:accumulateuniformly}, we can choose $m=m(\gamma,C)\in\N$ such that
	\be\label{eq:projchibound}\sup_{y\in C}\left|\lsb\proj_{\z(\chi_m)}\cdots\proj_{\z(\chi_1)}\rsb y-\proj_{\z(0)}y\right|\leq\gamma.\ee

For $j\in\{0,m\}$, define $\x^{\chi_j},\z^{\chi_j},\y^{\chi_j}$ as in \eqref{eq:xS}--\eqref{eq:yS}, but with $\chi_j$ in place of $S$, so $\x^{\chi_j}\in\cts_G$ and by the time-shift property of the ESP, $(\z^{\chi_j},\y^{\chi_j})$ solves the ESP for $\x^{\chi_j}$. As in Case 1, define
	$$\tau^{\chi_j}\doteq\inf\{t\in[0,\infty):\z^{\chi_j}\in\W\}=\tau-\chi_j,$$
where the last equality last equality uses \eqref{eq:zS}, \eqref{eq:tau} and the fact that $\chi_j<\tau$. As in Case 1, since $(\z,\y)$ satisfies the boundary jitter property on $[0,\tau)$, $(\z^{\chi_j},\y^{\chi_j})$ satisfies the boundary jitter property on $[0,\tau^{\chi_j})$. Define $\theta_n^{\chi_j}$ and $\theta_{n+1}^{\chi_j}$ as in \eqref{eq:thetan}, but with $\z^{\chi_j}$ and $\tau^{\chi_j}$ in place of $\z$ and $\tau$, respectively. Since $\allN(\z(t))\subsetneq\allN(\z(0))$ for all $t\in(0,T)$, $\chi_j\in(0,T)$ and $|\allN(\z(0))|=n$ by our assumption $\theta_{n+1}>\theta_n=0$, it follows from \eqref{eq:zS} that $\theta_n^{\chi_j}\geq T-\chi_j$. Set $t_1^{\chi_j}\doteq\theta_n^{\chi_j}$ and define
\begin{align*}
	\rho_1^{\chi_j}&\doteq\inf\{t\in(\theta_n^{\chi_j},\theta_n^{\chi_j}+\delta]:\allN(\z^{\chi_j}(t))\not\subseteq\allN(\z^{\chi_j}(\theta_n^{\chi_j}))\}\wedge(\theta_n^{\chi_j}+\delta)>T-\chi_j.
\end{align*}
By our assumption that Statement \ref{state:thetan} holds, Statement \ref{state:main} holds with $\x^{\chi_j},\z^{\chi_j}$ and $t_1^{\chi_j}\doteq\theta_n^{\chi_j}$ in place of $\x,\z$ and $T$, respectively. Therefore, the conditions in Lemma \ref{lem:rhok} (when $k=1$) hold, so we can conclude that
\begin{itemize}
	\item[(a)] Statement \ref{state:main} holds with $\x^{\chi_j},\z^{\chi_j}$ and $\rho_1^{\chi_j}$ in place of $\x,\z$ and $T$, respectively.
\end{itemize}
By Lemma \ref{lem:nablaesmtcompact}, there is a further subsequence of $\{\ve_\ell\}_{\ell\in\N}$, also denoted $\{\ve_\ell\}_{\ell\in\N}$, such that the following limits exist:
	\be\label{eq:vchimvchi0}v_{\chi_m}\doteq\lim_{\ell\to\infty}\nabla_\psi^{\ve_\ell}\esm(\x)(\chi_m)\qquad\text{and}\qquad v_{\chi_0}\doteq\lim_{\ell\to\infty}\nabla_\psi^{\ve_\ell}\esm(\x)(\chi_0).\ee
For $j\in\{0,m\}$, define
	\be\label{eq:psichij}\hat{\psi}^{\chi_j}(\cdot)\doteq v_{\chi_j}+\psi(\chi_j+\cdot)-\psi(\chi_j).\ee
Note that since $\psi$ is constant on $[0,\rho_1]$ and $T\in(0,\rho_1)$, 
\begin{itemize}
	\item[(b)] $\hat{\psi}^{\chi_j}\equiv v_{\chi_j}$ on $[0,T-\chi_j]$. 
\end{itemize}
As in Case 1, the fact that $\psi\in\cts^{\delta,\z}$ implies that $\hat{\psi}^{\chi_j}\in\cts^{\delta,\z^{\chi_j}}$. Thus, by (a), we have the following:
\begin{itemize}
	\item[(c)] $\nabla_{\hat{\psi}^{\chi_j}}\esm(\x^{\chi_j})$ exists on $[0,\rho_1^{\chi_j})$; 
	\item[(d)] there exists a unique solution $({\phi}^{\chi_j},{\eta}^{\chi_j})$ of the DP associated with $\z^{\chi_j}$ for $\hat{\psi}^{\chi_j}$ on $[0,\rho_1^{\chi_j})$, and ${\phi}^{\chi_j}(t)=\nabla_{\hat{\psi}^{\chi_j}}\esm(\x^{\chi_j})(t+)$ for $t\in[0,\rho_1^{\chi_j})$.
\end{itemize}
For each $\ell\in\N$, define $\x_{\ve_\ell}^{\chi_j},\z_{\ve_\ell}^{\chi_j}$ and $\hat{\psi}_{\ve_\ell}^{\chi_j}$ as in \eqref{eq:xveS}--\eqref{eq:zveS} and \eqref{eq:psiveS}, but with $\chi_j$ and $\ve_\ell$ in place of $S$ and $\ve$, respectively. Then by \eqref{eq:xveSxS} and \eqref{eq:vchimvchi0}--\eqref{eq:psichij},
\begin{itemize}
	\item[(e)] $\x_{\ve_\ell}^{\chi_j}=\x^{\chi_j}+\ve_\ell\hat{\psi}_{\ve_\ell}^{\chi_j}$;
	\item[(f)] $\hat{\psi}_{\ve_\ell}^{\chi_j}\to\hat{\psi}^{\chi_j}$ uniformly on $[0,\infty)$ as $\ell\to\infty$.
\end{itemize}

Let $I\doteq\allN(\z(0))$. Recall that $\Pi_I$ denotes the orthogonal projection from $\R^J$ onto $\spaan(\{n_i,i\in\allN(\z(0))\})=H_{\z(0)}^\perp$. By \eqref{eq:zS}, \eqref{eq:rho1} and \eqref{eq:vet}, for all $\ell\in\N$ sufficiently large, $\z^{\chi_0}(t),\z_{\ve_\ell}^{\chi_0}(t)\subseteq I$ for all $t\in[0,T-\chi_0]$. By \eqref{eq:vT}, \eqref{eq:nablapsive}, the time-shift property of the ESP, Lemma \ref{lem:projlip} with $\z_1=\z_{\ve_\ell}^{\chi_0}$ and $\z_2=\z^{\chi_0}$, and (b), (e) and (f) above, we have
\begin{align*}
	|\Pi_Iv_T|&=\lim_{\ell\to\infty}\frac{|\Pi_I(\z_\ve(T)-\z(T))|}{\ve_\ell}\\
	&=\lim_{\ell\to\infty}\frac{|\Pi_I(\z_{\ve_\ell}^{\chi_j}(T-\chi_j)-\z^{\chi_j}(T-\chi_j))|}{\ve_\ell}\\
	&\leq\lim_{\ell\to\infty}\lip_I\norm{\Pi_I\hat{\psi}_{\ve_\ell}^{\chi_0}}_{T-\chi_0}=\kappa_I|\Pi_Iv_{\chi_0}|.
\end{align*}
Thus, to prove that $v_T\in H_{\z(0)}$, it suffices to show that $v_{\chi_0}\in H_{\z(0)}$.

We claim, and prove below, that for each $1\leq j\leq m$,
	\be\label{eq:phiSchijS}{\phi}^{\chi_m}(\chi_{j-1}-\chi_m)=\proj_{\z(\chi_j)}{\phi}^{\chi_m}(\chi_j-\chi_m).\ee
Iterating this recursion relation yields
	\be\label{eq:phiSchi0S}{\phi}^{\chi_m}(\chi_0-\chi_m)=\lsb\proj_{\z(\chi_1)}\cdots\proj_{\z(\chi_m)}\rsb {\phi}^{\chi_m}(0).\ee
By \eqref{eq:zS} and Lemma \ref{lem:chij}, $\z^{\chi_m}(\chi_0-\chi_m)=\z(\chi_0)\in G^\circ$. By the continuity of $\z^{\chi_m}$ and the fact that $G^\circ$ is open, $\z^{\chi_m}(t)\in G^\circ$ for all $t$ in a neighborhood of $\chi_0-\chi_m$. By condition 3 of the DP, \eqref{eq:allNx} and \eqref{eq:Hx}, this implies that $\eta^{\chi_m}$ is constant in a neighborhood of $\chi_0-\chi_m$. Since $\hat{\psi}^{\chi_m}$ is constant on $[0,T-\chi_m]$ by (b) above, condition 1 of the DP implies that $\phi^{\chi_m}$ is also constant in a neighborhood of $\chi_0-\chi_m$. In particular, $\phi^{\chi_m}(\chi_0-\chi_m)=\nabla_{\hat{\psi}^{\chi_m}}\esm(\x^{\chi_m})(\chi_0-\chi_m)$ due to (d) above. This, along with \eqref{eq:vchimvchi0}, \eqref{eq:nablapsive}, the time-shift property of the ESP, (c) and (f) above, Proposition \ref{prop:psive} and \eqref{eq:phiSchi0S}, implies
\begin{align}\label{eq:vchi0}
	v_{\chi_0}&=\lim_{\ell\to\infty}\nabla_{\psi_{\ve_\ell}^{\chi_m}}^{\ve_\ell}\esm(\x^{\chi_m})(\chi_0-\chi_m)\\ \notag
	&=\nabla_{\psi^{\chi_m}}\esm(\x^{\chi_m})(\chi_0-\chi_m)\\ \notag
	&={\phi}^{\chi_m}(\chi_0-\chi_m)\\ \notag
	&=\lsb\proj_{\z(\chi_1)}\cdots\proj_{\z(\chi_m)}\rsb{\phi}^{\chi_m}(0).
\end{align}
By (d) above, the Lipschitz continuity of the DM (Theorem \ref{thm:dmlip}), \eqref{eq:psichij}, \eqref{eq:vchimvchi0}, the Lipschitz continuity of the ESM (Theorem \ref{thm:esmlip}) and because $\psi$ is constant on $[0,T]$, we have
\begin{align*}
	|{\phi}^{\chi_m}(0)|&\leq\lip_\dm|\hat{\psi}^{\chi_m}(0)|=\lip_\dm|v_{\chi_m}|\leq\lip_\dm\lip_{\esm}|\psi(0)|.
\end{align*}
Then by \eqref{eq:vchi0}, \eqref{eq:compactC} and \eqref{eq:projchibound}, 
\begin{align*}
	|v_{\chi_0}-\proj_{\z(0)}\phi^{\chi_m}(0)|\leq\gamma.
\end{align*}
Since $\gamma>0$ was arbitrary and $\proj_{\z(0)}$ projects onto $H_{\z(0)}$, we have $v_{\chi_0}\in H_{\z(0)}$.

We are left to prove that \eqref{eq:phiSchijS} holds. Fix $1\leq j\leq m$. We first show that ${\phi}^{\chi_m}$ is constant on $[\chi_j-\chi_m,u_j-\chi_m]$, where $u_j$ is as in \eqref{eq:Tchi0}. By condition 2 of the DP, \eqref{eq:zS}, \eqref{eq:Hx} and \eqref{eq:chijuj}, ${\phi}^{\chi_m}(\chi_j-\chi_m)\in H_{\z(\chi_j)}\subseteq H_{\z(t)}$ for all $t\in[\chi_j,u_j]$. Since $\hat{\psi}^{\chi_m}$ is constant on $[\chi_j-\chi_m,u_j-\chi_m]$, it is readily checked that if ${\phi}^{\chi_m}$ is constant on $[\chi_j-\chi_m,u_j-\chi_m]$, then $({\phi}^{\chi_m},{\phi}^{\chi_m}-\psi^{\chi_m})$ satisfies conditions 1--3 of the DP associated with $\z^{\chi_m}$ on the interval $[\chi_j-\chi_m,u_j-\chi_m]$. It then follows from uniqueness of solutions to the DP that ${\phi}^{\chi_m}$ must be constant on $[\chi_j-\chi_m,u_j-\chi_m]$. Next, by condition 2 of the DP and \eqref{eq:zS}, we have
	\be\label{eq:phichim}{\phi}^{\chi_m}(\chi_{j-1}-\chi_m)\in H_{\z(\chi_{j-1})}.\ee
By condition 3 of the DP and the facts that ${\phi}^{\chi_m}$ is constant on $[\chi_j-\chi_m,u_j-\chi_m]$ and $\hat{\psi}^{\chi_m}$ is constant on $[0,T-\chi_m]$ by (b) above, we have
\begin{align*}
	\phi^{\chi_m}(\chi_{j-1}-\chi_m)-{\phi}^{\chi_m}(\chi_j-\chi_m)&={\phi}^{\chi_m}(\chi_{j-1}-\chi_m)-{\phi}^{\chi_m}(u_j-\chi_m)\\
	&\in\spaan[\cup_{u\in(u_j-\chi_m,\chi_{j-1}-\chi_m]}d(\z^{\chi_m}(u))].
\end{align*}
Due to \eqref{eq:zS} and \eqref{eq:chijuj}, the above display implies
	\be\label{eq:phichimchij1chim}\phi^{\chi_m}(\chi_{j-1}-\chi_m)-{\phi}^{\chi_m}(\chi_j-\chi_m)\in\spaan[d(\z(\chi_{j-1}))].\ee
Thus, \eqref{eq:phichim}--\eqref{eq:phichimchij1chim} and the uniqueness of the projection operators shown in Lemma \ref{lem:projx} imply that \eqref{eq:phiSchijS} holds.
\end{proof}

\begin{lem}\label{lem:sigma2}
Given $\x\in\cts_G$, let $(\z,\y)$ denote the solution of the ESP for $\x$. Define $\tau$ as in \eqref{eq:tau} and assume that $(\z,\y)$ satisfies the boundary jitter property on $[0,\tau)$. Define $\theta_n$ and $\theta_{n+1}$ as in \eqref{eq:thetan} and assume that $\theta_{n+1}>\theta_n$. Define $\{t_k^{(n)}\}_{k=1,\dots,K_n}$, $\{\rho_k^{(n)}\}_{k=1,\dots,K_n-1}$, $K_n\in\N$, as in \eqref{eq:t1n}--\eqref{eq:Kn}. Let $1\leq k<K_n$. Assume that Statement \ref{state:main} holds with $T=t_k^{(n)}$. Then Statement \ref{state:main} holds with $T=t_{k+1}^{(n)}$.
\end{lem}

\begin{proof}
For notational convenience, we drop the superscript $n$ notation and write $t_k$, $t_{k+1}$ and $\rho_k$ in place of $t_k^{(n)}$, $t_{k+1}^{(n)}$ and $\rho_k^{(n)}$, respectively.

Fix $\psi\in\cts^{\delta,\z}$.  By Lemma \ref{lem:rhok} and Lemma \ref{lem:rho1}, Statement \ref{state:main} holds with $T=\rho_k$, so if $\rho_k=t_{k+1}$, we are done. Suppose that $\rho_k<t_{k+1}$. Let $S\in(\rho_k,t_{k+1})$. By the definition \eqref{eq:tk1n} of $t_{k+1}$, $|\allN(\z(t))|\leq n$ for all $t\in[S,t_{k+1})$. Since Statement \ref{state:main} holds with $T=\rho_k$ and $S<\rho_k$, $\nabla_\psi\esm(\x)(S)$ exists. Define $\x^S,\z^S,\y^S,\hat{\psi}^S$ as in \eqref{eq:xS}--\eqref{eq:yS} and \eqref{eq:psiSlimit}. Define $\tau^S\doteq\inf\{t\in[0,\infty):\z^S\in\W\}=\tau-S$, where we have used \eqref{eq:zS} and the fact that $S\in[0,\tau)$. Then, since $(\z,\y)$ satisfies the boundary jitter property on $[0,\tau)$, it is straightforward to verify, using \eqref{eq:zS}--\eqref{eq:yS}, that $(\z^S,\y^S)$ satisfies the boundary jitter property on $[0,\tau^S)$. Using \eqref{eq:zS} and \eqref{eq:psiSlimit}, it is readily verified that $\hat{\psi}^S\in\cts^{\delta,\z^S}$. Define $\theta_n^S\doteq\inf\{t\in[0,\tau^S):|\allN(\z^S(t))|\geq n\}\wedge\tau^S$, so by \eqref{eq:tk1n} and the fact that $|\allN(\z^S(t))|=|\allN(\z(S+t))|\leq n$ for all $t\in[0,t_{k+1}-S)$, we have $t_{k+1}=S+\theta_n^S$. Since Statement \ref{state:thetan} holds by assumption, Statement \ref{state:main} holds with $\x^S,\z^S,\hat{\psi}^S,\phi^S,\eta^S$ and $\theta_n^S$ in place of $\x,\z,\psi,\phi,\eta$ and $T$, respectively. Therefore, by Lemma \ref{lem:stitchmain} (with $t_k$ and $S+\theta_n^S$ in place of $T$ and $U$, respectively), Statement \ref{state:main} holds with $t_{k+1}$ in place of $T$.
\end{proof}

We conclude this section with the proof of Lemma \ref{lem:inductionstep}.

\begin{proof}[Proof of Lemma \ref{lem:inductionstep}]
By assumption, Statement \ref{state:main} holds with $T=\theta_n$. Since $t_1^{(n)}\doteq\theta_n$, Statement \ref{state:main} holds with $T=t_1^{(n)}$. Then Lemma \ref{lem:sigma2} and the principle of mathematical induction imply that Statement \ref{state:main} holds with $T=t_k^{(n)}$ for $1\leq k<K_n+1$. Since either $K_n<\infty$ and $t_{K_n}^{(n)}=\theta_{n+1}$, or $K_n=\infty$ and $t_k^{(n)}\to\theta_{n+1}$ as $k\to\infty$, it follows that Statement \ref{state:main} holds with $T=\theta_{n+1}$.
\end{proof}

\appendix

\section{Proof of Lemma \ref{lem:local}}\label{apdx:local}

In this section we prove Lemma \ref{lem:local}. Throughout this section we fix an SP $\{(d_i,n_i,c_i),i\in\allN\}$ satisfying Assumption \ref{ass:linearlyind} and assume that $(\z,\y)$ is a solution of the SP for $\x\in\cts_G$.

\begin{proof}[Proof of Lemma \ref{lem:local}]
By \cite[Definition 1.1]{Ramanan2006}, there exists a measurable function $\gamma:[0,\infty)\mapsto \mathbb{S}^{J-1}$ such that $\gamma(t)\in d(\z(t))$ ($d|\y|$-almost everywhere) and
		\be\label{eq:yty0gamma}\y(t)=\int_{(0,t]}\gamma(s)d|\y|(s).\ee 
By \eqref{eq:dx} and Assumption \ref{ass:linearlyind}, there exists a unique (up to a set of $d|\y|$-measure zero) measurable function $\xi:[0,\infty)\mapsto\R_+^N$ such that
	\be\label{eq:gammasumzeta}\gamma(t)=\sum_{i\in\allN(\z(t))}\xi^i(t)d_i\qquad\text{$d|\y|$-almost everywhere}.\ee
For each $i\in\allN$, define
	\be\label{eq:Lizetai}L^i(t)\doteq \int_{(0,t]}1_{\{\z(s)\in F_i\}}\xi^i(s)d|\y|(s).\ee
Since $\xi$ takes values in $\R_+^N$, \eqref{eq:Lizetai} implies that for each $i\in\allN$, $L^i$ is nondecreasing and \eqref{eq:dLi} holds. By \eqref{eq:yty0gamma}--\eqref{eq:Lizetai}, for all $t\in[0,\infty)$,
	$$\y(t)=\sum_{i\in\allN}\lb\int_{(0,t]}1_{\{\z(s)\in F_i\}}\xi^i(s)d|\y|(s)\rb d_i=R L(t).$$
Together with the linear independence condition Assumption \ref{ass:linearlyind}, this implies that $L$ is uniquely defined and there exists a positive constant $\tilde{\lip}<\infty$ such that if, for $k=1,2$, $(\z_k,\y_k)$ is the solution of the SP $\{(d_i,n_i,c_i),i\in\allN\}$ for $\x_k\in\cts_G$ and $L_k$ is as above, but with $\z_k,\y_k$ and $L_k$ in place of $\z,\y$ and $L$, then for all $T\in(0,\infty)$, $\norm{L_1-L_2}_T\leq\tilde{\lip}\norm{\y_1-\y_2}_T$. The Lipschitz continuity property \eqref{eq:locallip} then follows (with $\lip_L=\tilde{\lip}(1+\lip_\sm)$) from condition 1 of the ESP, the triangle inequality and the Lipschitz continuity property \eqref{eq:Zlip}.
\end{proof}

\section{Proof of Lemma \ref{lem:CZdense}}\label{apdx:CZdense}

In this section we prove that if $\z$ satisfies condition 2 of the boundary jitter property, then $\cts^\z$ is dense in $\cts$.

\begin{proof}[Proof of Lemma \ref{lem:CZdense}]
Fix $\psi\in\cts$. We need to show that given $T<\infty$ and $\ve>0$, there exists $\delta>0$ and $\zeta\in\cts^{\delta,\z}$ such that $\norm{\psi-\zeta}_T<\ve$. Let $T<\infty$ and $\ve>0$ be arbitrary. Since $\psi$ is uniformly continuous on the compact interval $[0,T]$, we can choose $\gamma>0$ such that
	$$w_T(\psi,\gamma)\doteq\sup_{0\leq s<t\leq T,|t-s|<\gamma}|\psi(t)-\psi(s)|<\ve.$$
Since $\z$ is continuous and $G\setminus\U$ is relatively open in $G$, $\{s\in(0,T):\z(s)\in G\setminus\U\}$ is open and can thus be written as the countable union of disjoint open intervals $\{(s_j,t_j)\}_{j\in\N}\subseteq(0,T)$. By condition 2 of the boundary jitter property, we can choose $m\in\N$ sufficiently large so that $\sum_{j=1,\dots,m}|t_j-s_j|\geq T-\gamma/4$. Without loss of generality, we can assume the intervals are ordered so that $t_j\leq s_{j+1}$ for $j=1,\dots,m-1$. Consequently, $s_1\leq\gamma/4$, $t_m\geq T-\gamma/4$ and
	$$\{t\in[0,\infty):\z(t)\in\U\}\subseteq [0,s_1]\cup(\cup_{j=1,\dots,m-1}[t_j,s_{j+1}])\cup[t_m,\infty).$$

Let $0<\delta<\frac{\gamma}{4}\wedge\frac{1}{3}\min_{j=1,\dots,m}(t_j-s_j)$ and define the partially linearly interpolated paths $\zeta\in\cts$ as follows: set $\zeta(t)\doteq{\psi}(s_1+\delta)$ for all $t\in[0,s_1+\delta]$ and for $j=1,\dots,m-1$, define
	$$\zeta(t)\doteq
	\begin{cases}
		\psi(t)&t\in[s_j+\delta,t_j-2\delta],\\
		\psi(t_j-2\delta)+\frac{\psi(s_{j+1}+\delta)-\psi(t_j-2\delta)}{\delta}(t-t_j+2\delta)&t\in[t_j-2\delta,t_j-\delta],\\
		\psi(s_{j+1}+\delta)&t\in[t_j-\delta,s_{j+1}+\delta].
	\end{cases}$$
Set $\zeta(t)\doteq\psi(t)$ for all $t\in[s_m+\delta,t_m-\delta]$ and $\zeta(t)\doteq\psi(t_m-\delta)$ for all $t\in[t_m-\delta,\infty)$. By definition, $\zeta$ is constant on a $\delta$-neighborhood of $I$ in $[0,\infty)$, so $\zeta\in\cts^{\delta,\z}$. We are left to show that $\norm{\psi-\zeta}_T<\ve$. Since $\delta<\gamma/4$, $s_1<\gamma/4$, $T-t_m<\gamma/4$ and $s_{j+1}-t_j<\gamma/4$ for all $j=1,\dots,m-1$, we have
\begin{align*}
	\norm{\psi-\zeta}_T&\leq\sup_{0\leq t\leq s_1+\delta}|\psi(t)-\psi(s_1+\delta)|\\
	&\qquad\vee\max_{j=1,\dots,m-1}\sup_{t\in[t_j-2\delta, t_j-\delta]}|\psi(t)-\psi(s_{j+1}+\delta)|\vee|\psi(t)-\psi(t_j-2\delta)|\\
	&\qquad\vee\max_{j=1,\dots,m-1}\sup_{t\in[t_j-\delta,s_{j+1}+\delta]}|\psi(t)-\psi(s_{j+1}+\delta)|\\
	&\qquad\vee\sup_{t_m-\delta\leq t\leq T}|\psi(t)-\psi(t_m-\delta)|\\
	&\leq w_T(\psi,\gamma)<\ve,
\end{align*}
which is our desired conclusion.
\end{proof}

\section{Proof of Lemma \ref{lem:projlip}}\label{apdx:projlip}

In this section we prove Lemma \ref{lem:projlip} which states that certain orthogonal projections of solutions of the ESP satisfy Lipschitz continuity properties. In order to prove the lemma, we show that these orthogonal projections of solutions of the ESP satisfy a transformed ESP in which the directions of reflection are orthogonally projected. 

Fix an ESP $\{(d_i,n_i,c_i),i\in\allN\}$. Given a subset $I\subseteq\allN$, let $\Pi_I:\R^J\mapsto\spaan(\{n_i,i\in I\})$ denote the orthogonal projection with respect to the usual Euclidean inner product $\ip{\cdot,\cdot}$. For $f\in\cts$, we define $\Pi_If\in\cts$ by $(\Pi_If)(t)\doteq\Pi_I(f(t))$ for all $t\in[0,\infty)$.

\begin{lem}\label{lem:projesp}
Suppose that $(\z,\y)$ solves the ESP $\{(d_i,n_i,c_i),i\in\allN\}$ for $\x\in\cts$. Given $T\in(0,\infty]$ and $I\subseteq\allN$, suppose that $\allN(\z(t))\subseteq I$ for all $t\in[0,T)$. Then $(\Pi_I\z,\Pi_I\y)$ solves the ESP $\{(\Pi_Id_i,n_i,c_i),i\in I\}$ for $\Pi_I\x$ on $[0,T)$.
\end{lem}

\begin{proof}
By condition 1 of the ESP $\{(d_i,n_i,c_i),i\in\allN\}$ and the linearity of $\Pi_I$, $(\Pi_I\z,\Pi_I\y)$ satisfies condition 1 of the ESP $\{(\Pi_Id_i,n_i,c_i),i\in I\}$ for $\Pi_I\x$. Next, $\Pi_I$ being an orthogonal projection onto $\spaan(\{n_i,i\in I\})$, we have 
	\be\label{eq:ipPiIzni}\ip{\Pi_I\z(t),n_i}=\ip{\z(t),n_i}\geq c_i\qquad\text{for all }t\in[0,T)\text{ and }i\in I,\ee
so $\Pi_I\z$ satisfies condition 2 of the ESP $\{(\Pi_Id_i,n_i,c_i),i\in I\}$. To show $(\Pi_I\z,\Pi_I\y)$ satisfies condition 3 of the ESP $\{(\Pi_Id_i,n_i,c_i),i\in I\}$, fix $0\leq s<t<T$. Since $(\z,\y)$ satisfies the ESP $\{(d_i,n_i,c_i),i\in\allN\}$, \eqref{eq:ytys} implies that there exist $r_i\geq0$, $i\in\cup_{u\in(s,t]}\allN(\z(u))$, such that
	$$\y(t)-\y(s)=\sum_{i\in\cup_{u\in(s,t]}\allN(\z(u))}r_id_i=\sum_{i\in\cup_{u\in(s,t]}\allN(\Pi_I\z(u))}r_id_i,$$
where the second equality uses the equality in \eqref{eq:ipPiIzni} and the fact that $\allN(\z(u))\subseteq I$ for all $u\in(s,t]$. Then, by the linearity of $\Pi_I$,
\begin{align*}
	\Pi_I\y(t)-\Pi_I\y(s)&=\sum_{i\in\cup_{u\in(s,t]}\allN(\Pi_I\z(u))}r_i\Pi_Id_i,
\end{align*}
so $(\Pi_I\z,\Pi_I\y)$ satisfies condition 3 of the ESP $\{(\Pi_Id_i,n_i,c_i),i\in I\}$. Lastly, the fact that $\Pi_I\y(0)\in\conv\lb\lcb d_i,i\in\allN(\Pi_I\z(0))\rcb\rb$ follows from an argument analogous to the one used to prove that condition 3 of the ESP holds, so we omit it.
\end{proof}

In the next lemma, we show that if the ESP $\{(d_i,n_i,c_i),i\in\allN\}$ satisfies Assumption \ref{ass:setB}, then the transformed ESP $\{(\Pi_Id_i,n_i,c_i),i\in I\}$ also satisfies Assumption \ref{ass:setB}.

\begin{lem}\label{lem:projsetB}
Given an ESP $\{(d_i,n_i,c_i),i\in\allN\}$ satisfying Assumption \ref{ass:setB} and $I\subseteq\allN$, there is a compact, convex, symmetric set $B_I$ with $0\in B_I^\circ$ such that for all $i\in I$,
	\be\label{eq:setBproj}\lcb\begin{array}{l}\tilde{z}\in\partial B_I\\|\ip{\tilde{z},n_i}|<1\end{array}\rcb\qquad\Rightarrow\qquad\ip{\tilde{\nu},\Pi_Id_i}=0\qquad\text{for all }\;\tilde{\nu}\in\nu_{B_I}(\tilde{z}).\ee
In other words, the ESP $\{(\Pi_Id_i,n_i,c_i),i\in I\}$ satisfies Assumption \ref{ass:setB}.
\end{lem}

\begin{proof}
Fix $I\subseteq\allN$ and let $V\doteq\spaan(\{n_i,i\in I\})$. By Assumption \ref{ass:setB}, there is a compact, convex, symmetric set $B$ with $0\in B^\circ$ such that \eqref{eq:setB} holds. Define $B_I$ by 
	\be\label{eq:barBx}B_I\doteq\lcb z+y:z\in B\cap V,y\in V^\perp,|y|\leq1\rcb.\ee
Then $B_I$ is a compact, convex, symmetric set with $0\in B_I^\circ$ and
\begin{align}\label{eq:partialBI}
	\partial B_I&=\lcb z+y:z\in\partial B\cap V,y\in V^\perp,|y|\leq1\rcb\\ \notag
	&\qquad\cup\lcb z+y:z\in B^\circ\cap V,y\in V^\perp,|y|=1\rcb.
\end{align}
We now prove that \eqref{eq:setBproj} holds. Suppose $\tilde{z}\in\partial B_I$ satisfies $|\ip{\tilde{z},n_i}|<1$ for some $i\in I$ and let $\tilde{\nu}\in\nu_{B_I}(\tilde{z})$. By \eqref{eq:partialBI}, $\tilde{z}=z+y$ where either 
\begin{itemize}
	\item[(i)] $z\in\partial B\cap V$ and $y\in V^\perp$ satisfies $|y|\leq1$, or 
	\item[(ii)] $z\in B^\circ\cap V$ and $y\in V^\perp$ satisfies $|y|=1$. 
\end{itemize}
In either case, since $y\in V^\perp$, we have $|\ip{z,n_i}|=|\ip{\tilde{z},n_i}|<1$. 

Suppose (i) holds. Given $u\in B$, define $\tilde{u}\doteq\Pi_{I}u+y\in B_I$. Then $\tilde{u}-\tilde{z}=\Pi_Iu-\Pi_Iz$ and hence,
	\be\label{eq:tildenuinwardnormal}\ip{\Pi_{I}\tilde{\nu},u-z}=\ip{\tilde{\nu},\Pi_{I}u-\Pi_{I}z}=\ip{\tilde{\nu},\tilde{u}-\tilde{z}}\geq0,\ee
where we have used that $\Pi_{I}$ is self-adjoint in the first equality and the fact that $\tilde{\nu}\in\nu_{B_I}(\tilde{z})$, $\tilde{z}\in\partial B_I$ and $u\in B_I$ to justify the inequality. Since \eqref{eq:tildenuinwardnormal} holds for all $u\in B$, it follows that $\Pi_I\tilde{\nu}\in\nu_B(z)$. Thus, by (i) and \eqref{eq:setB}, $\ip{\tilde{\nu},\Pi_Id_i}=\ip{\Pi_{I}\tilde{\nu},d_i}=0$. On the other hand, suppose (ii) holds. Then $\ip{\tilde{\nu},\tilde{z}+\tilde{y}-z-y}\geq0$ for all $\tilde{z}\in B\cap V$ and $\tilde{y}\in V^\perp$ satisfying $|\tilde{y}|\leq1$. Letting $\tilde{y}=y$, we see that $\ip{\tilde{\nu},\tilde{z}-z}\geq0$ for all $\tilde{z}\in B\cap V$. Since $z\in B^\circ$, this implies $\tilde{\nu}\in V^\perp$ and so $\ip{\tilde{\nu},\Pi_Id_i}=\ip{\Pi_I\tilde{\nu},d_i}=0$. In either case, $\ip{\tilde{\nu},\Pi_Id_i}=0$. This completes the proof of \eqref{eq:setBproj}.
\end{proof}

\begin{proof}[Proof of Lemma \ref{lem:projlip}]
The lemma follows immediately from Lemma \ref{lem:projesp}, Lemma \ref{lem:projsetB} and Theorem \ref{thm:esmlip}.
\end{proof}

\section{Examples and counter-examples}\label{apdx:examples}

\subsection{An ESP with nonempty $\V$-set}\label{apdx:esp}

Our theory applies to ESMs that are not SMs (see Remark \ref{rmk:sp} to recall the distinction between the two). Here, to illustrate this point, we provide a concrete example of an ESP with a nonempty $\V$-set that satisfies the stated assumptions.

Let $J=3$, $N=4$ and consider the ESP $\{(d_i,n_i,0),i\in\allN\}$ given by $d_i=n_i=e_i$ for $i=1,2,3$, and $d_4=-\sqrt{3}e_3$, $n_4=\frac{1}{\sqrt{3}}(e_1+e_2-e_3)$. It is readily verified that $\V=\{0\}$ and $\W=\emptyset$. Since $\V\neq\emptyset$, a solution to the ESP is not necessarily a solution to the SP. Define the closed, convex, symmetric set $B\subseteq\R^3$ by
	$$B\doteq\lcb z\in\R^3:|\ip{z,e_1}|\leq 1,\;|\ip{z,e_2}|\leq 1,\;|\ip{z,e_3}|\leq 4\rcb.$$
Let $z\in\partial B$. Suppose $|\ip{z,e_i}|<1$ for some $i\in\{1,2,3\}$. Then it is clear that $|\ip{z,e_i}|<1$ implies $\ip{\nu,d_i}=0$ for all $\nu\in\nu(z)$. Suppose $|\ip{z,n_4}|<1$. Then$-\sqrt{3}<\ip{z,e_1}+\ip{z,e_2}-\ip{z,e_3}<\sqrt{3}$. This implies that $|\ip{z,e_3}|<4$, from which it follows that $z$ lies in $D\doteq\partial B\setminus\{x\in B:|\ip{x,e_3}|=4\}$. Since $\{\nu:\nu\in\nu_B(z),z\in D\}\subseteq\spaan(\{e_1,e_2\})$, it follows that for every $\nu\in\nu_B(z)$, $|\ip{\nu,d_4}|=|\ip{\nu,e_3}|=0$. The above assertions, together show that the proposed set $B$ satisfies Assumption \ref{ass:setB}. To verify that the ESP also satisfies Assumption \ref{ass:projection}, consider the map $\pi$ defined by $\pi(x)=x$ if $x\in G$ and for $x\not\in G$, let $a_i=\ip{x,e_i}\vee0$, $i=1,2,3$, and
	$$\pi(x)=a_1e_1+a_2e_2+((a_1+a_2)\wedge a_3)e_3.$$
Then for $x\not\in G$, it is readily verified that $\pi(x)\in\partial G$ and $\pi(x)-x\in d(\pi(x))$. Therefore, it follows from Theorem \ref{thm:main} that given $\x\in\cts_G$, if the solution $(\z,\y)$ of the ESP for $\x$ satisfies the boundary jitter property, then $\nabla_\psi\esm(\x)$ exists and its right continuous regularization can be characterized as the unique solution of the DP associated with $\z$.

\subsection{An SP with nonempty $\W$-set}\label{apdx:nonemptyW}

In this section we present an example of an SP with nonempty $\W$-set, defined as in \eqref{eq:Wset}, and show that the directional derivative of the SM does not necessarily exist at times $t\in[0,\infty)$ such that $\z(t)\in\W$.

Let $J=2$. Consider the SP $\{(d_i,e_i,0),i=1,2\}$ on the nonnegative quadrant with $d_i=(1,1)'$ for $i=1,2$. It is straightforward to verify that $\W=\{0\}$ and the SP satisfies Assumption \ref{ass:setB} with $\delta=1$ and
	$$B\doteq\{x\in\R^2:|\ip{x,e_1}-\ip{x,e_2}|\leq 1,|\ip{x,e_1}+\ip{x,e_2}|\leq 3\},$$ 
and Assumption \ref{ass:projection} with 
	$$\pi(x)\doteq x+(\ip{-x,e_1}\vee\ip{-x,e_2}\vee0)(1,1)'.$$ 
Set $t_0=0$, $\gamma=1/2$ and define $t_n=1-\gamma^n$ so that $t_n\to1$ as $n\to\infty$. Consider the piecewise linear function $\x\in\cts$ defined as follows: $\x(0)=(1,0)'$. For $n\in\N_0$, recursively define
\begin{align*}
	\x(t_{2n+1})&\doteq\x(t_{2n})-3\gamma^{2n+1}e_1,\\
	\x(t_{2n+2})&\doteq\x(t_{2n+1})-3\gamma^{2n+2}e_2.
\end{align*}
Then linearly interpolate as follows: for $t\in[0,1)$,
	$$\x(t)\doteq\x(t_n)+(\x(t_{n+1})-\x(t_n))(t-t_n)\gamma^{-(n+1)},\qquad t\in[t_n,t_{n+1}].$$
Observe that
\begin{align*}
	\x(t_{2n+2})&=\x(t_{2n})-3\gamma^{2n+1}e_1-3\gamma^{2n+2}e_2=(1,0)'+(-3/2,-3/4)'\sum_{j=0}^n\gamma^{2j}.
\end{align*}
Since $\sum_{j=0}^\infty\gamma^{2j}=1/(1-\gamma^2)=1/3\gamma^2$, to ensure $\x$ is continuous, we define $\x(t)\doteq(-1,-1)'$ for $t\in[1,\infty)$. 

Let $\z=\esm(\x)$. Then $\z$ is given as follows: $\z(0)=(1,0)'$ and for $n\in\N_0$, $\z$ is recursively given by
\begin{alignat*}{2}
	\z(t_{2n+1})&=\pi(\z(t_{2n})-3\gamma^{2n+1}e_1)&\;=\;&(0,\gamma^{2n+1})'\\
	\z(t_{2n+2})&=\pi(\z(t_{2n+1})-3\gamma^{2n+2}e_2)&\;=\;&(\gamma^{2n+2},0)'
\end{alignat*}
and for $n\in\N_0$:
	$$\z(t)=\z(t_n)+(\z(t_{n+1})-\z(t_n))(t-t_n)\gamma^{-n-1},\qquad t\in[t_n,t_{n+1}].$$
The continuity of $\z$ implies that $\z(t)=0$ for all $t\in[1,\infty)$.

Let $\psi\in\cts$ be the constant function $\psi\equiv(1,0)'$. For each $\ve>0$, let $\z_\ve\doteq\esm(\x+\ve\psi)$. Consider two sequences $\{\ve_k\}_{k\in\N}$ and $\{\ve_k'\}_{k\in\N}$ given by
\begin{align*}
	\ve_k&\doteq\gamma^{2k},&&k\in\N,\\
	\ve_k'&\doteq\gamma^{2k+1},&&k\in\N.
\end{align*}
We show that $\lim_{k\to\infty}\nabla_\psi^{\ve_k}\esm(\x)(1)\neq\lim_{k\to\infty}\nabla_\psi^{\ve_k'}\esm(\x)(1)$ and thus $\nabla_\psi\esm(\x)(1)$ does not exist.

Fix $k\in\N$. Then $\z_{\ve_k}(0)=(1+\gamma^{2k},0)$ and for $n=0,1,\dots,k-1$, we have
\begin{alignat*}{2}
	\z_{\ve_k}(t_{2n+1})&=\pi(\z_{\ve_k}(t_{2n})-3\gamma^{2n+1}e_1)&\;=\;&(0,\gamma^{2n+1}-\gamma^{2k})',\\
	\z_{\ve_k}(t_{2n+2})&=\pi(\z_{\ve_k}(t_{2n+1})-3\gamma^{2n+2}e_2)&\;=\;&(\gamma^{2n+2}+\gamma^{2k},0)'.
\end{alignat*}
For $n=k$, we have
\begin{alignat*}{2}
	\z_{\ve_k}(t_{2k+1})&=\pi(\z_{\ve_k}(t_{2k})-3\gamma^{2k+1}e_1)&\;=\;&(\gamma^{2k+1},0)',\\
	\z_{\ve_k}(t_{2k+2})&=\pi(\z_{\ve_k}(t_{2k+1})-3\gamma^{2k+2}e_2)&\;=\;&(\gamma^{2k+1}+3\gamma^{2k+2},0)'.
\end{alignat*}
For $n=k+1,k+2,\dots$, we have
\begin{alignat*}{2}
	\z_{\ve_k}(t_{2n+1})&=\pi(\z_{\ve_k}(t_{2n})-3\gamma^{2n+1}e_1)&\;=\;&\lb0,\gamma^{2k+1}+3\sum_{j=k+1}^n\gamma^{2j+1}\rb' ,\\
	\z_{\ve_k}(t_{2n+2})&=\pi(\z_{\ve_k}(t_{2n})-3\gamma^{2n+2}e_2)&\;=\;&\lb0,\gamma^{2k+1}+3\sum_{j=k+1}^n\gamma^{2j+1}+3\gamma^{2n+2}\rb'.
\end{alignat*}
Letting $n\to\infty$, we see that $\z_{\ve_k}(1)=(0,\gamma^{2k})'$. Thus,
	$$\lim_{k\to\infty}\nabla_\psi^{\ve_k}\esm(\x)(1)=\lim_{k\to\infty}\frac{\z_{\ve_k}(1)-\z(1)}{\ve_k}=(0,1)'.$$

We now perform the analogous computations with $\ve_k'$ in place of $\ve_k$. Fix $k\in\N$. Then $\z_{\ve_k'}(0)=(1+\gamma^{2k+1},0)$ and for $n=0,1,\dots,k-1$, we have
\begin{alignat*}{2}
	\z_{\ve_k'}(t_{2n+1})&=\pi(\z_{\ve_k'}(t_{2n})-3\gamma^{2n+1}e_1)&\;=\;&(0,\gamma^{2n+1}-\gamma^{2k+1})',\\
	\z_{\ve_k'}(t_{2n+2})&=\pi(\z_{\ve_k'}(t_{2n+1})-3\gamma^{2n+2}e_2)&\;=\;&(\gamma^{2n+2}+\gamma^{2k+1},0)'.
\end{alignat*}
For $n=k$, we have
\begin{alignat*}{2}
	\z_{\ve_k'}(t_{2k+1})&=\pi(\z_{\ve_k'}(t_{2k})-3\gamma^{2k+1}e_1)&\;=\;&(0,0)',\\
	\z_{\ve_k'}(t_{2k+2})&=\pi(\z_{\ve_k'}(t_{2k+1})-3\gamma^{2k+2}e_2)&\;=\;&(3\gamma^{2k+2},0)'.
\end{alignat*}
For $n=k+1,k+2,\dots$, we have
\begin{alignat*}{2}
	\z_{\ve_k'}(t_{2n+1})&=\pi(\z_{\ve_k'}(t_{2n})-3\gamma^{2n+1}e_1)&\;=\;&\lb3\sum_{j=k+1}^n\gamma^{2j+1}+3\gamma^{2n+2},0\rb', \\
	\z_{\ve_k'}(t_{2n+2})&=\pi(\z_{\ve_k'}(t_{2n})-3\gamma^{2n+2}e_2)&\;=\;&\lb3\sum_{j=k+1}^{n+1}\gamma^{2j+1},0\rb'.
\end{alignat*}
Letting $n\to\infty$, we see that $\z_{\ve_k'}(1)=(\gamma^{2k+1},0)'$. Thus,
	$$\lim_{k\to\infty}\nabla_\psi^{\ve_k'}\esm(\x)(1)=\lim_{k\to\infty}\frac{\z_{\ve_k'}(1)-\z(1)}{\ve_k'}=(1,0)'.$$
	
\bibliographystyle{plain}
\bibliography{DP}

\end{document}